\documentclass[11pt]{amsart}

\usepackage{geometry}
\usepackage{todonotes}
\usepackage{amsfonts, amssymb, amscd, amsmath}
\numberwithin{equation}{section}

\usepackage[symbol]{footmisc}

\usepackage{bm}
\usepackage{verbatim}
\usepackage{mathrsfs}
\usepackage{graphicx}
\usepackage{tikz-cd}
\usepackage{subcaption}
\usepackage{listings}
\usepackage{subfiles}
\usepackage[toc,page]{appendix}
\usepackage{mathtools}
\usepackage{comment}
\usepackage{enumerate}
\usepackage{enumitem}
\usepackage[all]{xy}

\usepackage{graphicx}
\graphicspath{{images/}}

\usepackage{appendix}
\usepackage{hyperref}
\hypersetup{
    colorlinks=true,
    citecolor=red,
    linkcolor=blue,
    filecolor=magenta,      
    urlcolor=red,
}
\newcommand{\Mm}{{\bf{M}}}

\newcommand{\bM}{{\bf M}}

\newcommand{\Pp}{\mathbb{P}}
\newcommand{\Qq}{\mathbb{Q}}
\newcommand{\QQ}{\mathbb{Q}}
\newcommand{\Rr}{\mathbb{R}}
\newcommand{\RR}{\mathbb{R}}
\newcommand{\Zz}{\mathbb{Z}}
\newcommand{\ZZ}{\mathbb{Z}}

\newcommand{\Ivol}{\operatorname{Ivol}}

\newcommand{\vol}{\operatorname{vol}}
\newcommand{\Center}{\operatorname{center}}

\newcommand{\Exc}{\operatorname{Exc}}

\newcommand{\mld}{{\rm{mld}}}

\newcommand{\lc}{\mathrm{lc}}

\newcommand{\lct}{\operatorname{lct}}

\newcommand{\Supp}{\operatorname{Supp}}

\newcommand{\codim}{\operatorname{codim}}
\newcommand{\mult}{\operatorname{mult}}

\newcommand{\lf}{\lfloor}
\newcommand{\rf}{\rfloor}

\newcommand{\Oo}{\mathcal{O}}

\newcommand{\Ii}{\Gamma}

\newtheorem{thm}{Theorem}[section]
\newtheorem{conj}[thm]{Conjecture}
\newtheorem{cor}[thm]{Corollary}
\newtheorem{lem}[thm]{Lemma}
\newtheorem{prop}[thm]{Proposition}

\newtheorem{claim}[thm]{Claim}

\theoremstyle{definition}
\newtheorem{defn}[thm]{Definition}

\theoremstyle{definition}
\newtheorem{rem}[thm]{Remark}
\newtheorem{rmk}[thm]{Remark}

\newtheorem{ex}[thm]{Example}
\newtheorem{nota}[thm]{Notation}

\theoremstyle{definition}

\newcommand{\gE}{\mathfrak{E}}

\newcommand{\gP}{\mathfrak{P}}

\newcommand{\FF}{\mathbb{F}}

\newcommand{\PP}{\mathbb{P}}

\newcommand{\sC}{\mathcal{C}}
\newcommand{\sE}{\mathcal{E}}

\newcommand{\sO}{\mathcal{O}}

\newcommand{\tC}{\widetilde{C}}

\newcommand{\tF}{\widetilde{F}}

\newcommand{\tX}{\widetilde{X}}

\newcommand{\tZ}{\widetilde{Z}}

\newcommand{\tGamma}{\widetilde{\Gamma}}

\newcommand{\rb}{\mathrm{b}}

\newcommand{\bfB}{\mathbf{B}}

\newcommand{\I}{\mathrm{I}}
\newcommand{\II}{\mathrm{II}}
\newcommand{\III}{\mathrm{III}}
\newcommand{\IV}{\mathrm{IV}}
\newcommand{\klt}{\mathrm{klt}}
\newcommand{\pr}{\mathrm{pr}}
\newcommand{\red}{\mathrm{red}}
\newcommand{\sm}{\mathrm{sm}}

\begin{document}

\title{On the Iitaka volumes of log canonical surfaces and threefolds}

\author{Guodu Chen, Jingjun Han, and Wenfei Liu}

\address{School of Mathematical Sciences, Shanghai Jiao Tong University, Shanghai, 200240, China}
\email{chenguodu@sjtu.edu.cn}

\address{Shanghai Center for Mathematical Sciences, Fudan University, Shanghai, 200438, China}
\email{hanjingjun@fudan.edu.cn}

\address{School of Mathematical Sciences, Xiamen University, Siming South Road 422, Xiamen, Fujian 361005, China}
\email{wliu@xmu.edu.cn}

%\thanks{\emph{Keywords}: Iitaka volume, surfaces, good minimal model.}
\date{\today}

\begin{abstract}
Given positive integers $d\geq\kappa$, and a subset $\Gamma\subset [0,1]$, let $\Ivol_\lc^\Gamma(d,\kappa)$ denote the set of Iitaka volumes of $d$-dimensional projective log canonical pairs $(X, B)$ such that the Iitaka--Kodaira dimension $\kappa(K_X+B)=\kappa$ and the coefficients of $B$ come from $\Gamma$. In this paper, we show that, if $\Gamma$ satisfies the descending chain condition, then so does $\Ivol_\lc^\Gamma(d,\kappa)$ for $d\leq 3$. In case $d\leq 3$ and $\kappa=1$, $\Gamma$ and $\Ivol_\lc^\Gamma(d,\kappa)$ are shown to share more topological properties, such as closedness in $\RR$ and local finiteness of accumulation complexity. In higher dimensions,  we show that the set of Iitaka volumes for $d$-dimensional klt pairs with Iitaka dimension $\geq d-2$ satisfies the DCC, partially confirming a conjecture of Zhan Li.

We give a more detailed description of the sets of Iitaka volumes for the following classes of projective log canonical surfaces: (1) smooth properly elliptic surfaces, (2) projective log canonical surfaces with coefficients from $\{0\}$ or $\{0,1\}$. In particular, the minima as well as the minimal accumulation points are found in these cases.
\end{abstract}

\maketitle

\pagestyle{myheadings}\markboth{\hfill G. Chen, J. Han, and W. Liu \hfill}{\hfill On the Iitaka volumes of log canonical surfaces and higher dimensional varieties\hfill}

\tableofcontents

\section{Introduction}
We work over the field of complex numbers $\mathbb C$.

The log canonical divisor $K_X+B$ plays a pivotal role in birational geometry and in the classification theory of projective pairs $(X, B)$, say with at most log canonical singularities. One can extract two basic invariants of the pair from the asymptotic behaviour of the pluricanonical systems $|m(K_X+B)|$, as $m\rightarrow \infty$. Namely, the \emph{Iitaka--Kodaira dimension}
\[
\kappa(K_X+B)=
\begin{cases}
\limsup\limits_{m\rightarrow \infty}  \frac{\log \dim H^0(X, \sO_X(m(K_X+B)))}{\log m} & \text{ if $|m(K_X+B)|\neq\emptyset$ for some $m\in \Zz_{>0}$,} \\
-\infty & \text{otherwise,}
\end{cases}
\]
and, in case $\kappa:=\kappa(K_X+B)\geq 0$, the \emph{Iitaka volume}
\[
\vol_\kappa(K_X+B):=\limsup_{m \rightarrow \infty} \frac{ h^0\left(X, \mathcal{O}_X(m (K_X+B))\right)}{m^{\kappa}/\kappa!}.
\]
In this paper, we are interested in the distribution of Iitaka volumes. 

More specifically, for positive integers $d\geq \kappa$, and a subset $\Gamma\subset [0,1]$, we introduce the following set of pairs
\[
\gP_\lc^\Gamma(d,\kappa) =  \{(X,B) \mid (X, B) \text{ is projective lc, } \dim X=d,\, \kappa(K_X+B)=\kappa,\, B\in \Gamma\}
\]
where the condition $B\in \Gamma$ means that the coefficients of $B$ lie in $\Gamma$. The corresponding set of Iitaka volumes is then $$\Ivol_\lc^\Gamma(d,\kappa)=\left\{\vol_\kappa\left(K_X+B\right) \mid (X, B)  \in \gP_\lc^\Gamma(d,\kappa) \right\}.$$
Similarly, one may define $\gP_\klt^\Gamma(d,\kappa)$, $\gP_\sm(d,\kappa)$ with restrictions on the allowed singularities, and then the corresponding sets of Iitaka volumes $\Ivol_\klt^\Gamma(d,\kappa)$, $\Ivol_\sm(d,\kappa)$; see Notation~\ref{nota: Ivol}.

In case $\kappa=d$, $\vol_d (K_X+B)$ is the usual volume and we may just denote it by $\vol(K_X+B)$. It is a very deep fact that, if $\Gamma$ satisfies the descending chain condition (DCC), then $\Ivol_\lc^\Gamma(d,d)$ also satisfies the DCC (\cite{Ale94, HMX14}). This is essential in establishing various boundedness results for projective log canonical pairs and, in turn, the projectivity of the moduli spaces of stable pairs with a given volume (\cite{HMX18, Kol23}).

For $0<\kappa<d$, one expects similar phenomena occur. The following conjecture is just an extension of \cite[Conjecture 1.6]{Li20} from klt pairs to lc pairs.
\begin{conj}[DCC of Iitaka volumes]\label{conj:ivoldcc}
Let $\kappa <d$ be positive integers and $\Gamma \subset[0,1]$ a DCC set. Then the set of Iitaka volumes $\Ivol_\lc^\Gamma(d,\kappa)$ is a DCC set.
\end{conj}
Conjecture \ref{conj:ivoldcc} is known to hold for log canonical surfaces, as well as log canonical threefolds with Iitaka--Kodaira dimension $2$ (\cite[Theorem~8.1]{PS09}). Zhan Li showed that $\Ivol_\klt^\Gamma(3,\kappa)$ is a DCC set if $\Gamma$ is a DCC subset of $[0,1]\cap \QQ$ (\cite[Corollary~4.4]{Li20}). Birkar showed the DCC for the set of Iitaka volumes of lc-trivial fibrations with an ample $\ZZ$-divisor of fixed volume on the general fibers (\cite[Theorem~1.7]{Bir21}). Based on the ideas of previous authors, we prove that Conjecture \ref{conj:ivoldcc} follows from two other important conjectures in birational geometry (see Lemma \ref{lem:ivolaccfollowfrom2conj}).  In particular, we verify it in low dimensions.
\begin{thm}[= Corollary~\ref{cor: conj holds d=3}]\label{thm: main1}
Conjecture \ref{conj:ivoldcc} holds when $d\le3$.
\end{thm}

%One goal of this paper is to confirm Conjecture \ref{conj:ivoldcc} for threefolds in full generality.\begin{thm}\label{thm: ivoldim3}Conjecture \ref{conj:ivoldcc} holds when $d\le3$. \end{thm}
In higher dimensions, we prove the DCC of a subset of $\Ivol_\lc^\Gamma(d,\kappa)$.
\begin{thm}\label{thm:main2}
Let $\kappa\leq d$ be positive integers with $\kappa\geq d-2$, and $\Gamma \subset[0,1]$ a DCC set. Then the following subset of $\Ivol_\lc^\Gamma(d,\kappa)$: 
\[
\left\{
 \vol_\kappa\left(K_X+B\right)  \mid
(X,B) \in \gP_\lc^\Gamma(d,\kappa), 
\text {either $(X,B)$ is klt or $K_X+B$ is semi-ample}
\right\}
\]
is a DCC set.
\end{thm}

In fact, more topological properties are shared by the coefficient set $\Gamma$ and the set of Iitaka volumes $\Ivol_\lc^\Gamma(d,\kappa)$. Filipazzi \cite[Theorem~1.2]{Fil20} showed that the usual volume set $\Ivol_\lc^\Gamma(d,d)$ is closed if $\Gamma\subset[0,1]$ is a closed DCC subset and if $1\in \Gamma$. His result can be generalized to the setting of Iitaka volumes, at least in low dimensions. 
\begin{thm}[= Corollary~\ref{cor:fac} + Theorem~\ref{thm: closed}]\label{thm: main3}
Let $\Gamma\subset [0,1]$ be a DCC subset, $M$ a positive real number, and $d\in \{2,3\}$. 
\begin{enumerate}
    \item If $\Gamma$ has finite accumulation complexity, then so does $\Ivol_\lc^\Gamma(d,1)_{\leq M}$.
    \item If $\Gamma$ is closed in $\RR$, then $\Ivol_\lc^\Gamma(d,1)$ is also closed. 
\end{enumerate}
\end{thm}
Here, for a subset $A\subset \RR$, if there is a minimal $n\in\ZZ_{\geq 0}$ such that its $(n+1)$-th derived set $A^{(n+1)}$ is empty, then we say \emph{$A$ has accumulation complexity $n$}; otherwise, we say \emph{$A$ has infinite accumulation complexity}.

In the proofs of Theorems~\ref{thm: main1} and \ref{thm:main2} in Section~\ref{sec: DCC}, one needs the existence of a good minimal model $(X, B)\dashrightarrow (Y, B_Y)$,  which are fully available in dimensions $\leq 3$ and still conditional in higher dimensions, in order to set up a Iitaka fibration $f\colon Y\rightarrow Z$ of the pair $(Y, B_Y)$. Via the canonical bundle formula $K_Y+B_Y\sim_{\RR, Z} f^*(K_Z+B_Z+\bM_Z)$, we have 
\[
\vol_k(K_X+B) = \vol_k(K_Y+B_Y) = \vol(K_Z+B_Z+\bM_Z).
\]
In order to obtain control over $\vol(K_Z+B_Z+\bM_Z)$, the key point is to bound uniformly the $\rb$-Cartier index of the moduli part $\bM$ (see Lemma~\ref{lem: cbfindex}). Currently, this is contingent on the Existence of Complements Conjecture (Conjecture \ref{conj: boundedness and existence of n complement nft}), which is known to hold only in dimension $\leq 3$ (\cite[Theorem 2.14]{CHL24}). For the proof of Theorem~\ref{thm: main3}, we need to run a further MMP to make the vertical part $B^v$ of the boundary $\RR$-divisor $B$ relatively trivial over $Z$ (Lemma~\ref{lem: model nef Bv}), so that we have a better control of the accumulation behaviour of the coefficients of the discriminant part $B_Z$ (Lemma~\ref{lem: cbf (d,1)}).

Invariant Iitaka volumes (Definition \ref{defn:inviiitdim}) behave differently compared to Iitaka volumes (see Example \ref{ex: why invairant iitaka dimension}). Despite these differences, it is still conjectured that the DCC property holds for invariant Iitaka volumes (Conjecture \ref{conj:invivoldcc}) which can be confirmed in low dimensions (Theorem \ref{thm:invivoldcc3}). A brief discussion on this topic will be provided in subsection \ref{sec:invivol}.

Most of the remaining part of the paper, Section 4, is devoted to a more detailed description of the sets of Iitaka volumes of several classes of log canonical surfaces. For smooth elliptic surfaces with Kodaira dimension 1, the set $\Ivol_\sm(2,1)$ of their Iitaka volumes can be completely determined as follows.
\begin{thm}[= Theorem~\ref{thm: sm surf Ivol}]\label{thm: main4}
Using Notation~\ref{nota: Ivol}, we have 
\[
\Ivol_\sm(2,1)= \Ivol(\sE_{1,0})\cup \Ivol(\sE_{0,0}),
\]
where $\Ivol(\sE_{1,0})$ and $ \Ivol(\sE_{0,0})$ are specified as follows:
\[
\Ivol(\sE_{1,0}) = \left\{-1+\sum_{1\leq i\leq r}\left(1-\frac{1}{m_i}\right) \,\bigg|\, r\in \Zz_{>0},\, m_i\in \ZZ_{\geq 2}\right\}_{>0}
\]
and 
\[
\Ivol(\sE_{0,0})=\left\{-2 +\sum_{1\leq i\leq r} \left(1- \frac{1}{|\sigma_i|}\right) \,\,\Bigg|\,\,
\begin{aligned}
& \sigma_i\text{ are torsion elements of}\\
& \text{an elliptic curve such that $\sum_i\sigma_i=0$}
\end{aligned}
\right\}_{>0}.
\]
\end{thm}
In the proof of Theorem~\ref{thm: main4}, one direction of the inclusions is essentially by the canonical bundle formula of Kodaira for elliptic fibrations. For the other direction of inclusions, we need to construct elliptic fibrations with prescribed invariants, using techniques such as base change, logarithmic transformations, and $*$-transfers of elliptic fibers.

As a corollary of Theorem~\ref{thm: main4}, we can draw the following conclusions about the geometry of the set $\Ivol_\sm(2,1)$.
\begin{cor}[= Theorem~\ref{thm: min Ivol sm} + Corollary~\ref{cor: sm acc points}] The following holds for $\Ivol_\sm(2,1)$.
    \begin{enumerate}
    \item $\min\Ivol_\sm(2,1) =\frac{1}{6}$.
      \item The minimal accumulation point of $\Ivol_\sm(2,1)$ is $\frac{1}{2}$.
      \item For a given $M\geq\frac{1}{2}$, the accumulation complexity of $\Ivol_\sm(2,1)_{\leq M}$ is $\lfloor M \rfloor + 1$.
      \item The accumulation complexity of $\Ivol_\sm(2,1)$ is $\infty$.
      \item $\Ivol_\sm(2,1)$ is closed in $\RR$.
  \end{enumerate} 
\end{cor}
Parallel results are proved if one fixes the geometric genus; see Theorem~\ref{thm: Ivol sm prescribe pg}, Corollaries~\ref{cor: min Ivol sm prescribe pg} and \ref{cor: Ivol ccc sm prescribe pg}.

For singular surfaces, we give a recipe in Theorem~\ref{thm: lc surf Ivol}, reducing the computation of the Iitaka volume of a log canonical surface $(X, B)$ to that of a crepant smooth model $(S, B_S)$. We can then proceed according to the geometry of $(S, B_S)$, which is more transparent than that of $(X, B)$. However, the coefficients of $B_S$ do not necessarily come from the set $\Gamma$ we start with. In the cases $\Gamma=\{0\}$ or $\{0,1\}$, it is possible to describe the possible coefficients of $B_S$ in a satisfactory way, and in turn one may find the minima of $\Ivol_\lc^\Gamma(2,1)$ and $\Ivol_\lc^\Gamma(2,1)'$ in these cases.
\begin{thm}[= Theorem~\ref{thm: min lc Ivol 0} + Theorem~\ref{thm: min lc Ivol 1}]
We have
\[
\min\Ivol_\lc^{\{0\}}(2,1) =\min \Ivol_\lc^{\{0,1\}}(2,1) = \frac{1}{671},\quad \min\Ivol_\lc^{\{0\}}(2,1)' =\min \Ivol_\lc^{\{0,1\}}(2,1)' = \frac{1}{66}.
\]
\end{thm}
For comparison, note that in the case $(d,\kappa)=(2,2)$, the conjectural minima of $\Ivol_\lc^{\{0\}}(2,2)$ and $\Ivol_\lc^{\{0,1\}}(2,2)$ are the same $\frac{1}{48983}$, which is realized by an example constructed in \cite{AL19a}. Recently, the minimal accumulation points of $\Ivol_\lc^{\{0\}}(2,2)$ and $\Ivol_\lc^{\{0,1\}}(2,2)$ were found in \cite{LL23}:
\[
\min\Ivol_\lc^{\{0\}}(2,2)' =\min \Ivol_\lc^{\{0,1\}}(2,2)' = \frac{1}{825}.
\]

\medskip

In the end of the paper, Section~\ref{sec: anticanonical}, we investigate the Iitaka volumes of anti-log canonical divisors $-(K_X+B)$, which have completely different behaviour. Still, it is possible to find some pattern under some assumptions on singularities and on the existence of certain horizontal $\ZZ$-divisors.
\begin{thm}\label{thm: anti Iitaka ACC}
Let $d$ be a positive integer, $\epsilon,u$ two positive rational numbers, and $\Gamma\subset[0,1]$ a DCC set. Then there exists an ACC set $\mathcal{A}$ depending only on $d,\epsilon,u,$ and $\Gamma$ satisfying the following.

Assume that $(X,B)$ is a projective pair such that
\begin{enumerate}
    \item $(X,B)$ is $\epsilon$-lc of dimension $d$ with $B\in\Gamma$,
    \item $-(K_X+B)$ is semi-ample defining a contraction $f\colon X\to Z$, and
    \item there is a $\Qq$-Cartier $\ZZ$-divisor $N$ on $X$ with $\vol(N|_F)=u$ for the general fibers $F$ of $f$.
\end{enumerate}
Then the Iitaka volume of $-(K_X+B)$ belongs to $\mathcal{A}$.
\end{thm}

\medskip

\noindent{\bf Notation and Conventions.} Let $A\subset\RR$ be a subset of real numbers.
\begin{itemize}[leftmargin=*]
\item For $b\in \RR$,  we denote $A_{>b}:=\{a\in A\mid a>b\}.$ The subsets $A_{\geq b}$, $A_{<b}$, $A_{\leq b}$ of $A$ are similarly defined.
\item We use $A^{(1)}$ or $A'$ to denote the derived set of  $A$, that is, the set of accumulation points of $A$. Inductively, for $n\geq 2$, $A^{(n)}:=\left(A^{(n-1)}\right)'$ is the $n$-th derived set of $A$. 
\item For a positive integer $n$, we denote
\[\sum_{\leq n} A :=\left\{a=\sum_{i=1}^r a_i\,\Bigg|\, r\leq n, a_i\in A\right\} \text{ and } \sum A = \bigcup_{n\geq 1} \left(\sum_{\leq n} A\right).
\]
\item The set $A$ is said to satisfy the \emph{descending chain condition} (DCC for short) if every infinite descreasing sequence in $A$ stabilizes. In this case, we say that $A$ is a DCC set.

Similarly, the set $A$ is said to satisfy the \emph{ascending chain condition} (ACC for short) if every infinite increasing sequence in $A$ stabilizes. In this case, we say that $A$ is an ACC set.
\end{itemize}

\bigskip
\noindent\textbf{Acknowledgement}.
This work originates from the 2023 Workshop on Explicit Birational Geometry in Fudan University.
The authors would like to thank Junpeng Jiao, Jihao Liu, and Xin Lu for valuable discussions and suggestions. The second named author was supported by National Key Research and Development Program of China (Grant No. 2020YFA0713200), and he is a member of LMNS, Fudan University.

%%%%%%%%%%%%%%%%%%%%%%%%%%%%%%%
%%%%%%%%%%%%%%%%%%%%%%%%%%%%%%%Section line
 
\section{Preliminaries}\label{sec2}
We adopt the standard notation and definitions (e.g.~pairs and singularities) in \cite{KM98,BCHM10} and will freely use them. Note that, for a pair $(X,B)$, we require the boundary $B$ to be effective in this paper. If $X$ is equipped with a projective morphism $f\colon X\rightarrow Z$, then we say the pair $(X, B)$ is over $Z$, and write it as $(X/Z, B)$ if $f$ is clear from the context or not important. Furthermore, if $z\in Z$ is a scheme theoretic point, then we write $(X/Z\ni z, B)$ to indicate that we are mainly concerned with the germ of $(X, B)$ over a neighborhood of $z$ in $Z$.

The following are the sets of pairs and of their Iitaka volumes considered in this paper.
\begin{nota}\label{nota: Ivol}
For positive integers $d\geq \kappa$, and a DCC set $\Gamma\subset [0,1]\cap\Qq$, we introduce the following sets of pairs:
\begin{itemize}
    \item $\gP_\lc^\Gamma(d,\kappa) =  \{(X,B) \mid (X, B) \text{ is lc projective, } \dim X=d,\,  \kappa(K_X+B)=\kappa,\, B\in \Gamma\}$.
   \item $\gP_\klt^\Gamma(d,\kappa) =  \{(X,B) \mid (X, B) \text{ is klt projective, } \dim X=d,\,  \kappa(K_X+B)=\kappa,\,  B\in \Gamma\}$.
    \item $\gP_\sm(d,\kappa) =  \{X\mid X \text{ is smooth projective, } \dim X=d,\,  \kappa(K_X)=\kappa\}$.
\end{itemize}
The corresponding sets of Iitaka volumes are defined as follows:
\begin{itemize}
  \item $\Ivol_\lc^\Gamma(d,\kappa) =\big\{\vol_\kappa(K_X+B) \mid (X, B)\in \gP_\lc^\Gamma(d,\kappa) \big\}$.
  \item $\Ivol_\klt^\Gamma(d,\kappa)=\big\{\vol_\kappa(K_X+B) \mid (X, B)\in \gP_\klt^\Gamma(d,\kappa) \big\}$.
  \item $\Ivol_\sm(d,\kappa)=\big\{\vol_\kappa(K_X) \mid X\in \gP_\sm(d,\kappa) \big\}$.
\end{itemize}
When $\Gamma=\{0\}$, we often omit it from above notation.
\end{nota}

\subsection{Divisors}\label{div}
Let $\FF$ be one of the rings $\ZZ,\,\QQ,\,\RR$, and $X$ a normal variety. An $\FF$-divisor $B$ on $X$ is simply a finite formal sum $\sum_{i=1}^s b_i B_i$, where $b_i\in \FF$, and $B_i$ are prime divisors on $X$. An $\FF$-divisor is $\FF$-Cartier if it is an $\FF$-linear combination of Cartier divisors. For a set $\Gamma\subset \RR$,  we write $B\in\Ii$ if $b_i\in\Ii$ for every $i$. We usually omit $\FF$ from the notation if $\FF=\ZZ$. 

Let $B:=\sum_{i=1}^s b_iB_i$ and $B':=\sum_{i=1}^s b_i'B_i$ be $\FF$-divisors on a normal variety $X$, where $B_i$ are distinct prime divisors. We define 
$$\lf B\rf:=\sum_{i=1}^s \lf b_i\rf B_i, \quad \{ B\}:=\sum_{i=1}^s \{ b_i\} B_i,  \quad B\wedge B':=\sum_{i=1}^s \min\{ b_i,b_i'\} B_i.$$ 
The two $\FF$-divisors $B$ and $B'$ are \emph{$\FF$-linearly equivalent}, denoted $B\sim_\FF B'$, if 
\[
B-B' =\sum_{1\leq i\leq k} r_i\cdot (\varphi_i)_X
\]
where $r_i\in \FF$, $\varphi_i$ are nonzero rational functions on $X$, and $(\varphi_i)_X$ denotes their divisors. 

Let $f\colon X\rightarrow Z$ be a projective morphism from $X$ to another normal variety $Z$. The horizontal$/Z$ part and the vertical$/Z$ part of $B$ with respect to $f$ are respectively 
\[
B^h = \sum_{f(B_i)=Z} b_i B_i, \quad B^v= \sum_{f(B_i)\neq Z} b_i B_i.
\]
We say $B$ and $B'$ are \emph{relatively $\FF$-linear equivalent} over $Z$, denoted $D\sim_{\FF, Z} D'$, if there is an $\FF$-Cartier $\FF$-divisor $H$ on $Z$ such that $D\sim_{\FF} D'+f^*H$. For a point $z\in Z$, we define $$B_z:=\sum_{i\mid f(B_i)=\bar z}b_iB_i.$$

A \emph{$\rb$-$\FF$-divisor} $\bfB$ on a normal variety $X$ is a collection of $\FF$-divisors $\bfB_Y$ on the birational models $Y$ of $X$ that are compatible under pushforward; $\bfB_Y$ is then called the trace of $\bfB$ on $Y$ (see \cite[Definition 2.1]{HL21b}). The following constructions of $\rb$-$\FF$-divisors are used in this paper:
\begin{itemize}[leftmargin=*]
    \item For an $\FF$-Cartier $\FF$-divisor $B$ on $X$, we may define a $\rb$-$\FF$-divisor $\overline{B}$ so that its trace on a higher birational model $Y$ of $X$ is the pullback of $B$. Such a $\rb$-$\FF$-divisor is called $\rb$-$\FF$-Cartier. A $\rb$-$\FF$-divisor $\bfB$ on $X$ is said to \emph{descend to $X$} if $\bfB_X$ is $\FF$-Cartier and $\bfB=\overline{\bfB_X}$.
    \item For each normal birational model $Y$ of $X$, we may and will choose a canonical Weil-divisor $K_Y$ in such a way, that they form a $\rb$-divisor as $Y$ varies. Two pairs $(Y, B_Y)$ and $(X, B_X)$ are \emph{crepant} to each other if there is an equality of $\rb$-$\RR$-Cartier divisors $\overline{K_Y+B_Y} = \overline{K_X+B_X}$.
    \item Let $v$ be a divisorial valuation of the function field $K(X)$ such that $\Center_X(X)\neq \emptyset$. Then we may define a $\rb$-divisor whose trace on a higher birational model $Y$ of $X$ is
    $$ E_Y(v):=\begin{cases}
        \Center_Y(v) & \text{ if $\Center_Y(v)$ is a divisor,}\\
        0& \text{ if $\Center_Y(v)$ is not a divisor.}
    \end{cases} $$
\end{itemize}
The notion of $\rb$-$\FF$-divisors is also a convenient gadget in the definition of generalized pairs and in the formulation of canonical bundle formulas, to be discussed in Sections~\ref{sec: g-pair} and \ref{sec: cbf}.

\subsection{Complements}
We recall the definition of complements and the Existence of Complements Conjecture.
\begin{defn}\label{defn complement}
Let $n$ be a positive integer, $\Ii\subset (0,1]$ a set, and $(X/Z\ni z,B)$ and $(X/Z\ni z,B^+)$ two pairs. We say that $(X/Z\ni z,B^+)$ is an \emph{$\Rr$-complement} of $(X/Z\ni z,B)$ if $(X,B^+)$ is lc, $B^+\geq B$, and $K_X+B^+\sim_{\Rr}0$ over a neighborhood of $z$. We say that $(X/Z\ni z,B)$ is \emph{$\Rr$-complementary} if $(X/Z\ni z,B)$ has an $\Rr$-complement. 

We say that $(X/Z\ni z,B^+)$ is an \emph{$n$-complement} of $(X/Z\ni z,B)$ if
\begin{itemize}
\item $(X/Z\ni z,B^+)$ is lc,
\item $nB^+\geq \lfloor (n+1)\{B\}\rfloor+n\lfloor B\rfloor$, and
\item $n(K_X+B^+)\sim 0$ over a neighborhood of $z$.
\end{itemize}
\end{defn}

\begin{conj}[Existence of Complements Conjecture]\label{conj: boundedness and existence of n complement nft}
Let $d$ be a positive integer and $\Ii\subset [0,1]$ a DCC set. Then there exists a positive integer $n$ depending only on $d$ and $\Ii$ satisfying the following. 

Assume that $(X/Z\ni z,B)$ is an $\Rr$-complementary pair of dimension $d$ such that $B\in\Ii$, then $(X/Z\ni z,B)$ has an $n$-complement $(X/Z\ni z,B^+)$. Moreover, if the closure of $\Ii$ belongs to $[0,1]\cap\mathbb Q$, then we can pick $B^+\ge B$.
\end{conj}

\subsection{Generalized pairs and their singularities}\label{sec: g-pair}
We briefly discuss generalized pairs and their singularities, and refer the reader to \cite{BZ16,HL22a,HL23a} for further details. 

\begin{defn}
A projective \emph{generalized pair} $(X,B+\bM)$ consists of a normal projective variety $X$, an effective $\Rr$-divisor $B$ on $X$, and a nef b-$\Rr$-divisor $\bM$ on $X$, such that $K_X+B+\bM_X$ is $\Rr$-Cartier. 

Let $(X,B+\bM) $ be a generalized pair and $\phi\colon W \to X$ a log resolution of $(X,\Supp B)$ such that $\bM$ descends to $W$. We may write 
$$K_{W}+B_W+\bM_W=\phi^*(K_X+B+\bM_X)$$
for some $\Rr$-divisor $ B_W$ on $W$. Let $E$ be a prime divisor on $W$. The \emph{log discrepancy} of $E$ with respect to $(X,B+\bM)$ is defined as
$$a(E,X,B+\bM):=1-\mult_EB_W.$$
We say $(X,B+\bM)$ is \emph{lc} if $a(E,X,B+\bM)\ge0$ for any prime divisor $E$ over $X$.
\end{defn}
The importance of generalized pairs is reflected in their appearance in the canonical bundle formula, to be treated in the next subsection.

\subsection{Canonical bundle formulas}\label{sec: cbf}
We recall the canonical bundle formula for log canonical pairs, and refer the reader to \cite[\S3.4]{Bir19} for its basic properties.

Recall that a projective morphism $f\colon  X \to Z$ between normal varieties is a \emph{contraction} or \emph{fibration} if $f_*\Oo_X = \Oo_Z$; the word \emph{fibration} is used more often  if $\dim X>\dim Z$. A pair $(X, B)$ is called \emph{lc-trivial} with respect to $f$ if $K_X+B\sim_{\RR, Z} 0$. 
\begin{thm}\label{thm: cbf}
Let $(X,B)$ be an lc pair and $f\colon X\to Z$ a contraction between normal quasi-projective varieties such that $K_X+B\sim_{\Rr,Z}0$. Then we can find an $\Rr$-divisor $B_Z\ge0$ and a nef over $Z$ b-$\Rr$-divisor $\bM$ on $Z$, such that $(Z,B_Z+\bM)$ is an lc generalized pair, and
$$K_X+B\sim_\Rr f^*\left(K_Z+B_Z+\bM_{Z}\right).$$
Here $B_Z$ (resp. $\bM$) is called the \emph{discriminant part} (resp. a \emph{moduli part}) of a canonical bundle formula for $(X,B)$ over $Z$ which is uniquely determined (resp. determined up to $\Rr$-linear equivalence).
\end{thm}
Here we emphasize that $\bM$ is only determined up to $\Rr$-linear equivalence, so there are many choices of $\bM$, some of which could behave badly. For many purposes, we need to choose $\bM$ with controlled b-Cartier index, as in the following statement.
\begin{lem}\label{lem: cbfindex}
Let $\Gamma\subset[0,1]$ be a DCC set. Then there exist a finite subset $\Gamma_0\subset \Gamma\cap\QQ$, a DCC subset $\tGamma\subset[0,1]$, and a positive integer $p$, depending only on $\Gamma$, such that the following holds.

Assume that $(X,B)$ is an lc pair and $f\colon X\to Z$ is a contraction such that $\dim X-\dim Z\le2$, $B\in\Gamma,$ $K_X+B\sim_{\Rr,Z}0,$ and $\kappa(K_X+B)\ge0$. Then the following holds.
\begin{enumerate}
\item  We can choose a moduli part $\bM$ of the canonical bundle formula for $(X,B)$ over $Z$ such that $p\Mm$ is b-Cartier, $B_Z\in\tGamma$, and 
$$p(K_X+B)\sim pf^*(K_Z+B_Z+\bM_{Z}),$$
where $B_Z$ is the discriminant part of the canonical bundle formula for $(X,B)$ over $Z$.
\item If $B$ is horizontal$/Z$, then $B\in \Gamma_0$. 
\end{enumerate}
Moreover, if $\Gamma$ is a hyperstandard set (see \cite[2.2]{Bir19}), then so is $\tGamma$.
\end{lem}

\begin{proof}
By \cite[Theorem 2.14]{CHL24}, we may find a positive integer $p$ depending only on $\Gamma$ such that for any pair $(W/T\ni t,\Delta)$ of dimension $\le3$, if $(W/T\ni t,\Delta)$ has an $\Rr$-complement, then $(W/Z\ni z,\Delta)$ has a $p$-complement. We will show that $p$ has the required properties. 

Let $F$ be a general fiber of $f$, and set $K_F+B_F:=(K_X+B)|_F$. According to the proof of \cite[Lemma~5.1]{CHL24} and the global ACC for numerically trivial pairs (\cite[Theorem~D]{HMX14}), $B_F\in\Gamma_0$ for some finite set $\Gamma_0\subset\Gamma\cap\Qq$ that only depends on $\Gamma$. Then as $\dim X-\dim Z\le2$, we see that $p(K_F+B_F)\sim0$. Hence $p(K_X+B)\sim0$ over the generic point $\eta_Z$ of $Z$. There exists a nonzero $\varphi\in K(X)$ such that $pL:=p(K_X+B)+(\varphi)_X$ is zero near $\eta_Z$. In particular, $L$ is vertical over $Z$ and
$$pL\sim p(K_X+B)\sim_{\Rr,Z}0.$$ 
By \cite[Lemma 2.11]{Li20}, $L=f^*L_Z$ for some $\Rr$-Cartier $\Rr$-divisor $L_Z$ on $Z$. Let $B_Z$ be the discriminant part of the canonical bundle formula of $(X,B)$ over $Z$, and $\bM_{Z}:=L_Z-K_Z-B_Z.$ Then
$$p(K_X+B)\sim pL=pf^*L_Z=pf^*(K_Z+B_Z+\bM_{Z}).$$
Then by the same arguments as the proof of \cite[Proposition 3.1]{CHL24}, one can see that $p\bM$ is b-Cartier. Note here that the existence of $\tGamma$ and the last statement of the proposition also follow.
\end{proof}

For later use, we will show some additional properties enjoyed by the canonical bundle formula when the base is a curve. It is contingent on the Existence of Complements Conjecture (Conjecture~\ref{conj: boundedness and existence of n complement nft}) and the following Good Minimal Model Conjecture, both of which hold true in dimensions $\leq 3$.
\begin{conj}[Good Minimal Model Conjecture]\label{conj: exist gmm}
Let $d$ be a positive integer. Assume that $(X/Z,B)$ is an lc pair of dimension $d$ such that $K_X+B$ is pseudo-effective over $Z$. Then $(X/Z,B)$ has a good minimal model $(Y/Z,B_Y)$ over $Z$, that is, $(Y,B_Y )$ is a minimal model of $(X,B)$ over $Z$ and $K_Y + B_Y$ is semi-ample over $Z.$
\end{conj}
\begin{lem}\label{lem: cbf (d,1)}
Let $d$ be a positive integer and $\Gamma\subset[0,1]$ a DCC set. Assume that both the Existence of Complements Conjecture (Conjecture \ref{conj: boundedness and existence of n complement nft}) and the Good Minimal Model Conjecture hold in dimension $d$. Then there exist a positive integer $p$, a finite subset $\Gamma_0\subset\Gamma\cap\QQ$, two DCC subsets $\tGamma_0 \subset\tGamma\subset [0,1]$, all depending only on $\Gamma$ and $d$, such that the following holds. 
\begin{enumerate}
    \item Let $(X, B)$ be a projective log canonical pair, and $f\colon X\rightarrow Z$ a fibration onto a smooth projective curve $Z$ such that $K_X+B\sim_{\RR, Z} B^v \sim_{\RR, Z}0$ and $\kappa(K_X+B)\ge0$. Let $B_Z$ and $B_Z'$ be the discriminant parts of the canonical bundle formula for $(X, B)$ and $(X, B^h)$ over $Z$ respectively. Then we may take a common moduli part $\bM$ of the canonical bundle formula for the two pairs $(X, B^h)$ and $(X, B)$ such that
\[
B^h\in \Gamma_0,\quad  B_Z'\in \tGamma_0,\quad B_Z\in \tGamma,\quad p\bM_Z\in \ZZ.
\]
 \item The set $\tGamma_0$ is a hyperstandard set, so its only possible accumulation point is $1$.
 \item We have 
 \[
 \tGamma\subset\bigcup_{m\in\Zz_{>0}}\left(\frac{1}{m}\Gamma+\tGamma_0\right),
 \]
 where $\frac{1}{m}\Gamma+\tGamma_0=\left\{\frac{a}{m}+b\,\bigg|\, a\in\Gamma,\, b\in\tGamma_0\right\}$. In particular, the accumulation complexity of $\tGamma$ is at most one more than that of $\Gamma$.
\end{enumerate}
\end{lem}
\begin{proof}
(1) The positive integer $p$ and the DCC sets $\Gamma_0, \,\tGamma_0, \tGamma$ exist by Lemma~\ref{lem: cbfindex}, applied to the pairs $(X, B)$ and $(X, B^h)$ respectively. 

It remains to show that the same moduli part may be taken for the canonical bundle formula of the two pairs $(X, B^h)$ and $(X, B)$ over $Z$.

Since $B^v$ is a vertical $\RR$-divisor such that $B^v\in \Gamma$ and $B^v\sim_{\RR, Z}0$, it is a finite $\RR_{\geq 0}$-linear combination of fibers:
\[
B^v = \sum_{1\leq j\leq n}c_{j} f^* z_{j},
\]
where $c_{j}\in \RR_{\geq 0}$ is such that $mc_j\in \Gamma$ for any multiplicity $m$ of an irreducible component appearing in the fiber $f^* z_{j}$. By the construction of the discriminant parts (cf.~\cite[Lemma 7.4]{PS09}), we have
\begin{equation}\label{eq: BZ vs BZ'}
B_{Z} = B_{Z}' + \sum_{1\leq j\leq n}c_{j} z_{j}.
\end{equation}
Thus $f^*B_Z = f^*B_Z' + B^v$, and we may take the moduli parts of the canonical bundle formulas for the pairs $(X, B)$ and $(X, B^h)$ over $Z$ to be the same.

\medskip

(2) follows from the last statement of Lemma~\ref{lem: cbfindex} and the fact that $\Gamma_0$ is finite.

\medskip

(3) By \eqref{eq: BZ vs BZ'}, the set $\tGamma$ may be taken as a subset of $\bigcup_{m\in\Zz_{>0}}(\frac{1}{m}\Gamma+\tGamma_0)$.
\end{proof}

\subsection{Crepant birational models of pairs with an lc-trivial fibration over a curve}
For an lc pair $(X, B)$ and a fibration $f\colon X\rightarrow Z$ onto a curve such that $K_X+B\sim_{\RR, Z} 0$, we will use the minimal model program to construct a crepant birational model $(Y, B_Y)$ over $Z$ such that the coefficients of $B_Y$ are the ``right" ones.

The following lemma is well-known to experts. For lack of definite reference, we write a proof of it.
\begin{lem}\label{lem: R-trivial}
Let $f\colon X\rightarrow Z$ and $g\colon Y\rightarrow Z$ be projective morphisms between normal varieties, and $D$ an $\RR$-Cartier $\RR$-divisor on $X$ such that $D\sim_{\RR, Z}0$. Let $\phi\colon X\dashrightarrow Y$ be a birational contraction over $Z$. Then $D_Y:=\phi_* D$ is $\RR$-Cartier, $D_Y\sim_{\RR, Z}0$, and there is an equality of $\rb$-$\RR$-Cartier divisors $\overline D=\overline{D_Y}$.
\end{lem} 
\begin{proof}
Since $D\sim_{\RR, Z}0$, there are nonzero rational functions $h_1,\dots, h_r\in K(X)$ and real numbers $b_1,\dots, b_r$ such that 
\[
D=\sum_{1\leq i\leq r} b_i\cdot (h_i)_X + f^* H
\]
where $H$ is some $\RR$-Cartier $\RR$-divisor on $Z$. 

Let $p\colon W\rightarrow X$ and $q\colon W\rightarrow Y$ be a common resolution. Since $\phi\colon X\dashrightarrow Y$ is a birational contraction, there is an open subset $U\subset Y$ such that $\codim_Y(Y\setminus U)\geq 2$, and $q^{-1}(U)\rightarrow U$ and $p(q^{-1}(U))\rightarrow U$ are isomorphisms. Thus $D_Y$ is the closure of its restriction $(D_Y)|_U$ to $U$, which is exactly $\phi_*\left( D|_{p(q^{-1}(U))}\right)$, and we infer that
\[
D_Y =\sum_{1\leq i\leq r} b_i (\phi^{-1*} h_i)_Y + g^* H.
\]
Therefore, $D_Y\sim_{\RR, Z}0$, and
\[
q^*D_Y = \sum_{1\leq i\leq r} b_i (q^*\phi^{-1*} h_i)_W + q^*g^* H = \sum_{1\leq i\leq r} b_i (p^*h_i)_W + p^*f^*H = p^* D.
\]
\end{proof}

% \begin{lem}\label{lem: model lc trivial}
% Let $(X, B)$ be an lc pair, and $f\colon X\rightarrow Z$ a contraction such that $K_X+B\sim_{\RR, Z} 0$. Let $\phi\colon X\dashrightarrow Y$ is a birational contraction onto a normal variety $Y$ over $Z$ such that $K_Y+B_Y$ is $\RR$-Cartier, where $B_Y=\phi_* B$.   Then $(Y, B_Y)$ an lc pair such that $K_Y+B_Y\sim_{\RR, Z} 0$, and there is an equality of $\mathrm{b}$-divisors $\overline{K_X+B} = \overline{K_Y+B_Y}$, meaning that the pull-backs of $K_Y+B_Y$ and $K_X+B$ to a (equivalently, every) common higher birational model are the same.
% \end{lem} 
% \begin{proof}
% Let $p\colon W\rightarrow X$ and $q\colon W\rightarrow Y$ be a common resolution. Then, by \cite[Lemma~3.6.4]{BCHM10} with $D=K_X+B$ and $D'=0$, we have
% \begin{equation}\label{eq: num trivial crepant contraction}
% p^*(K_X+B) = q^*p_*(K_X+B) = q^*(K_Y+B_Y),
% \end{equation}
% and hence $\overline{K_X+B} = \overline{K_Y+B_Y}$ holds. Since $(X, B)$ is lc, so is $(Y, B_Y)$ by \eqref{eq: num trivial crepant contraction}. Also, we have the relative $\RR$-linear equivalence 
% \[
% K_Y+B_Y =q_*q^*(K_Y+B_Y)= q_*p^*(K_X+B)\sim_{\RR, Z}0.
% \]
% \end{proof}

\begin{lem}\label{lem: model nef Bv}
Let $(X, B)$ be an lc pair, and $f\colon X\rightarrow Z$ a contraction such that $K_X+B\sim_{\RR, Z} 0$. Then there is a birational contraction $\phi\colon X\dashrightarrow Y$ over $Z$ such that $(Y, B_Y)$ is crepant to $(X, B)$ and $B_Y^v\sim_{\RR, Z} 0$, where $B_Y=\phi_*B$ and $B_Y^v$ is the vertical$/Z$ part of $B_Y$. More precisely, the following holds:
\begin{enumerate}
    \item $K_{Y}+B_{Y}\sim_{\RR, Z} K_{Y}+B_{Y}^h\sim_{\RR, Z} B_{Y}^v\sim_{\RR, Z}0$, where $B_{Y}^h$ is the horizontal$/Z$ part of $B_{Y}$.
    \item We have an equality of $\mathrm{b}$-divisors $\overline{K_X+B} = \overline{K_{Y}+B_{Y}}$.
\end{enumerate} 
%In particular, if $d\leq 3$, then $(Y, B_Y)$ as above always exists.
\end{lem}
\begin{proof}
By \cite[Theorem 1.1]{Has19}, we may run a $(K_X+B^h)$-MMP$/Z$ that terminates with a good minimal model $\phi:X\dashrightarrow Y$ over $Z$. Then pair $(Y,B_{Y}:=\phi_* B)$ is lc and has the required properties in (1). The property (2) follows from Lemma~\ref{lem: R-trivial}.
\end{proof}

\begin{lem}\label{lem:adltmodel}
Assume that $(X,B)$ is an lc pair and $v$ a divisorial valuation over $X$ such that the log discrepancy $a(v, X, B)\leq 1$. Then there is a projective birational morphism $\varphi\colon W\to X$ and a $\Qq$-factorial dlt pair $(W,B_W)$ crepant to $(X,B)$ such that $E_W(v):=\Center_W(v)$ is a divisor and $\mult_EB_W=1$ for any prime $\varphi$-exceptional divisor $E\neq E_W$.
\end{lem}
\begin{proof}
We may assume that the center of $v$ on $X$ is not a divisor. Let $g\colon Y\to X$ be a log resolution of $(X,\Supp B)$ such that the center of $v$ on $Y$ is a prime divisor $E_1$. Let $E_1,\dots,E_n$ be all the prime divisors on $Y$ that are exceptional$/X$ and set $B_Y:=g_*^{-1}B+(1-a(v,X,B))E_1+\sum_{i=2}^nE_i$. Then $(Y,B_Y)$ is log smooth and
$$K_Y+B_Y-g^*(K_X+B)=G\ge0$$
for some $\Rr$-divisor $G$ that is exceptional over $ X$ and $E_1\not\subset\Supp G$. By \cite[Theorem 1.8]{Bir12}, we may run a $(K_Y+B_Y)$-MMP over $X$ that contracts exactly the components of $G$ and get a model $W$. Let $B_W$ be the strict transform of $B_Y$ on $W$ and then $(W,B_W)$ has the required properties.
\end{proof}

\begin{lem}\label{lem: model right coefficient}
Let $(X, B)$ be a projective lc pair of dimension $d\geq 2$, and $f\colon X\rightarrow Z$ a fibration onto a smooth projective curve $Z$ such that $K_X+B\sim_{\Rr,Z}0$. Let $S=\{z_i\mid 1\leq i\leq n\}\subset Z$ be a (possibly empty) finite set of points such that $B_{z_i}\neq 0$ (see \ref{div} for the definition of $B_{z_i}$) and $(X, B)$ is klt along the fibers $f^*z_i$ for each $1\leq i\leq n$, and let $\{v_i\}_{1\leq i\leq n}$ be a set of divisorial valuations with $\mathrm{center}_X(v_i)\subset \Supp(f^*z_i)$ and $a(v_i, X, B) \leq 1$ for each $i$.

Assume that the Good Minimal Model Conjecture (Conjecture \ref{conj: exist gmm}) holds in dimension $d$. Then there is a birational map $\phi\colon X\dashrightarrow Y$ over $Z$ and  a projective log canonical pair $(Y, B_{Y})$ such that the following holds.
\begin{enumerate}
    %\item  $K_{Y}+B_{Y}\sim_{\RR, Z} K_{Y}+B_{Y}^h\sim_{\RR, Z} B_{Y}^v\sim_{\RR, Z}0,$ where, as before, $B_{Y}^h$ and $B_{Y}^v$ denotes the horizontal$/Z$ part and vertical$/Z$ parts of $B_{Y}$ respectively.
    \item We have an equality of $\mathrm{b}$-divisors $\overline{K_{X}+B} = \overline{K_{Y}+B_{Y}}$. 
    %\item $\phi$ is an isomorphism over $Z\setminus\{z\}$.
    \item Let $g\colon Y\rightarrow Z$ be the structural fibration. Then, for each $1\leq i\leq n$, we have
    $$E_{Y}(v_i):=\Center_Y(v_i) = \Supp(g^*z_i),$$ which is thus irreducible.
        \item We have $B_Y^h=\phi_*B^h$. Moreover, if $d\geq 3$ then $\lfloor B_Y^h\rfloor =0$.
    \item We have $B_Y^v=\sum_{1\leq i\leq n} b_i E_{Y}(v_i)$, where $b_i:=1-a(v_i, X, B)$ for $1\leq i\leq n$. (In case $S=\emptyset$, we have $B_Y^v=0$.)
\end{enumerate} 
In particular, if $d\leq 3$, then $(Y, B_Y)$ as above always exists.
\end{lem}
\begin{proof}
%We only give the proof for the case $d=3$; the case $d=2$ is similar and much easier. 
By Lemma \ref{lem:adltmodel}, there is a projective birational morphism $\varphi\colon W\to X$ and a $\Qq$-factorial dlt pair $(W,B_W)$ crepant to $(X,B)$ such that $E_{W}(v_i):=\Center_W(v_i)$ is a divisor for $1\leq i\leq n$ and $\mult_E B_W=1$ for any prime $\varphi$-exceptional divisor $E\neq E_{W}(v_i)$ for any $1\leq i\leq n$. Since $\Center_X(v_i)\subset f^*z_i$, $E_{W}(v_i)$ is vertical/$Z$ for each $i$. Note that the coefficient of $E_{W}(v_i)$ in $B_W$ is $b_i:=1-a(v_i, X, B)$, which is non-negative.

Define a vertical$/Z$ divisor $G$ on $W$ such that its support is contained in the union of the fibers over $z_i$ ($1\leq i\leq n)$ and of $\Supp(B_W^v)$, and its part over a point $z\in Z$ is as follows:
\[
G_{z}:=
\begin{cases}
(\varphi^*f^*z)_\red - E_{W}(v_i), & \text{ if $z\in S$,} \\
(\lceil B_W\rceil-\lfloor B_W\rfloor)_{z}, & \text{ if $z\notin S$.} 
\end{cases}
\]
Then $(W,B_W+\epsilon G)$ is dlt for any $0<\epsilon\ll 1$. By construction, $G_{z}$ is either $0$ or very exceptional over $Z$. By \cite[Theorem 1.8]{Bir12}, we may run a $(K_W+B_W+\epsilon G)$-MMP over $Z$, $\psi\colon W\dashrightarrow X_1$, which contracts exactly the components of $G$. Let $B_{X_1}:=\psi_* B_W$ and $f_1\colon X_1\rightarrow Z$ is the induced fibration. Then $(X_1,B_{X_1})$ is dlt, and by Lemma~\ref{lem: R-trivial}, we have 
\[
K_{X_1}+B_{X_1} = \psi_{*}(K_W+B_W)\sim_{\RR, Z} 0.
\] 
For $1\leq i\leq n$, the fiber $f_1^*z_i$ is irreducible with support $E_{X_1}(v_i) = \psi_* E_W(v_i)$, which is the center of $v_i$ on $X_1$, and  
\begin{equation}\label{eq: BX1 zi}
    (B_{X_1})_{z_i} = \psi_*(B_W)_{z_i} =b_i\psi_*E_W(v_i) =b_i E_{X_1}(v_i).
\end{equation}
Let $H_{X_1}$ be a general very ample divisor on $X_1$, so that 
\begin{itemize}
    \item the divisor $H_{X_1} -\lfloor B_{X_1} \rfloor$ is ample,
    \item the pair $(X_1, B_{X_1}+H_{X_1})$ is dlt, and
    \item the divisor $H_{X_1}$ is irreducible and normal.
\end{itemize}
Let $\rho\colon \tX_1\rightarrow X_1$ be the blow-up along $H_{X_1}\cap \lfloor B_{X_1}\rfloor$, and $\gE$ the set of valuations of $\rho$-exceptional divisors dominating the codimension-$2$ locus of $H_{X_1}\cap \lfloor B_{X_1}\rfloor$. By \cite[Corollary~1.4.3]{BCHM10}, there is a birational morphism $\psi_1\colon W_1\rightarrow X_1$ such that $W_1$ is $\QQ$-factorial and the $\psi_1$-exceptional divisors correspond to the elements of $\gE$. Let $B_{W_1}$ be the strict transform of $B_{X_1}$. Then $(W_1, B_{W_1})$ is lc, and since $a(v, X_1, B_{X_1}) =1$ for any $v\in \gE$, we have 
\[
K_{W_1}+ B_{W_1} = \psi_1^*(K_{X_1}+B_{X_1}) \sim_{\RR, Z} 0.
\]
Let $H_{W_1}$ be the strict transform of $H_{X_1}$ on $W_1$. 
%By the generality of $H_{X_1}$, 
Since $(X_1, B_{X_1}+H_{X_1})$ is dlt, $(W_1, B_{W_1}+H_{W_1})$ is lc. Since $H_{X_1} -\lfloor B_{X_1} \rfloor$ is ample, its pull-back $\varphi_1^*(H_{X_1} -\lfloor B_{X_1}\rfloor)$ is big$/Z$, and so is 
\[
K_{W_1}+ B_{W_1} +H_{W_1} \sim_{\RR, Z} H_{W_1} \sim_{\RR, Z}\varphi_1^*(H_{X_1} -\lfloor B_{X_1}\rfloor) + \lfloor B_{W_1}\rfloor.
\]
%\cheninline{follows from $\varphi_1^*H_{X_1}=H_{W_1}+Exc(\varphi_1)$ and $\varphi_1^*\lf B_{X_1}\rf=\lf B_{W_1}\rf+Exc(\varphi_1)$? shall we explain?}

Since we are assuming the Good Minimal Model Conjecture in dimension $d$, we may take the canonical model $\psi_1\colon W_1\dashrightarrow Y$ of $K_{W_1}+B_{W_1}+H_{W_1}$  over $Z$, so that $K_{Y} + B_{Y} +H_{Y}$ is ample over $Z$, where $B_{Y}$ and $H_{Y}$ are the strict transforms of $B_{W_1}$ and $H_{W_1}$ respectively. 

Let $g\colon Y\rightarrow Z$ and $\phi\colon X\dashrightarrow Y$ be the induced maps.
We will verify that $(Y, B_Y)$ satisfies the properties (1)--(4) of the lemma. For the reader's convenience, we illustrate the constructed birational maps in the following diagram:
\[
\begin{tikzcd}
  & W\arrow[ld, "\varphi"']\arrow[rd, dashed, "\psi"] &  & W_1 \arrow[ld, "\varphi_1"'] \arrow[rd, dashed, "\psi_1"] &  \\  
X &   & X_1 & & Y
\end{tikzcd}
\]

By Lemma~\ref{lem: R-trivial}, we have in fact
\[
\overline{K_Y+B_Y} = \overline{K_{W_1}+B_{W_1}} = \overline{K_{X_1}+B_{X_1}} =  \overline{K_{W}+B_{W}} = \overline{K_X+B}. 
\]
In particular, (1) holds.

By \eqref{eq: BX1 zi}, $(B_{X_1})_{z_i} = b_iE_{X_1}(v_i)$ for each $z_i\in S$. Since $\varphi_1\colon W_1\rightarrow X_1$ only extracts divisors $E$ not contained in the fibers over $z_i$, we infer that $E_{W_1}(v_i):=\Center_{W_1}(v_i)$ supports the fiber $\varphi_1^*f_1^*z_i$. It follows that $E_{Y}(v_i) = \psi_{1*}E_{W_1}(v_i)$ also supports the fiber $g^*z$. This proves (2).

Note that the birational morphisms $\varphi\colon W\rightarrow X$ and $\varphi_1\colon W_1\rightarrow X_1$ only extract divisors $E$ with $a(E, X, B)\in\{0,1\}$ or with $E=E_{W}(v_i)$ for some $1\leq i\leq n$, while $\psi\colon W\rightarrow X_1$ and $\psi_1\colon W_1\rightarrow Y$ are birational contractions. To show (3) and (4), it suffices to show that every irreducible component of $\lfloor B_{W_1}\rfloor$ that is not generically finite over $Z$ is contracted by $\psi_1$. We achieve this by showing the following two claims.

\begin{claim}\label{claim: HW BW empty intersect}
We have $\Supp H_{W_1} \cap\, \Supp \lfloor B_{W_1}\rfloor =\emptyset$. 
\end{claim}
\noindent \textit{Proof of the Claim~\ref{claim: HW BW empty intersect}.} Suppose on the contrary that $\Supp H_{W_1} \cap\, \Supp \lfloor B_{W_1}\rfloor \neq \emptyset$. Since $W_1$ is $\Qq$-factorial, each irreducible component of $\Supp H_{W_1} \cap\, \Supp \lfloor B_{W_1}\rfloor$ is of codimension 2 in $W_1$. By construction, the image of $\Supp H_{W_1} \cap\, \Supp \lfloor B_{W_1}\rfloor$ in $X_1$ has codimension at least 3 in $X_1$, and hence $\Supp H_{W_1} \cap\, \Supp \lfloor B_{W_1}\rfloor$ is contained in the exceptional locus $\Exc(\varphi_1)$ of $\varphi_1$. Since $X_1$ is $\QQ$-factorial, $\Exc(\varphi_1)$ is of pure codimension 1  (\cite[Corollary~2.63]{KM98}), and hence is the union of the prime divisors $E_{W_1}(v):=\Center_{W_1}(v)$ for $v\in \gE$. 

Note that the intersection of lc centers is a union of lc centers (cf.~\cite[Theorem 5.14]{Kol13}), so the irreducible components of $\Supp H_{W_1} \cap\, \Supp \lfloor B_{W_1}\rfloor$ are lc centers of the lc pair $(W_1, B_{W_1}+H_{W_1}+\sum_{v\in \gE}E_{W_1}(v))$. But each of them is contained in $E_{W_1}(v)$ for some $v\in\gE$, which is a lc place of $(W_1, B_{W_1}+H_{W_1}+\sum_{v\in \gE}E_{W_1}(v))$, and this contradicts the fact that $(W_1, B_{W_1}+H_{W_1}+\sum_{v\in \gE}E_{W_1}(v))$ is lc.

\begin{claim}\label{claim: contract BW1}
Any component of $\lfloor B_{W_1}\rfloor$ that is not generically finite over $Z$ is contracted by $\psi_1$. In particular, each component of $\lfloor B_{W_1}\rfloor$ is contracted by $\psi_1$ if $d\geq 3$.
\end{claim}
\noindent \textit{Proof of Claim~\ref{claim: contract BW1}.} Let $p\colon W_2\to W_1$, $q\colon W_2\to Y$ be a common resolution. Then 
$$p^{*}H_{W_1}=q^{*}H_{Y}+E$$ 
for some $q$-exceptional $\RR$-divisor $E\ge 0$. Let $D_2\subset W_2$ be the strict transform of an irreducible component $D_1$ of $\lfloor B_{W_1}\rfloor$ that contains a vertical/$Z$ curve, and $\Sigma$ a general curve on $D_2$ that is contracted by $D_2\rightarrow Z$. By Claim~\ref{claim: HW BW empty intersect}, we have $\Supp H_{W_1}\cap \Supp B_{W_1}=\emptyset$, and hence
$$0=H_{W_1}\cdot(p_*\Sigma) = p^{*}H_{W_1}\cdot \Sigma=(q^{*}H_{Y}+E)\cdot\Sigma.$$ 
If $D_1$ were not contracted by $\psi_1$, then $D_2$ is not contracted by $q$ and hence $q(\Sigma)$ is still a curve. Now we obtain the following contradiction:
$$0\le E\cdot \Sigma=-H_{Y}\cdot q(\Sigma)<0$$
as $H_Y$ is ample over $Z$. 
\end{proof}

\section{Accumulation properties  for Iitaka volumes}\label{sec: DCC}

\subsection{Proofs of Theorems \ref{thm: main1}, \ref{thm:main2}, \ref{thm: main3}}
\begin{lem}\label{lem:ivolaccfollowfrom2conj}
Let $d$ be a positive integer. Assume that both the Existence of Complements Conjecture (Conjecture \ref{conj: boundedness and existence of n complement nft}) and the Good Minimal Model Conjecture (Conjecture \ref{conj: exist gmm}) hold in dimension $d$. Then Conjecture \ref{conj:ivoldcc} holds in dimension $d$. 
\end{lem}
\begin{proof}%[Proof of Theorem \ref{thm: ivoldim3}]
Let $(X,B)$ be a projective lc pair of dimension $d$ such that $\kappa(K_X+B)\ge0$ and $B\in\Ii$. Possibly replacing $(X,B)$ with a good minimal model, we may assume that $(X,B)$ is lc and $K_X+B$ is semi-ample. Let $f\colon X\to Z$ be the ample model of $K_X+B$.

By \cite[Theorem 1.1]{HMX14} and \cite[Proposition 3.1]{CHL24}, if the Existence of Complements Conjecture and the Good Minimal Model Conjecture hold in dimension $d$, then we can find a positive integer $p$ and a DCC set $\Ii'\subset[0,1]$ depending only on $d$ and $\Ii$ such that there is a generalized pair $(Z,B_Z+\bM)$ satisfying that $B_{Z}\in\Ii'$, $p\bM$ is b-Cartier, and 
$$p(K_X+B)\sim pf^*(K_Z+B_Z+\bM_Z).$$
%where $B_{Z}$ is the discriminant part of the canonical bundle formula for $(X,B)$ over $Z$. 
Moreover, $K_Z+B_Z+\Mm_{Z}$ is big as $Z$ is the ample model of $K_X+B$. By \cite[Theorem 1.3]{Bir21}, $\vol(K_Z+B_Z+\bM_{Z})$ belongs to a DCC set $\mathcal{V}$ depending only on $d$, $p$, and $\tGamma$. It follows by \cite[II Lemma 2.11]{Nak04} that the Iitaka volume of $K_X+B$ equals to
$$\vol(K_Z+B_Z+\bM_{Z})$$
and thus also belongs to the DCC set $\mathcal{V}$. We may finish the proof.
\end{proof}

\begin{cor}\label{cor: conj holds d=3}
Conjecture \ref{conj:ivoldcc} holds when $d\le 3.$
\end{cor}
\begin{proof}
This immediately follows from Lemma \ref{lem:ivolaccfollowfrom2conj}, since the Existence of Complements Conjecture and the Good Minimal Model Conjecture hold in dimension $d\leq 3$.
\end{proof}

\begin{lem}\label{lem: finite acc complexity}
Let $d$ be a positive integer, $\Gamma\subset[0,1]$ a DCC subset, and $M$ a positive real number. Assume that both the Existence of Complements Conjecture (Conjecture \ref{conj: boundedness and existence of n complement nft}) and the Good Minimal Model Conjecture (Conjecture \ref{conj: exist gmm}) hold in dimension $d$. If $\Gamma$ has finite accumulation complexity, then so does $\Ivol_\lc^\Gamma(d,1)_{\le M}$. 
\end{lem}
\begin{proof}
Suppose that $(X,B)$ is a projective lc pair of dimension $d$ such that 
\[
\kappa(K_X+B)=1,\quad B\in\Gamma, \quad \vol_1(K_X+B)\leq M.
\]
Possibly replacing $(X,B)$ with a good minimal model, we may assume that $X$ is $\Qq$-factorial and $K_X+B$ is semi-ample. Let $f\colon X\to Z$ be the contraction induced by $K_X+B$ where $Z$ is a curve. Let $B^h$ and $B^v$ be the horizontal$/Z$ and vertical$/Z$ parts of $B$ respectively. 
By Lemma~\ref{lem: model nef Bv}, we may assume that
$$K_{X}+B^h\sim_{\Rr,Z}0\sim_{\Rr,Z}B^v.$$
% Since $K_X+B^h\sim_{\Rr,Z}-B^v$ is vertical over $Z$, we can run a $(K_X+B^h)$-MMP$/Z$ that terminates with a model $X'$ on which 
% $$K_{X'}+(B^h)'\sim_{\Rr,Z}-(B^v)'\sim_{\Rr,Z}0,$$
% where $(B^h)'$ and $(B^v)'$ are the strict transforms of $B^h$ and $B^v$ on $X'$ respectively. Note that $(X,B)$ is crepant to $(X',B')$. So we can replace $(X,B)$ with $(X',B')$, and thus assume that
% $$K_{X}+B^h\sim_{\Rr,Z}0\sim_{\Rr,Z}B^v.$$

Let $p\in \Zz_{>0}$, $\Gamma_0, \, \tGamma_0, \, \tGamma$ be as in Lemma~\ref{lem: cbf (d,1)}. Applying the canonical bundle formula, we have
\[
K_X + B \sim_{\RR, Z} f^*(K_Z+B_Z+\bM_Z)
\]
where $B_Z \in \tGamma$ and $\bM_Z\in \frac{1}{p}\ZZ$ are the discriminant and moduli parts respectively. Therefore,
\[
\vol_1(K_X+B) = \deg(K_Z+B_Z+\bM_Z) = 2g(Z)-2+\deg B_Z + \deg \bM_Z.
\]
Under the condition that $\vol_1(K_X+B)\leq M$, each summand above varies in a set of finite accumulation complexity, and hence so does the sum $\vol_1(K_X+B)$:
\begin{itemize}
    \item Since $g(Z)\in \ZZ_{\geq 0}$, $\deg\bM_Z\in\frac{1}{p}\ZZ_{\geq 0}$, and $\vol_1(K_X+B)\leq M$, there are finitely many choices for $g(Z)$ and $\deg \bM_Z$.
    \item Since $B_Z\in \tGamma$, we have $\deg B_Z\in (\sum\tGamma)_{\leq M}$. Since $\Gamma$ is of finite accumulation complexity, so is the set $\tGamma$ by Lemma~\ref{lem: cbf (d,1)}. It is then clear that $(\sum\tGamma)_{\leq M}$ has finite accumulation complexity.
\end{itemize}
\end{proof}

\begin{cor}\label{cor:fac}
If $\Gamma\subset[0,1]$ is a DCC subset with finite accumulation complexity, then both $\Ivol_\lc^\Gamma(2,1)_{\le M}$ and $\Ivol_\lc^\Gamma(3,1)_{\le M}$ have finite accumulation complexity.
\end{cor}

\begin{thm}\label{thm: closed}
Let $\Gamma\subset[0,1]$ be a closed DCC subset. Assume that the Existence of Complements Conjecture and the Good Minimal Model Conjecture hold in dimension $d$. Then $\Ivol_\lc^{\Gamma}(d,1)$ is closed. In particular, $\Ivol_\lc^{\Gamma}(d,1)$ is closed for $d\in\{2,3\}$.
\end{thm}
\begin{proof}
Without loss of generality, we may assume that $0\in \Gamma$. Let the number $p\in \Zz_{>0}$, and the DCC subsets $\Gamma_0, \,\tGamma_0, \, \tGamma$ of $[0,1]$ be as in Lemma~\ref{lem: cbf (d,1)}.

By Lemma~\ref{lem: model nef Bv}, we have
\[
\Ivol^\Gamma_\lc(d,1) = \{\vol_1(K_X+B)\mid (X, B)\in \gP^\Gamma_\lc(d,1), \text{$K_X+B$, and $B^v$ are nef}\}
\]
where $B^h$ and $B^v$ are the horizontal and vertical parts of $B$ respectively with respect to the Iitaka fibration of $(X, B)$.

% By Lemma~\ref{lem: cbfindex}, there is a positive integer $p$, a finite subset $\Gamma_0\subset\Gamma\cap\QQ$, two DCC subsets $\tGamma_0 \subset\tGamma$, all depending only on $\Gamma$ and $d$, such that the following holds. 
% \begin{enumerate}
% \item Let $(X, B)$ be a pair in $\gP^\Gamma_\lc(d,1)$ and $f\colon X\rightarrow Z$ the Iitaka fibration of $(X, B)$. Suppose that $K_X+B\sim_{\RR, Z} B^v \sim_{\RR, Z}0$, and let $B_Z$ and $B_Z'$ be the discriminant parts of the canonical bundle formula for $(X, B)$ and $(X, B^h)$. Then we may take a common moduli part $\bM_Z$ of the canonical bundle formula for the two pairs $(X, B^h)$ and $(X, B)$ such that
% \[
% B^h\in \Gamma_0,\quad  B_Z'\in \tGamma_0,\quad B_Z\in \tGamma,\quad p\bM_Z\in \ZZ.
% \]
%  \item $\tGamma_0$ is a hyperstandard set, so its only possible accumulation point is $1$.
%  \item The accumulation complexity of $\Gamma_2$ is at most one more than the accumulation complexity of $\Gamma$.
% \end{enumerate}

We need to show that $\Ivol^\Gamma_\lc(d,1)'\subset \Ivol^\Gamma_\lc(d,1)$. So take any $v\in \Ivol^\Gamma_\lc(d,1)'$. Then there is a sequence of fibrations $f_i\colon X_i\rightarrow Z_i$ with $(\dim X_i, \dim Z)=(d,1)$, and a boundary $\RR$-divisor $B_i$ on $X_i$ such that the following holds:
\begin{enumerate}
\item $(X_i, B_i)\in \gP_\lc^\Gamma(d,1)$.
\item One has
\[
K_{X_i} + B_i^h \sim_{\RR, \ZZ} B_i^v\sim_{\RR, \ZZ} 0,
\]
where $B_i^h$ and $B_i^v$ are the horizontal$/Z_i$ and vertical$/Z_i$ parts of $B_i$ respectively.
\item The Iitaka volumes $\vol_1(K_{X_i}+B_i)$ are increasing to $v$.
\end{enumerate} 
By the canonical bundle formula for the pairs $(X_i, B_i)$ and $(X_i, B_i^h)$ over $Z$, we have for each $i$
\[
K_{X_i} + B_i^h \sim_{\RR} f_i^*(K_{Z_i}+B_{Z_i}' +\bM_{Z_i}),\quad K_{X_i} + B_i \sim_{\RR} f_i^*(K_{Z_i}+B_{Z_i} +\bM_{Z_i}),
\]
where $B_{Z_i}'\in \tGamma_0$ and $B_{Z_i}\in\tGamma$ are the discriminant parts, and $\bM_{Z_i}\in\frac{1}{p}\ZZ$ is a common moduli part. Thus the Iitaka volume of $K_{X_i} + B_i$ is
\begin{equation}\label{eq: Ivol v_i}
\vol_1(K_{X_i} + B_i) = \deg(K_{Z_i}+B_{Z_i} + \bM_{Z_i}).
\end{equation}
Note that $\deg K_{Z_i} = 2g(Z_i)-2\in \ZZ$, $\deg B_{Z_i} \in \sum_{\leq n_i}\tGamma$ for some $n_i\in \ZZ_{\geq 0}$, and $\deg\bM_{Z_i}\in \frac{1}{p}\ZZ_{\geq 0}.$ Since $\vol_1(K_{Z_i}+B_i)< v$ for all $i$, there are finitely many possibilities for the triple $(g(Z_i), n_i, \deg\bM_{Z_i})$. By possibly passing to a subsequence of $\{(X_i, B_i)\}_i$, one may assume that $(g(Z_i), n_i, \deg\bM_{Z_i})$ are the same for all $i$, which is then denoted by $(g, n, \lambda)$. By passing to a subsequence, we may assume that $\Supp(B_{Z_i})$ consists of exactly $n$ points, which are denoted by $\{z_{ij}\}_{1\leq j\leq n}$.

Thus, the increment of $\vol_1(K_{X_i} + B_i)$ with $i$ is due to the increment of $\deg B_{Z_i}$. We need to make a more careful analysis of $\deg B_{Z_i}$. Recall that, $B_i^v$ is an effective vertical $\RR$-divisor such that $B_i^v\sim_{\RR, Z}0$, so it is a finite $\RR_{\geq 0}$-linear combination of fibers:
\[
B_i^v = \sum_{1\leq j\leq n}c_{ij} f_i^* z_{ij},
\]
where $c_{ij}\in \RR_{\geq 0}$. By the construction of the discriminant parts, we have (cf.~proof of Lemma~\ref{lem: cbf (d,1)})
\[
B_{Z_i} = B_{Z_i}' + \sum_{1\leq j\leq n}c_{ij} z_{ij}.
\]
Since $B_i^v\in \Gamma$, we have $mc_{ij}\in \Gamma$ for any $m$ appearing as the multiplicity of an irreducible component in $f_i^* z_{ij}$. Let $m_{ij}$ be the maximal multiplicity of an irreducible component in $f_i^* z_{ij}$. Then we have
\[
\mult_{z_{ij}} B_{Z_i}'\geq 1-\frac{1}{m_{ij}},
\]
and there is some $\gamma_{ij}\in \Gamma$ such that $c_{ij} = \frac{\gamma_{ij}}{m_{ij}}$. It follows that
\begin{equation}\label{eq: coeff B_Z_i}
\mult_{z_{ij}} B_{Z_i} = \mult_{z_{ij}} B_{Z_i}' + c_{ij}\geq 1-\frac{1}{m_{ij}}+\frac{\gamma_{ij}}{m_{ij}}.
\end{equation}
By relabelling the $z_{ij}\in \Supp(B_{Z_i})$, $1\leq j\leq n$, and taking subsequence, we may assume that the sequences of numbers $\{b_{ij}':=\mult_{z_{ij}}B_{Z_i}'\}_i$, $\{b_{ij}:=\mult_{z_{ij}}B_{Z_i}\}_i$,  $\{m_{ij}\}_i$, and $\{\gamma_{ij}\}_i$ are all non-decreasing for each $1\leq j\leq n$. Set $\delta_{ij}=\frac{\gamma_{ij}}{m_ij}$, and $$b_j':=\lim_i b_{ij}',\quad b_j:=\lim_i b_{ij}, \quad \delta_j :=\lim_i \delta_{ij}.$$
Then we have $$0< b_j = b_j'+\delta_j \leq 1.$$ Let $\lct(z_{ij})$ be the log canonical threshold of $f_i^*z_{ij}$ with respect to $(X_i, B_i^h)$ for any $i,\, j$. By the definition of the discriminant part $B_{Z_i}'$, we know that $$\lct(z_{ij}) = 1 - \mult_{z_{ij}}B_{Z_i}'$$
and hence
\begin{equation}\label{eq: lct limit}
    \lim_i \lct(z_{ij})  = 1 - \lim_i \mult_{z_{ij}}B_{Z_i}' = 1- b_j' \geq b_j-b_j' = \delta_j.
\end{equation}
%For fixed $1\leq j\leq n$, there are two possibilities, depending on whether $b_j =1$ or $<1$. 

%If $b_j=1$, then the inequality in \eqref{eq: lct limit} is an equality, and we have $\delta_j = \lim_i \lct(z_{ij})$.

Suppose that $b_j<1$. Then, by \eqref{eq: coeff B_Z_i}, $\{m_{ij}\}_i$ is bounded, and $\{\mult_{z_{ij}}B_{Z_i}'\}_i$ is a discrete subset of $[0,1]$ away from $1$.  By passing to a subsequence of $\{(X_i, B_i)\}_i$, we may assume that
\[
m_{ij} = m_j, \quad \mult_{z_{ij}} B_{Z_i}' = b_j',\quad
\]
where $m_j\in \ZZ_{>0}$ and $b_j'\in \tGamma_0$ are independent of $i$. In this case, we set
\begin{equation}
\gamma_{j}:= m_j\delta_j = \lim_{i} \gamma_{ij}.
\end{equation}
Since $\Gamma$ is closed in $[0,1]$ and $\gamma_{ij}\in \Gamma$, we have $\gamma_{j}\in \Gamma$. Moreover, we have 
\begin{equation}\label{eq: lct constant}
    \lct(z_{ij}) = 1 - \mult_{z_{ij}}B_{Z_i}' = 1-b_j' = \lct(z_{1j}),
\end{equation} 
which is independent of $i$.

% Let $\lct(z_{ij})$ be the log canonical threshold of $f_i^*z_{ij}$ with respect to $(X_i, B_i^h)$ for any $i,\, j$. By the definition of discriminant part $B_{Z_i}'$, we know that $$\lct(z_{ij}) = 1 - \mult_{z_{ij}}B_{Z_i}' = 1-\gamma_j' = \lct(z_{1j}),$$
% which is independent of $i$. On the other hand, by the definition of $B_{Z_i}$, we know that 
% \[
% \frac{\gamma_{ij}}{m_{j}}=\mult_{z_{ij}}B_{Z_i} - \mult_{z_{ij}}B_{Z_i}' = \mult_{z_{ij}}B_{Z_i} - \gamma_j'\leq  1-\gamma_j'.
% \]
% Taking the limit, we have
% \begin{equation}\label{eq: limit BZ-BZ'}
%    \frac{\gamma_j}{m_j} = \lim_i \frac{\gamma_{j}}{m_{ij}} \leq 1-\gamma_j' = \lct(z_{1j})
% \end{equation}
% where the inequality is an equality if and only if $\lim_i \mult_{z_{ij}} B_{Z_i} =1$.

Now define a projective log canonical pair $(X_\infty, B_\infty)$ of dimension $d$ as follows: The projective variety $X_\infty$ is taken to be $X_1$, and the boundary $\RR$-divisor $B_\infty=B_1^h + B_\infty^v$, where $B_{\infty}^v=\sum_{1\leq j\leq n}c_j f_1^* z_{1j},$ and the coefficients are specified as follows
$$ c_j = \begin{cases}
    \delta_j = \frac{\gamma_j}{m_j} & \text{if $b_j<1$} \\
    \lct(z_{1j}) & \text{if $b_j=1$.}
\end{cases}  $$  
By \eqref{eq: lct limit} and \eqref{eq: lct constant}, the pair $(X_\infty, B_\infty)$ has log canonical singularities.
% with the coefficients $b_j$ specified as follows:
% \[
% b_j = 
% \begin{cases}
% \frac{\gamma_j}{m_j} & \text{if $\lim_i \mult_{z_{ij}} B_{Z_i} <1$,}\\
% \lct(z_{1j}) &  \text{if $\lim_i \mult_{z_{ij}} B_{Z_i} =1$.}
% \end{cases}
% \]
% \noindent\textit {Proof of Claim~\ref{claim: lc}.} If $\lim_i \mult_{z_{ij}} B_{Z_i} =1$, then $b_j=\lct(z_{1j})$ by definition. Now assume that $\lim_i \mult_{z_{ij}} B_{Z_i} <1$. By the definition of discriminant part $B_{Z_i}'$, we know that $$\lct(z_{ij}) = 1 - \mult_{z_{ij}}B_{Z_i}' = 1-\gamma_j',$$
% which is independent of $i$. On the other hand, by the definition of $B_{Z_i}$, we know that 
% \[
% \mult_{z_{ij}}B_{Z_i} - \mult_{z_{ij}}B_{Z_i}' = \frac{\gamma_{ij}}{m_{ij}} \leq \lct(z_{ij})= 1-\gamma_j'.
% \]
% It follows that
% \[
% \frac{\gamma_j}{m_j}=\lim_{i}\frac{\gamma_{ij}}{m_{j}}\leq 1-\gamma_j' = \lct(z_{1j}).\qed
% \]
Moreover, we have $$K_{X_\infty}+B_{\infty}\sim_{\RR, Z_1} K_{X_1}+B_1^h \sim_{\RR, Z_1} 0.$$ Let
\[
J_0 = \{j\mid 1\leq j\leq n,\, b_j<1\},\quad J_1 = \{m\mid 1\leq j\leq n,\, b_j=1\}.
\]
Then
\begin{align*}
    \vol_1(K_{X_\infty}+B_{\infty})  &= 2g-2+ \lambda + \# J_1 + \sum_{j\in J_0} (b_j'+\delta_j)  \\
    &=2g-2+\lambda +\sum_{1\leq j\leq n} \lim_i b_{ij} \\&= \lim_i \vol_1(K_{X_i}+B_i) = v.
\end{align*}

By Lemma~\ref{lem: model right coefficient}, we may construct a crepant birational model $(X, B)$ of $(X_\infty, B_\infty)$ over $Z_1$, such that the horizontal$/Z_1$ part $B^h$ is the strict transform of $B_\infty$, and the vertical over$/Z_1$ part $B^v$ satisfies
\[
(B^v)_z = 
\begin{cases}
    \gamma_j E_j & \text{if $z=z_{1j}$ for some $1\leq j\leq n$ and $\frac{\gamma_j}{m_j}<\lct(z_{1j})$,} \\
    0 & \text{otherwise,}
\end{cases}
\]
where $E_j$ is the strict transform of one component of $f_1^*z_{1j}$ with multiplicity $m_j$. In particular, we have $B\in \Gamma$, and hence $(X, B)\in \gP^{\Gamma}_\lc(d,1)$. Its Iitaka volume is  
\[
\vol_1(K_X+B) =\vol_1(K_{X_\infty}+B_{\infty}) = v,
\]
and hence $v\in \Ivol^{\Gamma}_\lc(d,1)$.
\end{proof}
%In case $d=2$, the log canonical surface $(X, B)$ as above may be taken from $\gP^\Gamma_\lc(2,1)$ by Lemma~\ref{lem: surface model 0 coefficient}, and correspondingly, $v\in \Ivol^\Gamma_\lc(2,1)$, without assuming that $1\in \Gamma$.

\begin{prop}
For $d>\kappa>0$ integers, $\Ivol_\sm(d,\kappa)$ has infinite accumulation complexity. As a consequence, $\Ivol_\lc^{\Gamma}(d,\kappa)$ has infinite accumulation complexity for any subset $\Gamma\subset [0,1]$.
\end{prop}
\begin{proof}
For positive integers, $d>\kappa >0$, take two smooth projective curves $C$ and $E$ with $g(C)>1=g(E)$, and consider the set of smooth projective varieties
\[
C^{\kappa-1}\times E^{d-\kappa-1}\times \gP_{\sm}(2,1):=\left\{C^{\kappa-1}\times E^{d-\kappa-1} \times S \mid S\in \gP_{\sm}(2,1)\right\}.
\]
Then $C^{\kappa-1}\times E^{d-\kappa-1}\times \gP_{\sm}(2,1)\subset \gP_\sm(d,\kappa)$, and hence the inclusion
\[
V:=\left\{\vol_\kappa(K_X) \mid X = C^{\kappa-1}\times E^{d-\kappa-1}\times S \mid S\in \gP_{\sm}(2,1)\right\} \subset \Ivol_\sm(d,\kappa).
\]
By a straightforward computation, one sees that
\[
V=\left\{\kappa!\cdot (2g(C)-2)^{\kappa-1}\cdot v\mid v\in \Ivol_\sm(2,1)\right\}.
\]
Since $\Ivol_\sm(2,1)$ has infinite accumulation complexity by Corollary~\ref{cor: sm acc points}, the set $V$ also has infinite accumulation complexity. A fortiori, as supersets of $V$, both $\Ivol_\sm(d,\kappa)$ and $\Ivol_\lc^\Gamma(d,\kappa)$ have infinite accumulation complexity, where $\Gamma$ can be any subset of $[0,1]$.
\end{proof}

%\subsection{Proof of Theorem \ref{thm:main2}}
The following lemma is a generalization of \cite[Theorem 0.2]{Amb05}, from coefficients in $\Qq$ to coefficients in $\Rr$.
\begin{lem}\label{lem:ambrdiv}
Assume that $(X,B)$ is a klt pair and $f\colon X\to Z$ is a contraction such that $\dim Z>0$ and $K_X+B\sim_{\Rr,Z}0$. Then there exists an $\Rr$-divisor $B_Z$ on $Z$ such that $(Z,B_Z)$ is klt and $K_X+B\sim_{\Rr}f^*(K_Z+B_Z)$.
\end{lem}
\begin{proof}
According to \cite[Lemma 5.4 and Corollary 5.5]{HLS19}, we may find positive real numbers $a_1,\dots,a_l\in(0,1]$ and effective $\Qq$-divisors $B_i$ on $X$ such that
\begin{itemize}
    \item $\sum_{i=1}^l a_i=1$,
    \item $\sum_{i=1}^l a_i(K_X+B_i)=K_X+B$, and
    \item $(X,B_i)$ is klt and $K_X+B_i\sim_{\Qq,Z}0$ for any $1\le i\le l$.
\end{itemize}
By \cite[Theorem 0.2]{Amb05}, for each $1\le i\le l$, there exists a $\Qq$-divisor $B_{Z,i}$ on $Z$ such that $(Z,B_{Z,i})$ is klt and 
$$K_X+B_i\sim_{\Qq}f^*(K_Z+B_{Z,i}).$$ 
Set $B_Z:=\sum_{i=1}^la_iB_{Z,i}$. It is clear that $(Z,B_Z)$ has the required properties.
\end{proof}

%As one of the main ingredients in the proof of Theorem \ref{thm:main2}, we show that any klt pair with large Iitaka dimension has a good minimal model which is well-known to experts (cf.~\cite{GW22,HS21,Fil23}).

\begin{prop}[cf.~{\cite{GW22,HS21,Fil23}}]\label{thm:gmm}
Let $(X, B)$ be a projective klt pair of dimension $d\ge3$ with $\kappa(K_X+B) \geq d-3$. Then $(X, B)$ has a good minimal model.
\end{prop}

\begin{proof}
We may assume that $d\ge 4$ and $\kappa(K_X+B)>0$. Suppose that $X_\infty\to Z$ is an Iitaka fibration of $K_X+B$, where $X_\infty$ and $Z$ are smooth varieties (cf. \cite[Definition 3.19]{Li22}). We may pick a sufficiently small positive real number $\epsilon$ such that
$$K_{X_\infty}+h_*^{-1}B+(1-\epsilon){\Exc}(h)\ge h^*(K_X+B)$$
and $\left(X_\infty,h_*^{-1}B+(1-\epsilon){\Exc}(h)\right)$ is klt, where $h$ denotes the morphism $X_\infty\to X$ and ${\Exc}(h)$ is the sum of all the $h$-exceptional prime divisors. Let $B_{X_\infty}:=h_*^{-1}B+(1-\epsilon){\Exc}(h)$. Note that $$\kappa(K_{X_\infty}+B_{X_\infty})=\kappa(K_X+B)=\dim Z.$$
Let $F$ be a very general fiber of $X_\infty\to Z$, then $\dim F\le3$ by assumption. Thus $(F,(B_{X_\infty})|_F)$ has a good minimal model and $K_F+(B_{X_\infty})|_F=(K_{X_\infty}+B_{X_\infty})|_F$ is abundant which implies that the $\Rr$-divisor $K_{X_\infty}+B_{X_\infty}$ is abundant over $Z$ (see \cite[Definition 2.9]{HH20} for the definition of abundant $\Rr$-divisors). Then by \cite[Lemma 2.13]{HH20} we know that $(X_\infty,B_{X_\infty})$ has a good minimal model over $Z$ which we may denote by $Y$. 

Let $B_Y$ be the strict transform of $B_{X_\infty}$ on $Y$, then $(Y,B_Y)$ is klt and $K_Y+B_Y$ is semi-ample over $Z$. Suppose that $\psi\colon Y\to V$ is the ample model of $K_Y+B_Y$ over $Z$. Remark that since $\kappa(X_\infty/Z,K_{X_\infty}+B_{X_\infty})=0$, $\dim V=\dim Z$. According to Lemma \ref{lem:ambrdiv}, one can find an $\Rr$-divisor $B_V$ on $V$ such that $(V,B_V)$ is klt and 
$$K_Y+B_Y\sim_\Rr \psi^*(K_V+B_V).$$ 
By \cite[II Lemma 3.11]{Nak04} and \cite[Proposition 2.2.2]{Cho08}, $\kappa_\iota(K_V+B_V)=\kappa_\iota(K_Y+B_Y)=\kappa(K_X+B)$ (see Definition \ref{defn:inviiitdim}).
Thus $\dim V=\kappa_\iota(K_Y+B_Y).$ In particular, we have that $\dim V=\kappa_\iota(K_V+B_V)$, i.e., $K_V+B_V$ is a big $\Rr$-divisor and therefore is abundant \cite{BCHM10}. It immediately implies that $K_Y+B_Y$ is abundant and thus so is $K_{X}+B$. By \cite[Lemma 2.13]{HH20} again, $(X,B)$ has a good minimal model. 
\end{proof}

\begin{proof}[Proof of Theorem \ref{thm:main2}]
If $(X,B)$ is a projective klt pair of dimension $d$ such that $\kappa(K_X+B)\ge d-2$ and $B\in\Ii$, then by Proposition \ref{thm:gmm}, $(X,B)$ has a good minimal model $(X',B')$. Therefore we may assume that $(X,B)$ is lc and $K_X+B$ is semi-ample. Then the proof is the same as that of Lemma \ref{lem:ivolaccfollowfrom2conj}, replacing \cite[Proposition 3.1]{CHL24} with Lemma~\ref{lem: cbfindex}.
\end{proof}

\subsection{DCC property for invariant Iitaka volumes}\label{sec:invivol}

In this subsection, we briefly discuss the invariant Iitaka volumes.

\begin{defn}\label{defn:inviiitdim}
Let $X$ be a normal variety and $D$ an $\Rr$-divisor on $X$. 

(1) ({\cite[Definition 2.2.1]{Cho08}}) The \emph{invariant Iitaka dimension} $\kappa_{\iota}(D)$ of $D$ is defined as follows. If $|D|_{\Rr}\not=\emptyset$, then we define
$$\kappa_{\iota}(D):=\kappa\left(D'\right)$$
for some $\Rr$-divisor $D'\in|D|_\Rr$. Otherwise, let $\kappa_{\iota}(D):=-\infty$. By \cite[Corollary 2.1.4]{Cho08}, $\kappa_{\iota}(D)$ is independent of the choice of $D'$.

(2) In case $\kappa:=\kappa_\iota(D)\geq 0$, the \emph{invariant Iitaka volume}
\[
\vol^{\iota}_\kappa(D):=\limsup_{m \rightarrow \infty} \frac{ h^0\left(X, \mathcal{O}_X(m D')\right)}{m^{\kappa}/\kappa!}
\]
for some $\Rr$-divisor $D'\in|D|_\Rr$, which is independent of the choice of $D'$.
\end{defn}
The following example shows that the invariant Iitaka dimensions and thus the invariant Iitaka volumes behave differently.

\begin{ex}\label{ex: why invairant iitaka dimension}
Let $X:=\Pp^1$ and $p_1,p_2,p_3,p_4$ four different closed points on $X$. For any real number $a$, we let $B(a):=a(p_1+p_2)+(1-a)(p_3+p_4)$. Then whenever $a\in (0,1)$, $(X,B(a))$ is klt and $K_X+B(a)\sim_{\Rr}0$. However, for any integer $m$,
$$\deg \lfloor m(K_X+B(a))\rfloor=-2,\text{ if }ma\not\in\Zz,$$
and
$$\deg \lfloor m(K_X+B(a))\rfloor=0,\text{ if }ma\in\Zz.$$
This implies that
\begin{enumerate}
    \item $\kappa(K_X+B(a))=-\infty$ when $a\not\in\Qq$ and $\kappa(K_X+B(a))=0$ when $a\in\Qq$. Nevertheless, $\kappa_{\iota}(K_X+B(a))=0$ for any $a\in\Rr$.
    \item The Iitaka volume of $K_X+B(a)$ is undefined when $a\notin\Qq$, while the invariant Iitaka volume $\vol_{0}^{\iota}(K_X+B(a))=0$ for any $a\in\Rr$.
\end{enumerate}
\end{ex}

However, we still expect the DCC property for the invariant Iitaka volumes:

\begin{conj}[DCC of invariant Iitaka volumes]\label{conj:invivoldcc}
Let $\kappa <d$ be positive integers and $\Gamma \subset[0,1]$ a DCC set. Then the set of invariant Iitaka volumes:
\[
\left\{\vol^{\iota}_\kappa(K_X+B)\mid (X,B)\text{ is projective lc, $\dim X=d,\ B\in\Gamma,\ \kappa_\iota(K_X+B)=\kappa$}\right\}
\]
is a DCC set.
\end{conj}

The conjecture holds true in low dimensions.

\begin{thm}\label{thm:invivoldcc3}
Conjecture \ref{conj:invivoldcc} holds when $d\le 3$.
\end{thm}
The proof is quite similar to that of Theorem \ref{thm: main1}. Therefore, we will provide only a brief outline below. The key point is that we need to choose a moduli part in the canonical bundle formula with controlled coefficients, which then allows us to apply Birkar's result (\cite[Theorem 1.3]{Bir21}).
\begin{proof}
Let $(X,B)$ be a projective lc pair of dimension $\le3$ with $B\in\Gamma$ and $\kappa_\iota(K_X+B)=\kappa$. We may assume that $X$ is $\Qq$-factorial and $K_X+B$ is semi-ample. Let $f:X\to Z$ be the ample model of $K_X+B$. Replacing $(X,B)$ with a good minimal model of $(X,B^h)$ over $Z$, we may assume that $K_X+B^h\sim_{\Rr,Z}0$, where $B^h$ is the horizontal$/Z$ part of $B$. Note that $B^h\in\Gamma_0$ for some finite set $\Gamma_0\subset\Gamma$ which only depends on $\Gamma$. According to \cite[Lemma 5.4 and Corollary 5.5]{HLS19}, we may find positive real numbers $a_1,\dots,a_l\in(0,1]$ and a finite set $\Gamma_1\subset[0,1]\cap\Qq$ that only depends on $\Gamma_0$, such that there exist effective $\Qq$-divisors $B_1,\dots,B_l\in\Gamma_1$ on $X$ such that
\begin{itemize}
    \item $\sum_{i=1}^l a_i=1$,
    \item $\sum_{i=1}^l a_i(K_X+B_i)=K_X+B^h$, and
    \item $(X,B_i)$ is lc and $K_X+B_i\sim_{\Qq,Z}0$ for any $1\le i\le l$.
\end{itemize}
By \cite[Proposition 3.1]{CHL24}, there is an integer $p$ depending only on $\Gamma_1$ such that we can choose a moduli part $\bM_i$ of the canonical bundle formula for $(X,B_i)$ over $Z$ such that $p\bM_i$ is b-Cartier for each $1\le i\le l$. By \cite[Theorem 3.3]{HLX23} we infer that $\bM:=\sum_{i=1}^l a_i\bM_i$ is a moduli part of the canonical bundle formula for $(X,B^h)$ over $Z$, as well as a moduli part of the canonical bundle formula for $(X,B)$ over $Z$. Then the theorem follows from \cite[Theorem 1.3]{Bir21}.
\end{proof}

\begin{rem}
Theorem \ref{thm: main3} for invariant Iitaka volumes also holds true by the same arguments. However, Theorem \ref{thm:main2} is not clear for invariant Iitaka volumes since we do not know the existence of good minimal model for klt pairs with $\kappa_\iota(K_X+B)\ge\dim X-3$.
\end{rem}

%%%%%%%%%%%%%%
\section{Explicit description of the Iitaka volumes for several classes of surfaces}
For surfaces, one hopes to have a more explicit description of their Iitaka volumes. In particular, for a naturally appearing class of log canonical surfaces with Iitaka--Kodaira dimension 1, it would be interesting to know the minimum as well as the minimal accumulation point of the set of their Iitaka volumes.

We start with smooth properly elliptic surfaces in Section~\ref{sec: sm surf}. The set $\Ivol_\sm(2,1)$ of Iitaka volumes of this class of surfaces is completely determined in Theorem~\ref{thm: sm surf Ivol}. Based on this, we obtain the minimum, as well as information on the accumulation points of  of $\Ivol_\sm(2,1)$ in Theorem~\ref{thm: min Ivol sm} and Corollary~\ref{cor: sm acc points}. Then, similar results for Iitaka volumes of elliptic surfaces with prescribed geometric genus are proved in Theorem~\ref{thm: Ivol sm prescribe pg}, Corollaries~\ref{cor: min Ivol sm prescribe pg} and \ref{cor: Ivol ccc sm prescribe pg}.

For singular lc surfaces,  we give in Section~\ref{sec: lc surf} a recipe to use a crepant smooth model to compute the Iitaka volume set $\Ivol_\lc^{\Gamma}(2,1)$ for a DCC subset $\Gamma\subset[0,1]$. This is explicitly carried out for the coefficient sets $\Gamma=\{0\}$ and $\{0,1\}$. The main discovery are Theorems~\ref{thm: min lc Ivol 0} and \ref{thm: min lc Ivol 1}, giving the minima and minimal accumulation points of $\Ivol_\lc^{\{0\}}(2,1)$ and $\Ivol_\lc^{\{0,1\}}(2,1)$ respectively.

\subsection{Iitaka volumes of smooth properly elliptic surfaces}\label{sec: sm surf}
Let $X$ be a smooth projective surface with $\kappa(X)=1$, that is, $X\in \gP_\sm(2,1)$.  Such a surface is called a \emph{properly elliptic surface}. There is a (unique) elliptic fibration $f\colon X\rightarrow Z$ defined by the pluricanonical systems of $X$.

For the purpose of birational classification, we may well assume that $X$ is minimal, that is, $K_X$ is nef. We use holomorphic Euler characteristic $\chi(\sO_X)$ of $X$ and the genus $g(Z)$ of $Z$ to divide $\gP_\sm(2,1)$ into subsets:
\[
\sE_{\chi, g} = \{X\in \gP_\sm(2,1)\mid \chi(\sO_X)=\chi,\, g(Z)=g, \text{ $K_X$ is nef}\}
\]
and its corresponding Iitaka volume set is then
\[
\Ivol(\sE_{\chi, g}) = \{\vol_1(K_X)\mid  X\in \sE_{\chi, g}\}.
\]

For an elliptic surface $f\colon X\rightarrow Z$, it is well known that one of following cases occurs:
\begin{enumerate}
    \item $\chi(\sO_X)>0$ and $q(X)=g(Z)$. In this case, there exist singular fibers whose reduction is not smooth.
    \item $\chi(\sO_X)=0$, and $q(X)=g(Z)$ or $g(Z)+1$. In this case, all singular fibers have smooth reduction, and $f$ is a so-called quasi-bundle (\cite{Ser91}).
\end{enumerate}
\begin{lem}\label{lem: elliptic surf exists}
For each pair of non-negative integers $\chi$ and $g$, there is a relatively minimal elliptic surface $f\colon X\rightarrow Z$, not necessarily of Kodaira dimension 1, with a section such that $\chi(\sO_X)=\chi$ and $g(Z)=g$.
\end{lem}
\begin{proof}
If $\chi=0$, then we may take the projection $f\colon X=E\times Z\rightarrow Z$, where $E$ and $Z$ are smooth projective curves of genus $1$ and $g$ respectively.

Next we consider the case $\chi>0$. For $(\chi, q)=(1,0)$, we may take any rational elliptic surface $f_{(1,0)}\colon X_{(1,0)}\rightarrow \PP^1$ with a section. Moreover, we may assume that $f_0$ has exactly two singular fibers $f^*0$ and $f^*\infty$, $i=1,2$, which are of type $\I_0^*$ (\cite[VIII.1.4]{Mir89}). 

Now take a double cover $\pi\colon Z\rightarrow \PP^1$ branched at $2g+2$ points $t_1=0, t_2, \dots, t_{2g+2}\in \PP^1\setminus\{\infty\}$, and let $f_{(1,g)}\colon X_{(1,g)}\rightarrow Z$ be the relatively minimal elliptic fibration obtained by resolving the singularities of the fiber product $X_{(1,0)}\times_{\PP^1} Z$. Then $f_{(1,g)}$ has a section and exactly two singular fibers, which lie over the two points $\pi^{-1}(\infty)$ and are of type $\I_0^*$. By the Riemann--Hurwitz formula, we have $g(Z)=g$. It follows from the Noether formula that $\chi\left(\sO_{X_{(1,q)}}\right)=\frac{1}{12}e(X_{(1,q)})=1$, where for a projective variety $V$, $e(V)$ denotes its topological Euler characteristics.

For $(\chi, q)$ with $\chi\geq 2$, we may transfer $2(\chi-1)$ smooth fibers of $f_{(1,q)}$ to singular fibers of type $\I_0^*$, and obtain a relatively minimal elliptic fibration $f_{(\chi,q)}\colon X_{(\chi,q)}\rightarrow Z$ with a section such that $\chi(\sO_{X_{(\chi,q)}})=\chi$ and $g(Z)=g$ (\cite[V.4]{Mir89}).
\end{proof}

\begin{prop}\label{prop: sm Ivol1}
For given $\chi\in \ZZ_{>0}$ and $g\in\ZZ_{\geq 0}$, we have
\begin{equation}\label{eq: Ivol E_{chi g}}
\Ivol(\sE_{\chi, g})=
\left\{2g-2+\chi+\sum_{1\leq i\leq r} \left(1-\frac{1}{m_i}\right)\, \bigg|\, r\in \ZZ_{\geq 0},\, m_i\in\ZZ_{\geq 2} \right\}_{>0}.
\end{equation}
\end{prop}
\begin{proof}
Let $V_{\chi,g}$ denote the right hand side of \eqref{eq: Ivol E_{chi g}}.

For $X\in \sE_{\chi, g}$, let $f\colon X\rightarrow Z$ be the Iitaka fibration. Then $\chi(\sO_X)=\chi$, $g(Z)=g$, and by the canonical bundle formula,
    \[
    \vol_1(K_X) = 2g-2+\chi + \sum_{1\leq i\leq r} \left(1-\frac{1}{m_i}\right)>0
    \]
    where $m_i F_i$ ($1\leq i\leq r$) are the multiple fibers of $f$. Thus $\vol_1(K_X)\in V_{\chi, g}$.
    
Conversely, for any $v=2g-2+\chi + \sum_{1\leq i\leq r} \left(1-\frac{1}{m_i}\right)\in V_{\chi,g}$, we take a relatively minimal elliptic surface $f'\colon X'\rightarrow Z$ with a section such that $\chi(\sO_{X})=\chi$ and $g(Z)=g$ by Lemma~\ref{lem: elliptic surf exists}. We may perform logarithmic transformation along $r$ smooth fibers $f'^*z_i$ to obtain a new \emph{algebraic} elliptic surface $f\colon X\rightarrow Z$ such that $f^*z_i$ are fibers of type $m_i\I_0$ for $1\leq i\leq r$ (see \cite[Chapter I, Theorem~6.12]{FriMor94}). Then $X\in \sE_{\chi, g}$, and $\vol_1(K_{X}) = v$.  
\end{proof}

\begin{cor}
For $\chi\geq \chi'>0$ and $g\geq g'\geq 0$, we have $\Ivol(\sE_{\chi,g})\subset \Ivol(\sE_{\chi',g'})$.
\end{cor}
\begin{proof}
By Proposition~\ref{prop: sm Ivol1}, any element of
$\Ivol(\sE_{\chi,g})$ is a positive rational number $v$ of the form $2g-2+\chi + \sum_{1\leq i\leq r} \left(1-\frac{1}{m_i}\right)$. We may rewrite it as 
\[
v=2g'-2+\chi' + \sum_{1\leq i\leq r} \left(1-\frac{1}{m_i}\right)+\frac{1}{2}s,
\]
where $s =(2g-2+\chi)-(2g'-2+\chi')$. By Proposition~\ref{prop: sm Ivol1} again, we have $v\in \Ivol(\sE_{\chi',g'})$.
\end{proof}

\begin{cor}\label{cor: Ivol E10}
We have 
\[
\bigcup_{\chi>0, g\geq 0}\Ivol(\sE_{\chi,g}) =\Ivol(\sE_{1,0})=\left\{-1+\sum_{1\leq i\leq r}a_i \,\bigg|\, r\in \Zz_{>0}, a_i\in \Gamma\right\}_{>0},
\]
where $\Gamma= \{1-\frac{1}{m}\mid m\in \Zz_{>0}\}$.
\end{cor}

% \begin{cor}\label{cor: min Ivol E10}
% The following properties about $\Ivol_\sm(2,1)$ hold.
%   \begin{enumerate}
%       \item  $\min\sE_{1,0} = ?$.
%       \item The minimal accumulation point of $\min\sE_{1,0} $ is $?$.
%       \item The accumulation complexity of $\sE_{1,0}$ is $\infty$?
%       \item For a given $M>0$, the subset $\left(\sE_{1,0}\right)_{\leq M}:=\{v\in\sE_{1,0}\mid v\leq M\}$ has finite accumulation complexity?
%       \item $\min\sE_{1,0}$ is closed in $\RR$?
%   \end{enumerate} 
% \end{cor}

For ellptic surfaces $f\colon X\rightarrow Z$ with $\chi(\sO_X)=0$, due to restriction on the monodromy around the mutlitple fibers, there may not exist any logarithmic transformation resulting in an algebraic elliptic surface with arbitrarily prescribed constellation of multiple fibers. In order to understand the possible constellation of multiple fibers in this case, one needs to take base change to present $X$ as the \'etale quotient of a product surface.

Let $f\colon X\rightarrow Z$ be a relatively minimal elliptic fibration with $\chi(\sO_X)=0$. Then there is a Galois base change $\tZ\rightarrow Z$ such that the normalization of $X\times_Z \tZ$ is $\tZ\times E$, where $E$ is a smooth elliptic curve. Let $G$ be the Galois group of $\tZ\rightarrow Z$. Then $G$ acts on $E$ and $X = (\tZ\times E)/G$, where $G$ acts on $\tZ\times E$ diagonally. We have the following commutative diagram
\[
\begin{tikzcd}
\tZ\times E \arrow[r]\arrow[d, "\pr_1"'] & X=(\tZ\times E)/G\arrow[d, "f"] \\
\tZ \arrow[r, "\pi"]& Z = \tZ/G
\end{tikzcd}
\]
where the horizontal arrows denote the quotient maps, and $\pr_1$ is the projection onto the first factor.

Let $Q_1, \dots, Q_r\in Z$ be the branch points of the quotient map $\tZ\rightarrow Z$, and let $\sigma_i$ be a generator of the stabilizer $G_{P_i}$ for some $P_i\in \pi^{-1}(Q_i)$. Then by the canonical bundle formula, we have
\begin{equation}\label{eq: vol1 chi 0}
\vol_1(K_X) = 2g(Z)-2+\sum_{1\leq i\leq r}\left(1-\frac{1}{m_i}\right)
\end{equation}
where $m_i$ is the order of $\sigma_i$ for $1\leq i\leq r$. 

\begin{prop}\label{prop: vol sm g>0}
 For $g\geq 1$, we have $\Ivol(\sE_{0,g})\subset \Ivol(\sE_{1, 0})$.
\end{prop}
\begin{proof}
For $X\in \sE_{0,g}$, we have $\vol_1(K_X)\in \Ivol(\sE_{1,0})$ by \eqref{eq: vol1 chi 0} and Proposition~\ref{prop: sm Ivol1}. 
\end{proof}

In the following we assume that $\chi(\sO_X)=g(Z)=0$, so 
\[
\vol_1(K_X) = -2+\sum_{1\leq i\leq r}\left(1-\frac{1}{m_i}\right).
\]
Note that $q(X) = g(Z) + g(E/G) = g(E/G)\leq 1$. On the other hand, $q(X)= p_g(X)+\chi(\sO_X)=1+p_g(X)\geq 1$. Hence we have
\[
q(X)  = g(E/G) = 1.
\]
It follows that $G$ acts on $E$ by translations of torsion elements, so $G$ is abelian and is generated by at most two elements. Hence the stabilizers $G_{P_i}$ do not depend on the choice of $P_i\in \pi^{-1}(Q_i)$, and we may choose the generators $\sigma_i$ of $G_{P_i}$ for $1\leq i\leq r$, so that $\sum_i \sigma_i =0\in G$. 

% Under the identification $E[n]\cong (\ZZ/n\ZZ)^2$, we may write $\sigma_i=(a_i, b_i)\in (\ZZ/n\ZZ)^2$ with $a_i, b_i\in \ZZ/n\ZZ$. Then the above restrictions on the $\sigma_i$ can be reformulated as follows:
% \begin{equation}\label{eq: sigma}
% \begin{split}
%   &  \sum_{1\leq i\leq r} a_i =\sum_{1\leq i\leq r} b_i =0  \in \ZZ/n\ZZ, \\
%   &  \lcm\{|a_i|, |b_i| \mid 1\leq i\leq r\}=n
% \end{split}
% \end{equation}
% and since $m_i = |\sigma_i| = \lcm\{|a_i|, |b_i|\}$ for each $1\leq i\leq r$, we can write
% \[
% \Ivol(K_X) = -2 +\sum_{1\leq i\leq r}\left(1-\frac{1}{m_i}\right) = -2 +\sum_{1\leq i\leq r} \left(1- \frac{1}{\lcm\{|a_i|, |b_i|\}}\right),
% \]
% where $(a_i, b_i)\in (\ZZ/n\ZZ)^2$ are subject to the conditions specified in \eqref{eq: sigma}.

\begin{prop}\label{prop: Ivol E00}
We have
    \[
\Ivol(\sE_{0,0})=\left\{-2 +\sum_{1\leq i\leq r} \left(1- \frac{1}{|\sigma_i|}\right)\,\Bigg|\, 
   \sigma_i\in E \text{ are torsions such that $\sum_{1
\leq i\leq r}\sigma_i=0$}
\right\}_{>0}.
\]
\end{prop}
\begin{proof}
Denote the right hand side of the equality by $V_{0,0}$.
By the above discussions, we know that $\Ivol(\sE_{0,0})\subset V_{0,0}$. 

Conversely, for $v= -2 +\sum_{1\leq i\leq r} \left(1- \frac{1}{|\sigma_i|}\right)\in V_{0,0}$, where the $\sigma_i\in E$ are torsion elements such that $\sum_{1
\leq i\leq r}\sigma_i=0$, let $G=\langle\sigma_1,\dots, \sigma_r\rangle\subset E$ be the subgroup of $E$ generated by the $\sigma_i$'s. By the Riemann existence theorem, there is a ramified $G$-cover $\pi\colon \tZ \rightarrow \PP^1$ which has exactly $r$ branch points, say $Q_1,\dots, Q_r$, and the stabilizer over $Q_i$ is exactly $\langle \sigma_i\rangle$. Then the diagonal action of $G$ on the product $\tZ\times E$ is free, and the quotient surface $X=(\tZ\times E)/G$ is smooth with Kodaira dimension 1, with Iitaka fibration $X\rightarrow \tZ/G\cong \PP^1$, $\chi(\sO_X)=0$, and $\vol_1(K_X)=v$. Therefore, $v\in \Ivol(\sE_{0,0})$.
\end{proof}

% \begin{cor}\label{cor: min Ivol E00}
% The following properties about $\sE_{0,0}$ hold.
%   \begin{enumerate}
%       \item  $\min\sE_{0,0} = ?$.
%       \item The minimal accumulation point of $\min\sE_{0,0} $ is $?$.
%       \item The accumulation complexity of $\sE_{0,0}$ is $\infty$?
%       \item For a given $M>0$, the subset $\left(\sE_{0,0}\right)_{\leq M}:=\{v\in\sE_{0,0}\mid v\leq M\}$ has finite accumulation complexity?
%       \item $\min\sE_{0,0}$ is closed in $\RR$?
%   \end{enumerate} 
% \end{cor}

\begin{thm}\label{thm: sm surf Ivol}
 Using Notation~\ref{nota: Ivol}, we have $$\Ivol_\sm(2,1)= \Ivol(\sE_{1,0})\cup \Ivol(\sE_{0,0}),$$ 
where $\Ivol(\sE_{1,0})$ and $\Ivol(\sE_{0,0})$ are specified in Propositions~\ref{prop: sm Ivol1} and \ref{prop: Ivol E00}.
\end{thm}
\begin{proof}
By definition, we have $\Ivol_\sm(2,1)=\cup_{\chi, g} \Ivol(\sE_{\chi,g})$. The theorem follows from Corollary~\ref{cor: Ivol E10} and Proposition~\ref{prop: vol sm g>0}. The description of $\Ivol(\sE_{1,0})$ and $\Ivol(\sE_{0,0})$ are given in Propositions~\ref{prop: sm Ivol1} and \ref{prop: Ivol E00} respectively.
\end{proof}

We will study the geometry of $\Ivol_\sm(2,1)$, using the description of $\Ivol_\sm(2,1)$ given in Theorem~\ref{thm: sm surf Ivol}. First of all, we have
\begin{thm}\label{thm: min Ivol sm}
$\min\Ivol_\sm(2,1) = \frac{1}{6}$.
\end{thm}
\begin{proof}
By Theorem~\ref{thm: sm surf Ivol}, we have
\begin{equation}\label{eq: min Ivol sm}
    \min\Ivol_\sm(2,1) = \min\left\{\min\Ivol(\sE_{1,0}), \min\Ivol(\sE_{0,0})\right\}. 
\end{equation}
Now one verifies readily that 
\[
\min\Ivol(\sE_{1,0}) =-1+\frac{1}{2}+\frac{2}{3}  =\frac{1}{6}.
\]
Thus it suffices to show that $\min\Ivol(\sE_{0,0}) \geq \frac{1}{6}$. Note that any $v\in \Ivol(\sE_{0,0})$ is a positive rational number of the form
\[
v=-2+ \sum_{1\leq i\leq r}\left(1-\frac{1}{m_i}\right)
\]
where the $m_i=|\sigma_i|$ are the orders of some torsion elements $\sigma_i$ of an elliptic curve satisfying $\sum_{1\leq i\leq r}\sigma_i=0$. We may assume that $m_1\leq m_2\leq \cdots \leq m_r$. Since $v>0$ and $1-\frac{1}{m_i}<1$ for each $i$, we have $r\geq 3$.

If $r\geq 5$, then 
\[
v\geq -2 + \frac{1}{2}\cdot r \geq -2 + \frac{5}{2} = \frac{1}{2}.
\]
\begin{claim}\label{claim: r=4}
If $r=4$, then $m_3\geq 3$, and hence 
\[
v \geq -2 + \frac{1}{2}\cdot 2 + \frac{2}{3}\cdot 2 =\frac{1}{3}.
\]
\end{claim}
\begin{proof}[Proof of Claim~\ref{claim: r=4}]
Otherwise, $m_1=m_2=m_3=2$. Since $\sum_{1\leq i\leq 4}\sigma_i=0$, we have $m_4=|\sigma_4|=2$. It follows that $v=-2+ \frac{1}{2}\cdot 4=0$, which is a contradiction.
\end{proof}
\begin{claim}\label{claim: r=3}
 If $r=3$, then $v\geq \frac{1}{6}$, with the equality if and only if $m_1=2, m_2=m_3=6$.
\end{claim}
\begin{proof}[Proof of Claim~\ref{claim: r=3}]
If $r=3$, then we have 
\[
v=1-\left(\frac{1}{m_1}+\frac{1}{m_2}+\frac{1}{m_3} \right).
\]
If $m_1\geq 4$, then 
\[
v\geq 1-\frac{1}{4}\cdot 3 =\frac{1}{4}.
\]
If $m_1=3$, using the facts that $\sigma_1+\sigma_2+\sigma_3 = 0$ and $v>0$, one of the following cases occurs:
\begin{itemize}
    \item $(m_2, m_3)=(3k, 3k)$ for some $k\geq 2$. In this case, 
    \[
    v=1-\frac{1}{3}-\frac{2}{3k}=\frac{2k-2}{3k}\geq \frac{1}{3}.
    \]
    \item $\gcd(3, m_2)=1$ and $m_3=3m_2$. In this case,
    \[
    v=1-\frac{1}{3}-\frac{1}{m_2} - \frac{1}{3m_2} =\frac{2m_2-4}{3m_2}\geq \frac{1}{3}.
    \]
\end{itemize}
If $m_1=2$, using the facts that $\sigma_1+\sigma_2+\sigma_3 = 0$ and $v>0$ again, one of the following cases occurs:
\begin{itemize}
    \item $(m_2, m_3)=(2k+1, 4k+2)$ for some $k\geq 2$. In this case, 
    \[
    v=1-\frac{1}{2}-\frac{1}{2k+1} - \frac{1}{4k+2}=\frac{k-1}{2k+1}\geq \frac{1}{5}.
    \]
    \item $m_2=m_3=2k$ for some $k\geq 3$. In this case,
    \[
    v=1-\frac{1}{2}-\frac{1}{2k} -\frac{1}{2k} =\frac{k-2}{2k} \geq \frac{1}{6}
    \]
where $v=\frac{1}{6}$ holds if and only if $m_2=m_3=6$.
\end{itemize}
\end{proof}
\end{proof}
\begin{rmk}
    In the proof of Theorem~\ref{thm: min Ivol sm}, we have shown that $\min\Ivol(\sE_{0,0})\geq \frac{1}{6}$. In fact, as pointed out by Xin Lu, the equality can be realized: Let $C$ be the genus two curve with affine equation $y^2=x^6-1$. Then the group $G=\langle\sigma,\tau\rangle \cong (\ZZ/6\ZZ)\oplus (\ZZ/2\ZZ)$ acts on $C$ by $\sigma(x,y) = (\xi x, y)$ and $\tau(x,y)=(x,-y)$, where $\xi=\exp(\frac{2\pi i}{6})$ is a primitive $6$-th root of $1$. Then the quotient map $\pi\colon C\rightarrow C/G=\PP^1$ has three branch points with respective stabilizers $\langle\sigma\rangle$, $\langle\sigma^{-1}\tau\rangle$, $\langle\tau\rangle$. Let $G$ act on an elliptic curve by translations, and then diagonally on $C\times E$. Then the quotient surface $X=(C\times E)/G$ is smooth with $\kappa(K_X)=1$, and the Iitaka fibration is simply $X\rightarrow C/G=\PP^1$, induced by the projection $C\times E\rightarrow C$. We can check easily that $X\in \sE_{0,0}$ and $\vol_1(K_X) = \frac{1}{6}$.
\end{rmk}
Next we study the iterated derived sets of $\Ivol_\sm(2,1)$.
\begin{thm}\label{thm: Ivol sm derived}
For $n\in\Zz_{>0}$, the $n$-th iterated derived set of $\Ivol_\sm(2,1)$ is
\[
\Ivol_\sm(2,1)^{(n)} = \left\{n-1 + \sum_{1\leq i\leq s}\left(1 - \frac{1}{m_i}\right) \,\bigg|\, s\in \ZZ_{\geq 0},\, m_i\in \ZZ_{\geq 2} \right\}_{>0}.
\]
\end{thm}
\begin{proof}
By Theorem~\ref{thm: sm surf Ivol}, we have
\[
\Ivol_\sm(2,1)' = \Ivol(\sE_{1,0})'\cup \Ivol(\sE_{0,0})'.
\]
We look at $\Ivol(\sE_{1,0})'$ first. 
\begin{claim}
 We have   \[
 \Ivol(\sE_{1,0})' = \left\{\sum_{1\leq i\leq s}\left(1 - \frac{1}{m_i}\right) \,\bigg|\, s, m_i\in \Zz_{>0} \right\}_{>0}.
\]
\end{claim}
\begin{proof}
Suppose $v\in \Ivol(\sE_{1,0})'$. Then, since $\Ivol(\sE_{1,0})$ is DCC, there is an strictly increasing sequence $v_i\in \Ivol(\sE_{1,0})$ converging to $v$. We may write 
\[
v_i = -1+\sum_{1\leq j\leq r_i} \left(1-\frac{1}{m_{ij}}\right).
\]
Since $1-\frac{1}{m_{ij}}\geq \frac{1}{2}$ for each $1\leq j\leq r_i$, we have 
\[
\frac{1}{2} r_i-1 \leq v_i < v.
\]
It follows that $r_i < 2v + 4$ for any $i$. Therefore, by taking a subsequence of $\{v_i\}_i$, we may assume that $r_i=r$ is fixed for each $i$, and $\{(m_{i1}, \dots, m_{ir})\}_i$ is a strictly increasing sequence, where we use the lexicographic order for the integer vectors. By taking further subsequence, we may assume that, for $1\leq j\leq s$, $m_{ij}$ stay the same for any $j$, to be denoted by $m_j$, while for $s+1\leq j\leq r$, $m_{ij}$ goes to infinity with $i$. Then it is clear that the limit is
\begin{equation}\label{eq: v_infty}
 v = -1 + (r-s) + \sum_{1\leq j\leq s}\left(1 - \frac{1}{m_j}\right) = \frac{1}{2}(2r-2s-2)+\sum_{1\leq j\leq s}\left(1 - \frac{1}{m_j}\right).   
\end{equation}
and hence
\[
 \Ivol(\sE_{1,0})' \subset \left\{\sum_{1\leq j\leq s}\left(1 - \frac{1}{m_j}\right) \,\bigg|\, s\in \ZZ_{> 0},\, m_j\in \ZZ_{\geq 2}\right\}_{>0}.
\]
On the other hand, for any positive $v$ of the form $\sum_{1\leq j\leq s}\left(1 - \frac{1}{m_j}\right)$, we can write it as
\[
v = -1 + 1 + \sum_{1\leq j\leq s}\left(1 - \frac{1}{m_j}\right) = \lim_{m\to \infty}\left( -1+ \left(1-\frac{1}{m}\right) + \sum_{1\leq j\leq s}\left(1 - \frac{1}{m_j}\right)\right)
\]
which lies in $ \Ivol(\sE_{1,0})'$.
\end{proof}
Now we look at $\Ivol(\sE_{0,0})'$.
\begin{claim}
    We have  $ \Ivol(\sE_{0,0})' \subset\Ivol(\sE_{1,0})'$.
\end{claim}
\begin{proof}
Any given $v\in\Ivol(\sE_{0,0})'$ is the limit of an increasing sequence $\{v_i\}$ with
\[
v_i = -2+\sum_{1\leq j\leq r_i} \left(1-\frac{1}{m_{ij}}\right).
\]
Up to passing to a subsequence, we may assume that $r_i=r$ for each $j$, $m_{ij}=m_j$ for each $1\leq j\leq s$, and $m_{ij}$ goes to infinity with $j$ for $s+1\leq j\leq r$. Denote by $\sigma_{ij}$ the generator of local monodromy of order $m_{ij}$. Then we have $\sum_{1\leq i\leq r_i} \sigma_{ij}=0$. It follows that $r-s\geq 2$, that is, there are at least two indices $i$ such that $m_{ij}$ goes to infinity with $j$. It follows that
\[
v = (r-s-2) + \sum_{1\leq j\leq s}\left(1 - \frac{1}{m_j}\right) \subset \Ivol(\sE_{1,0})'.
\]
\end{proof}
In conclusion, we have proved that
\[
\Ivol_\sm(2,1)' = \left\{\sum_{1\leq j\leq s}\left(1 - \frac{1}{m_j}\right) \,\bigg|\, s\in \ZZ_{>0},\, m_j\in \ZZ_{\geq 2} \right\}_{>0}.
\]

Similarly, for $n\geq 2$, we may use induction to show that
\[
\Ivol_\sm(2,1)^{(n)} =\left(\Ivol_\sm(2,1)^{(n-1)}\right)' = \left\{n-1 + \sum_{1\leq j\leq s}\left(1 - \frac{1}{m_j}\right) \,\bigg|\, s\in \ZZ_{\geq 0}, m_j\in \ZZ_{\geq 2} \right\}.
\]
\end{proof}

\begin{cor}\label{cor: sm acc points}
\begin{enumerate}
      \item The minimal accumulation point of $\Ivol_\sm(2,1)$ is $\frac{1}{2}$.
      \item For a given $M\geq\frac{1}{2}$, the accumulation complexity of $\Ivol_\sm(2,1)_{\leq M}$ is $\lfloor M \rfloor + 1$.
      \item The accumulation complexity of $\Ivol_\sm(2,1)$ is infinite.
      \item $\Ivol_\sm(2,1)$ is closed in $\RR$.
  \end{enumerate} 
\end{cor}
\begin{proof}
(1) By Theorem~\ref{thm: Ivol sm derived}, we have 
\[
\min \Ivol_\sm(2,1)' =\left\{\sum_{1\leq i\leq s}\left(1 - \frac{1}{m_i}\right) \,\bigg|\, s\in \ZZ_{\geq 0},\, m_i\in \ZZ_{\geq 2} \right\}_{>0}= \frac{1}{2}. 
\]

(2) For any given $M\geq\frac{1}{2}$,
\[
\Ivol_\sm(2,1)_{\leq M}^{(n)} = \Ivol_\sm(2,1)_{\leq M} \cap \Ivol_\sm(2,1)^{(n)} = 
\begin{cases}
    \neq \emptyset &  \text{if $n\leq M +1$}, \\
    = \emptyset & \text{if $n> M + 1$}.
\end{cases}
\]
Therefore, the accumulation complexity of $\Ivol_\sm(2,1)_{\leq M}$ is $\lfloor M \rfloor +1$.

(3) The number $M$ in (2) can be arbitrarily large, so the accumulation complexity of $\Ivol_\sm(2,1)$ is infinite.

(4) Combining Theorems~\ref{thm: sm surf Ivol} and \ref{thm: Ivol sm derived}, we see that $\Ivol_\sm(2,1)'$ is subset of $\Ivol_\sm(2,1)$, and hence $\Ivol_\sm(2,1)$ is a closed subset of $\RR$.
\end{proof}

\begin{rmk}\label{rmk: not bdd}
For a given $v\in \Ivol_\sm(2,1)$, the set of surfaces $$\gP_\sm(2,1)_v:=\{X\in \gP_\sm(2,1)\mid \vol_1(K_X)=v\}$$ is usually not bounded. For example, for $v=2g-2$ with $g\in \ZZ_{\geq 2}$ and an arbitrary $n\in \ZZ_{>0}$, we can take a smooth projective curve $C$ with genus $g$ and an \'etale  $(\ZZ/n\ZZ)$-cover $\pi\colon \tC_n \rightarrow C$. Let $G=\ZZ/n\ZZ$ act on a smooth elliptic curve $E$ by translations. Then the diagonal action of $G$ on $\tC_n\times E$ is free, and $X_n:=(\tC_n\times E)/G\rightarrow C=\tC_n/G$ is an elliptic bundle with $\vol_1(K_{X_n}) = 2g-2 = v$. However, the set of surfaces $\{X_n\mid n\in \ZZ_{>0}\}$ is not bounded.
\end{rmk}

\begin{rmk}
    If we allow non-algebraic surfaces, then by applying the logarithmic transformation to three fibers of $\PP^1\times E\rightarrow \PP^1$, we can obtain a smooth elliptic surface $X\rightarrow \PP^1$ with exactly three singular fibers, which are of type $2\I_0, 3\I_0$ and $7\I_0$ respectively. Then $\chi(\sO_X)=0$ and by the canonical bundle formula we have
    \[
    \vol_1(K_X) = -2 +\left(\frac{1}{2} +\frac{2}{3} + \frac{6}{7} \right) = \frac{1}{42}.
    \]
    It is easy to see that this is the minimal possible Iitaka volume of a properly elliptic surface that is not necessarily algebraic.
\end{rmk}

    Next we want to investigate the Iitaka volumes of smooth elliptic surfaces with prescribed geometric genus $p_g(X)$. For a given non-negative integer $p_g$, we introduce the set
    \[
    \Ivol_\sm(2,1; p_g):=\{\vol_1(K_X)\mid X\in \gP_\sm(2,1),\, p_g(X) =p_g\}.
    \]
\begin{thm}\label{thm: Ivol sm prescribe pg}
The following holds.
\begin{enumerate}
    \item $\Ivol_\sm(2,1; 0)= \Ivol_\sm(2,1)$.
    \item If $p_g>0$, then
    \[
\Ivol_\sm(2,1;p_g) = \left\{p_g-1+\sum_{1\leq i\leq r} \left(1-\frac{1}{m_i}\right) \,\,\bigg|\,\, r\geq0,\, m_1,\dots, m_r\in \ZZ_{\geq 2}\right\}_{>0}.
\]
\end{enumerate}
\end{thm}
\begin{proof}
By the proofs of Propositions~\ref{prop: sm Ivol1} and \ref{thm: sm surf Ivol}, one sees that each $v\in \Ivol_\sm(2,1)$ is realized as $\vol_1(K_X)$ for some smooth elliptic surface $X$ with $\kappa(K_X)=1$ and $p_g(X)=0$. Therefore, $ \Ivol_\sm(2,1; 0)= \Ivol_\sm(2,1)$.

Let $V(p_g)$ denote the right hand side of the equality in (2). 

We first show $\Ivol_\sm(2,1;p_g)\supset V(p_g)$ by constructing smooth elliptic surfaces with positive geometric genus $p_g$ as follows: Start with a rational elliptic surface $f_1\colon X_1\rightarrow C_1\cong \PP^1$ with a section, which has invariants
    \[
    \chi(\sO_{X_1}) = 1,\, e(X_1) = 12\chi(\sO_{X_1}) =12.
    \]
For a given $n\in \Zz_{>0}$, take two smooth fibers $F_i=f_1^*b_i$, $i\in\{1,2\}$, and a $(\ZZ/n\ZZ)$-cover $C_n\cong \PP^1 \rightarrow C_1\cong \PP^1$ branched exactly at the two points $b_1$ and $b_2$, and then form the fiber product
\[
\begin{tikzcd}
    X_n = X_1\times_{C_1} C_n \arrow[r] \arrow[d, "f_n"']& X_1\arrow[d, "f_1"] \\ 
    C_n\arrow[r] & C_1.
\end{tikzcd}
\]
Then $f_n\colon X_n\rightarrow C_n$ is a relatively minimal elliptic fibration with 
\[
e(X_n) =  12n,\, \chi(\sO_{X_n}) = \frac{1}{12}e(X_n) = n.
\]
Since $f_n$ contains singular fibers not of type $m\I_0$, we have
\[
q(X_n) = g(C_n) = 0.
\]
It follows that 
\[
p_g(X_n) = \chi(\sO_{X_n})+q(X_n)-1 = n-1.
\]
Also, $f_n\colon X_n\rightarrow C_n$ has a section induced from that of $f_1$. 

For a given $p_g>0$, the algebraic elliptic surfaces $X$ obtained by performing appropriate logarithmic transformations along  smooth fibers of $X_{p_g+1}$ can have invariants \[
p_g(X)=p_g,\, \vol_1(K_X) = p_g-1+\sum_{1\leq i\leq r} \left(1-\frac{1}{m_i}\right)
\]
for any collection of multiplicities $m_1, \dots, m_r$.

Now we show the other inclusion $\Ivol_\sm(2,1;p_g)\subset V(p_g)$, using the canonical bundle formula for a properly elliptic surface $f\colon X\rightarrow C$ with $p_g(X)=p_g$:
\[
\vol_1(K_X) = 2g(C) -2 +\chi(\sO_X)+ \sum_{i} \left(1-\frac{1}{m_i}\right),
\]
where $\chi(\sO_X) = p_g(X) - q(X) + 1$. 

To show that $\vol_1(K_X)\in V(p_g)$, it suffices to show that $2g-2+\chi\geq p_g-1$.
Note that $g(C)\leq q(X)\leq g(C)+1$, and if $q(X)=g(C)+1$ then $f$ has only singular fibers of type $m\I_0$ and $\chi(\sO_X)=0$. Therefore, if $g(C)\geq 1$ then
\[
2g(C) -2 +\chi(\sO_X) = 2g(C) -2 + p_g(X) - q(X) + 1\geq p_g(X) + g(C) -2 \geq p_g-1.
\]
If $q(X)=0$ then again
\[
2g(C) -2 +\chi(\sO_X) = 2g(C) -2 + p_g(X) + 1\geq p_g-1.
\]
Now suppose that $g(C)=0$ and $q(X)>0$. Then $q(X)=1$ and the singular fibers of $f\colon X\rightarrow C$ are all of type $m\I_0$. It follows that $\chi(\sO_X)=0$, and hence $p_g(X) = 0$, which is a contradiction to the assumption that $p_g>0$. 
\end{proof}

\begin{cor}\label{cor: min Ivol sm prescribe pg}
Let $p_g$ be a non-negative integer. Then the following holds.
\begin{enumerate}
\item
\[
\min\Ivol_\sm(2,1;p_g) = 
\begin{cases}
    \frac{1}{6} & \text{if $p_g=0$},\\
    \frac{1}{2} & \text{if $p_g=1$},\\
    p_g-1 & \text{if $p_g\geq 2$}.\\
 \end{cases}
\]
\item For $0\leq p_g< p_g'$, we have $\Ivol_\sm(2,1;p_g) \supsetneq \Ivol_\sm(2,1;p_g')$.
\end{enumerate} 
\end{cor}

The following statements and their proofs are parallel to those of Corollary~\ref{cor: sm acc points}.
\begin{cor}\label{cor: Ivol ccc sm prescribe pg}
Let $p_g$ be a positive integer.
\begin{enumerate}
      \item The minimal accumulation point of $\Ivol_\sm(2,1; p_g)$ is $p_g$.
      \item For a given $M\geq p_g$, the accumulation complexity of $\Ivol_\sm(2,1; p_g)_{\leq M}$ is $\lfloor M\rfloor -p_g + 1$.
      \item The accumulation complexity of $\Ivol_\sm(2,1; p_g)$ is $\infty$.
      \item $\Ivol_\sm(2,1;p_g)$ is closed in $\RR$.
  \end{enumerate} 
\end{cor}

\subsection{Crepant smooth models for log canonical surfaces $(X, B)\in \gP_\lc^\Gamma(2,1)$}\label{sec: lc surf}
\begin{lem}\label{lem: rel min model}
Let $(X, B)$ be a projective log canonical surface, and $f\colon X\rightarrow Z$ a fibration onto a curve $Z$ such that $K_X+B\sim_{\RR, Z} 0$. Consider the following construction:
\begin{equation}\label{diag: sm model}
\begin{tikzcd}
    &(\widetilde X, B_{\widetilde X}) \arrow[rd, "\rho"]\arrow[ld,"\pi"']& \\
    (X,B) & & (S, B_S)
\end{tikzcd}
\end{equation}
where 
\begin{itemize}
    \item $\pi\colon\tX\rightarrow X$ is the minimal resolution of singularities, $K_{\tX}+B_{\tX} = \pi^*(K_X+B)$ such that $\pi_*B_{\tX} = B$, and
    \item $\rho \colon \tX\rightarrow S$ is a contraction of $\tilde f$-vertical curves, where $\tilde f=f\circ \pi$, and $B_{S}:=\rho _* B_{\tX}$.
\end{itemize}
Then $K_{S} + B_{S}\sim_{\RR, Z} 0$, $\mld(S, B_{S}) = \mld(X,B)$. Moreover, the Iitaka--Kodaira dimensions and Iitaka volumes of $K_X+B$ and $K_{S}+B_{S}$ are the same.
\end{lem}
\begin{proof}
Since $K_{\tX}+B_{\tX} = \pi^*(K_X+B)$, and $K_X+B\sim_{\RR, Z}0$, we infer that $K_{\tX}+B_{\tX}\sim_{\RR, Z} 0$. Since $\rho $ contracts only $\tilde f$-vertical curves, which is $(K_{\tX}+B_{\tX})$-trivial, we infer that $K_{S}+B_{S}$ is numerically trivial over $Z$, and $K_{\tX}+B_{\tX} = \rho ^*(K_{S}+B_{S})$ holds. It follows that 
\[
\mld(S, B_{S}) =\mld(\tX, B_{\tX}) = \mld(X,B), \text{ and } K_{S}+B_{S}\sim_{\RR, Z}0.
\]
By the fact that $\pi^*(K_X+B) = \rho ^*(K_{S}+B_{S})$, the log canonical divisors $K_X+B$ and $K_{S}+B_{S}$ have the same Iitaka--Kodaira dimension and Iitaka volume.
\end{proof}

Note that we can take $(S, B_S)$ in Lemma~\ref{lem: rel min model} so that $S$ is smooth and the induced $h\colon S\rightarrow Z$ is relatively minimal, that is, there are no $(-1)$-curves in the fibers of $h$. This gives a way of finding all the Iitaka volumes of log canonical surfaces with coefficients from a given set $\Gamma\subset[0,1]$.

\begin{thm}\label{thm: lc surf Ivol}
Let $\Gamma\subset[0,1]$ be a subset containing $0$. Then a positive real number $v$ lies in $\Ivol_\lc^\Gamma(2,1)$ if and only if there is a smooth projective surface $S$, a relative minimal fibration $h\colon S\rightarrow Z$, and a boundary divisor $B_S$ on $S$ such that  $(S, B_S)$ is lc, $\kappa(K_{S}+B_S)=1$, $K_{S}+B_S\sim_{\RR, Z}0$,  $\vol_1(K_{S}+B_S) =v$, and, additionally, the following holds:
\begin{enumerate}
    \item Let $B_S=B_S^h + B_S^v$ be the decomposition into the horizontal part $B_S^h$ and the vertical part $B_S^v$. Then the coefficients of $B_S^h$ lie in $\Gamma$.
    \item $B_S^v =\sum_i c(z_i) h^* z_i$, where the $h^* z_i$ are distinct fibers of $h$, and for each $i$ there is a prime divisor $E$ over $S$ with its center lying in $h^*z_i$ and $1-a(E, S, B_S)\in \Gamma\cup\{1\}$.
\end{enumerate} 
\end{thm}
\begin{proof}
We prove the only if part of the statement first. Starting with $(X, B)\in\gP_\lc^\Gamma(2,1)$ with $\vol_1(K_X+B)=v$ and Iitaka fibration $f\colon X\rightarrow Z$, we take $(\tX, B_{\tX})$ and $(S, B_S)$ as in Lemma~\ref{lem: rel min model} such that $S$ is the a relatively minimal smooth model of $\tX$ over $Z$. Then, by Lemma~\ref{lem: rel min model}, $(S, B_S)$ is lc, and 
\[
\kappa(K_{S}+B_S)=1,\,\, K_{S}+B_S\sim_{\RR, Z}0,\,\,  \vol_1(K_{S}+B_S) =v.
\]
Since the birational morphisms $\pi\colon \tX\rightarrow X$ and $\rho\colon\tX\rightarrow S$ are over $Z$, the horizontal parts $B_X^h$ and $B_S^h$ of $B_X$ and $B_S$ respectively are the strict transforms of each other. Since $B_X^h\in \Gamma$, so is $B_S^h$. This proves (1). 

For (2), note that $h\colon S\rightarrow Z$ is a relatively minimal elliptic fibration (resp.~a $\PP^1$-bundle) if $B^h=0$ (resp.~if $B^h\neq 0$). In both cases, we have $B^v\sim_{\RR, Z}0$, and hence $B_S^v =\sum_i c(z_i) h^* z_i$ for suitable $c(z_i)\in \RR_{\geq 0}$. Take $E\subset \tX$ to be the strict transform of an irreducible component of $f^*z_i$. Then
\[
1-a(E, S, B_S) = 1- a(E, X, B) \in \Gamma.
\]
Note that, even when $\pi(E)$ does not appear in $B$, we have $1-a(E, X, B)=0\in \Gamma$ by the assumption on $\Gamma$.  

Now we prove the if part. If all of the coefficients of $B_S$ are in $\Gamma$, then $v = \vol_1(K_{S}+B_S)\in \Ivol_\lc^\Gamma(2,1)$. In general, by Lemma~\ref{lem: model right coefficient}, we may construct a crepant model $(X, B)$ of $(S, B_S)$ that lies in $\gP_\lc^\Gamma(2,1)$. Therefore, $v\in \Ivol_\lc^\Gamma(2,1)$.
\end{proof}

Therefore, in order to determine $\Ivol_\lc^\Gamma(2,1)$, it suffices to find all the smooth lc surfaces $(S, B_S)$ satisfying the conditions specified in Theorem~\ref{thm: lc surf Ivol}. In the next subsection, we will use this recipe to give a more detailed description of the two sets $\Ivol_\lc^{\{0\}}(2,1)$ and $\Ivol_\lc^{\{0,1\}}(2,1)$. In particular, their minima and minimal accumulation points are found; see Theorems~\ref{thm: min lc Ivol 0} and \ref{thm: min lc Ivol 1}.

\begin{nota}\label{nota: Ivol(S/Z) and V(chi,g)}
For a relatively minimal fibration $h\colon S\rightarrow Z$ with fiber genus $\leq 1$, we define
\[
\Ivol_\lc^\Gamma(S/Z):=\left\{\vol_1(K_X+B)  \,\,\Bigg|\,\,
\begin{aligned}
& \text{$(X, B)$ is lc, $B\in\Gamma$, $\kappa(K_X+B)=1$},\\
& \text{and there is a diagram as in \eqref{diag: sm model}}
\end{aligned}
\right\}.
\]
For $(\chi, g)\in \ZZ_{\geq 0}^2$, let $V_{\lc}^\Gamma(\chi, g)$ be the union of all $\Ivol_\lc^\Gamma(S/Z)$ such that $S\rightarrow Z$ is an elliptic fibration with $\chi(\sO_S)=\chi$ and $g(Z)=g$.

When $\Gamma=\{0\}$, we usually omit $\Gamma$ from the above notation.
\end{nota} 

\subsection{The set $\Ivol_\lc^{\{0\}}(2,1)$}
\begin{prop}\label{prop: Ivol chi>0 lc}
    Let $h\colon S\rightarrow Z$ be a relatively minimal elliptic surface with $\chi(\sO_S)>0$, and let $n_1$, $n_2$, $n_3$, $n_4$, $n_0^*$, $n_1^*$, $n_2^*$, $n_3^*$ and $n_4^*$ be the number of singular fibers of type $m\I_k$ ($m\geq 2$ or $k>0$), $\II$, $\III$, $\IV$, $\I_0^*$, $\I^*_k (k\geq 1)$, $\II^*$, $\III^*$ and $\IV^*$ respectively. Then $$\Ivol_\lc(S/Z) = \left\{v\,\bigg|\, v=2g(Z)-2+\chi(\sO_S)+ \sum_{0\leq i\leq 4}c_i+\sum_{0\leq i\leq 4} c_i^* \right\}_{>0}$$
    where $c_i\in \sum_{\leq n_i} \sC_i$ and $c_i^*\in \sum_{\leq n_i^*} \sC_i^*$ for $0\leq i\leq 4$, and $n_0$ is understood to be $+\infty$. Here the sets $\sC_i$ and $\sC_i^*$, $0\leq j\leq 4$ are specified as follows:
\begin{itemize}[leftmargin=*]
    \item $\sC_0=\{1, 0\}$,
    \item $\sC_1=\{1, 1-\frac{1}{n}\mid n\in \Zz_{>0}\}=\{1,0,\frac{1}{2},\frac{2}{3},\frac{3}{4},\frac{4}{5},\frac{5}{6},\frac{6}{7},\frac{7}{8},\frac{8}{9},\frac{9}{10},\frac{10}{11},\frac{11}{12},\frac{12}{13},\frac{13}{14},\dots\}$,
    \item $\sC_2=\left\{\frac{5}{6}, \frac{5m+n-1}{6m+n}\,\,\bigg|\,\, m\in\ZZ_{\geq 0}, n\in \{1,2,3\} \right\} = \left\{\frac{5}{6},0,\frac{5}{7},\frac{3}{4},\frac{10}{13},\frac{7}{9},\frac{11}{14},\frac{15}{19},\frac{4}{5},\frac{25}{31},\frac{21}{26}, \frac{17}{21}, \frac{13}{16},\frac{22}{27}, \dots\right\}$, 
    \item $\sC_3=\left\{\frac{3}{4}, \frac{3m+n-1}{4m+n} \,\,\big|\,\, m\in\ZZ_{\geq 0}, n\in \{1,2\}\right\}=\left\{\frac{3}{4},0,\frac{1}{2},\frac{3}{5},\frac{2}{3},\frac{9}{13},\frac{7}{10},\frac{12}{17},\frac{5}{7},\frac{18}{25}, \frac{13}{18}, \frac{21}{29},\frac{8}{11},\frac{27}{37}, \frac{19}{26},\dots\right\}$,
    \item $\sC_4=\left\{\frac{2}{3}, \frac{1}{3}, \frac{2m}{3m+1} \,\,\big|\,\,m\in\ZZ_{\geq 0}\right\}=\left\{\frac{2}{3},0,\frac{1}{3},\frac{1}{2},\frac{4}{7},\frac{3}{5},\frac{8}{13},\frac{5}{8},\frac{12}{19},\frac{7}{11},\frac{16}{25},\frac{9}{14}, \frac{20}{31}, \frac{11}{17}, \frac{24}{37},\dots\right\}$,
    \item $\sC_0^*=\left\{\frac{1}{2}, \frac{m}{2m+1}\,\,\big|\,\, m\in \ZZ_{\geq 0}\right\}=\left\{\frac{1}{2},0,\frac{1}{3},\frac{2}{5},\frac{3}{7},\frac{4}{9},\frac{5}{11},\frac{6}{13},\frac{7}{15},\frac{8}{17},\frac{9}{19},\frac{10}{21},\frac{11}{23},\frac{12}{25},\frac{13}{27},\dots\right\}$,
    \item $\sC_1^*=\left\{\frac{1}{2}, \frac{\lfloor (m-1)/2\rfloor}{m}\,\,\big|\,\, m\in \Zz_{>0}\right\}=\left\{\frac{1}{2},0,\frac{1}{4},\frac{1}{3},\frac{3}{8},\frac{2}{5},\frac{5}{12}, \frac{3}{7},\frac{7}{16},\frac{4}{9},\frac{5}{11},\frac{6}{13},\frac{7}{15},\frac{8}{17},\frac{9}{19},\dots\right\}$,
    \item $\sC_2^*=\left\{\frac{1}{6}, \frac{m}{6m+n}\,\,\big|\,\, m\in \ZZ_{\geq 0}, n\in\{3,4,5\}\right\} = \left\{\frac{1}{6},0, \frac{1}{11},\frac{1}{10},\frac{1}{9},\frac{2}{17},\frac{1}{8},\frac{3}{23},\frac{2}{15},\frac{3}{22},\frac{4}{29},\frac{1}{7},\frac{5}{34},\frac{4}{27},\frac{5}{33},\dots\right\}$,
     \item $\sC_3^*=\left\{\frac{1}{4}, \frac{m}{4m+n}\,\,\big|\,\, m\in \ZZ_{\geq 0}, n\in\{2,3\}\right\}=\left\{\frac{1}{4},0,\frac{1}{7}, \frac{1}{6},\frac{2}{11}, \frac{1}{5},\frac{4}{19}, \frac{3}{14},\frac{5}{23}, \frac{2}{9}, \frac{7}{31}, \frac{5}{22},\frac{8}{35}, \frac{3}{13},\frac{10}{43},\dots\right\}$,
     %\frac{7}{30},\frac{11}{47}, \frac{4}{17},\frac{9}{38},\frac{5}{21},\frac{11}{46},\frac{6}{25},\frac{13}{54},
     \item $\sC_4^*=\left\{\frac{1}{3}, \frac{m}{3m+2}\,\,\big|\,\, m\in \ZZ_{\geq 0}\right\}=\left\{\frac{1}{3},0,\frac{1}{5},\frac{1}{4},\frac{3}{11},\frac{2}{7},\frac{5}{17},\frac{3}{10},\frac{7}{23},\frac{4}{13},\frac{9}{29},\frac{5}{16},\frac{11}{35},\frac{6}{19},\frac{13}{41},\dots\right\}$.
\end{itemize}   
\end{prop}
\begin{proof}
We divide the proof into two steps.

\medskip 

\noindent\textbf{Step 1.} In this step, we assume that $h\colon S\rightarrow Z$ does not have any multiple fibers. By Theorem~\ref{thm: lc surf Ivol}, we need to consider log canonical surface $(S, B_S)$ with $\kappa(K_S+B_S)=1$, where $B_S=\sum_{i=1}^s c(z_i) h^* z_i$ is an $\RR_{\geq 0}$-linear combination of fibers of $h$. Since $h$ does not have any multiple fiber, by the canonical bundle formula for elliptic fibrations, the Iitaka volume of $(S, B_S)$ is then
    \begin{equation}\label{eq: Ivol S}
            \vol_1(K_S+B_S) = 2g(Z)-2 + \chi(\sO_S) +\sum_{i=1}^s c (z_i).
    \end{equation}
 Therefore, the task is to find out the possibilities of the coefficients $c(z_i)$, which are subject to the condition (2) of Theorem~\ref{thm: lc surf Ivol}. For the coefficient set $\Gamma=\{0\}$, the requirement is that there is an exceptional divisor $E$ over $F$ such that the log discrepancy $a(E, S, B_S)\in\{0,1\}$. 

We can proceed according to the Kodaira types of the fibers (\cite[page 201]{BHPV04}), and find the following table:
\begin{center}
\begin{tabular}{|c|c|c|c|c|c|c|c|c|c|c|}
\hline
Kodaira type of $h^*z$ & $\I_0$ & $\I_k, k\geq 1$ & $\II$ & $\III$ & $\IV$ & $\I_0^*$ &$\I_k^*, k\geq 1$ & $\II^*$ & $\III^*$ & $\IV^*$\\
\hline
set of $c(z)$ & $\sC_0$  & $\sC_1$  & $\sC_2$  & $\sC_3$  & $\sC_4$  &$\sC_0^*$& $\sC_1^*$  & $\sC_2^*$  & $\sC_3^*$  & $\sC_4^*$  \\
\hline
\end{tabular}
\end{center}
where the first row denotes the Kodaira type of the fiber $h^*z$, and the second row denotes the set of possible values for $c(z)$.

We explain how we obtain the coefficient $c:=c(z)$ for a fiber $F:= h^*z$ of type $\II$; the other cases are similar. We blow up three times to arrive at a log resolution $\tilde\rho=\rho_1\circ\rho_2\circ\rho_3\colon \widetilde S\rightarrow S$, so that
\[
\tilde\rho^*(K_S+cF) = K_{\widetilde S} + c\widetilde F + (2c-1)E_1 + (3c-2)E_2 + (6c-4)E_3
\]
where $E_i\subset\widetilde S$ is the strict transform of $\Exc(\rho_i)$ and $\tF$ is the strict transform of $F$. The following picture illustrates the inverse images of $F$ under the blow-ups:
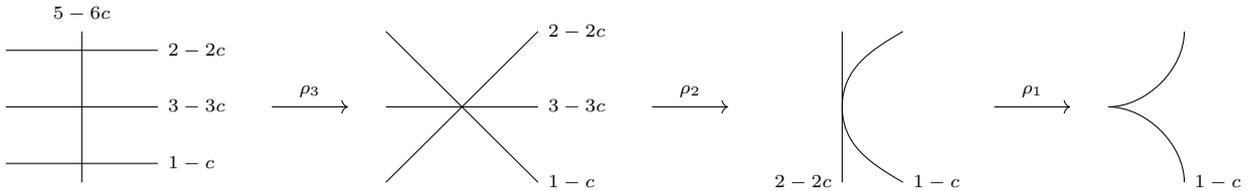
\begin{center}
\begin{tikzpicture}[font=\tiny]

\begin{scope}[xshift=-5cm]
\draw (-1,0.75) -- (1,0.75)node[right]{$2-2c$};
\draw (-1, 0) -- (1,0)node[right]{$3-3c$};
\draw (-1, -0.75) -- (1,-0.75)node[right]{$1-c$};
\draw (0, -1) -- (0,1) node[above]{$5-6c$};
\draw[->] (2.5,0) --node[above]{$\rho_3$}(3.5,0);
\end{scope}

\begin{scope}
\draw (-1, 1) --(1,-1) node[right]{$1-c$} ;
\draw (-1,-1) -- (1,1)node[right]{$2-2c$} ;
\draw (-1,0) -- (1,0)node[right]{$3-3c$} ;
\draw[->] (2.5,0) --node[above]{$\rho_2$}(3.5,0);
\end{scope}
\begin{scope}[xshift = 5cm]
\path (0,0) coordinate (O);
\path (.8,1) coordinate (P1);
\path (.8,-1) coordinate (P2);
\path (1.5, .8) coordinate (R2);
\path (1,1) coordinate (S1);
\draw (P1) to  [out=210, in = 90] (O);
\draw (P2) to  [out=150, in = -90] (O);
\draw (0,-1)node[left]{$2-2c$} -- (0,1);
\node[right] at (P2) {$1-c$};
\draw[->] (2,0) --node[above]{$\rho_1$}(3,0);
\end{scope}
\begin{scope}[xshift = 8.5cm]
\draw (0,0) .. controls (.5,0) and (1,.5).. (1,1);
\draw (0,0) .. controls (.5,0) and (1,-.5).. (1,-1)node[right]{$1-c$};
\end{scope}
\end{tikzpicture}
\end{center}
where the numbers indicate the log discrepancies of the components with respect to the pair $(S, cF)$. For $(S, cF)$ to be log canonical, we need $0\leq c\leq \frac{5}{6}$.  The log discrepancy of a divisor $E$ over the point $E_3\cap \widetilde F$ (resp.~$E_3\cap E_1$, resp.~$E_3\cap E_2$) can be written as $m(5-6c)+n(1-c)$ (resp.~$m(5-6c)+n(2-2c)$, resp.~$m(5-6c)+n(3-3c)$), where $m$ and $n$ are coprime positive integers. Thus, a divisor $E$ with log discrepancy $a(E, S, cF)\in\{0,1\}$ and $\Center_S(E)\subset F$ exists if and only if one of the following expressions has value in $\{0,1\}$ under the restriction that $0\leq c\leq \frac{5}{6}$:
\[
5-6c,\, 2-2c,\, 3-3c,\, 1-c,\, m(5-6c)+n(1-c),\, m(5-6c)+n(2-2c),\, m(5-6c)+n(3-3c)
\]
where $m,n$ are coprime positive integers. Solving these equations for $c$, we obtain the set 
\begin{align*}
\sC_2 &=\left\{0, \frac{1}{2}, \frac{2}{3}, \frac{5}{6}, \frac{5m+n-1}{6m+n},\frac{5m+2n-1}{6m+2n},\frac{5m+3n-1}{6m+3n} \,\,\bigg|\,\, (m,n)\in \Zz_{>0}^2, \gcd(m,n)=1\right\}_{\leq \frac{5}{6}} \\
&=\left\{0, \frac{1}{2}, \frac{2}{3}, \frac{5}{6}, \frac{5m}{6m+1}, \frac{5m+1}{6m+2}, \frac{5m+2}{6m+3} \,\,\bigg|\,\, m\in\Zz_{>0}\right\} \\
&= \left\{\frac{5}{6}, \frac{5m+n-1}{6m+n}\,\,\bigg|\,\, m\in\ZZ_{\geq 0}, n\in \{1,2,3\} \right\}
\end{align*}
where for the second equality we use the following transformation of expressions
\[
\frac{5m+4}{6m+5} = \frac{3(5m+4)}{3(6m+5)} = \frac{5(3m+2)+2}{6(3m+2)+3}
\]
and
\[
\frac{5m+3}{6m+4} = \frac{2(5m+3)}{2(6m+4)} = \frac{5(2m+1)+1}{6(2m+1)+2}.
\]

Now we can compute the $\vol_1(K_S+B_S)$ by \eqref{eq: Ivol S}, and the description of $\Ivol_\lc(S/Z)$ as in the proposition follows.

\medskip

\noindent\textbf{Step 2.} In this step, we consider the general case, where $h$ can have multiple fibers. For a multiple fiber $h^*z = mF $ of type $mI_k, k\geq 0$, we can compute in the same way as in Step 1 that the coefficient $c(z)$ should be as follows:
\begin{center}
\begin{tabular}{|c|c|c|}
\hline
Kodaira type & $m\I_0$ & $m\I_k, k\geq 1$ \\
\hline
 $c(z)$ & $\frac{1}{m}\sC_0$  & $\frac{1}{m}\sC_1$  \\
\hline
\end{tabular}
\end{center}
so that $(\widetilde S, c(z)mF))$ satisfies the condition (2) of Theorem~\ref{thm: lc surf Ivol}.

Up to relabeling, we may assume that $\tilde h^*z_i =m_i F_i$, $1\leq i\leq r$ are all of the multiple fibers of $\tilde h\colon\widetilde S\rightarrow Z$. Then by the canonical bundle formula, we have
  \begin{equation}\label{eq: Ivol S with multiple fib}
            \vol_1(K_S+B_S) = 2g(Z)-2 + \chi(\sO_S) +\sum_{1\leq i\le r}\left(1-\frac{1}{m_i}\right)+\sum_{i=1}^s c (z_i)
    \end{equation}
where $c(z_i)\in \frac{1}{m_i}\sC_0$ or $\frac{1}{m_i}\sC_1$ for $1\leq i\leq r$, depending on whether $\tilde h^*z_i$ is of type $mI_0$ or $mI_k$ with $k>0$. The total contribution of a multiple fiber $m_iF_i$ in \eqref{eq: Ivol S with multiple fib} can be 
$$
\begin{cases}
1-\frac{1}{m_i}, \text{ or } 1 =  (1-\frac{1}{m_i}) + \frac{1}{m_i} & \text{if $h^*z$ is of type $mI_0$}\\
1-\frac{1}{m_i},\, 1-\frac{1}{m_in_i} = (1-\frac{1}{m_i})+\frac{1}{m_i}(1-\frac{1}{n_i}), \text{ or } 1 = 1-\frac{1}{m_i} + \frac{1}{m_i} & \text{if $h^*z$ is of type $mI_k, k>0$}
\end{cases}
$$
In conclusion, we may write
\begin{equation}\label{eq: Ivol S with multiple fib'}
            \vol_1(K_S+B_S) = 2g(Z)-2 + \chi(\sO_S) +\sum_{i=1}^r c (z_i)' + \sum_{i=r+1}^s c (z_i)
    \end{equation}
where $c(z_i)'$ can take any value in $\sC_1 = \{1, 1-\frac{1}{n} \mid n\in \Zz_{>0}\}$. Therefore, $\Ivol_\lc(S/Z)$ is still of the form given in the lemma.
\end{proof}

% \begin{rmk}
%     $\Ivol_\klt^\Gamma(S/Z)$ and $\bigcup_{(\widetilde S/Z)\in \sE(S/Z)} \Ivol_\klt(\widetilde S/Z)$ have a similar description, replacing the sets $\sC_i$ and $\sC_i^*$ with $\sC_i\setminus \{\max\sC_i\}$ and $\sC_i^*\setminus \{\max\sC_i^*\}$ respectively.
% \end{rmk}

\begin{prop}\label{prop: Ivol chi=0 lc}
Let $h\colon S\rightarrow Z$ be a relatively minimal elliptic fibration with $\chi(\sO_X)=0$. Then $\Ivol_\lc(S/Z) \subset \Ivol_\sm(2,1)$.
\end{prop}
\begin{proof}
Let $h^*z_i = m_i F_i$, $1\leq i\leq r$ be the multiple fibers of $h$. Suppose that $B_S = \sum_{1\leq i\leq s} c(z_i) h^*z_i$, where $h^*z_i=m_i F_i$ is of type $m_i\I_0$ and $m_i=1$ for $r<i\leq s$. Then we have $c(z_i)\in\frac{1}{m_i}\sC_0=\left\{0,\frac{1}{m_i}\right\}$. Let $I=\{1\leq i\leq r \mid c(z_i)=\frac{1}{m_i}\}$ and $J=\{1\leq i\leq r \mid c(z_i)=0\}$. By the canonical bundle formula
\begin{align*}
\vol_1(K_S+B_S) & = 2g(Z) - 2 + \sum_{1\leq i\leq r} \left(1-\frac{1}{m_i}\right) + \sum_{1\leq i\leq s} c(z_i) \\
&  = 2g(Z) - 2 + (s-r)+ \#I +\sum_{i\in J} \left(1-\frac{1}{m_i}\right) \in \Ivol_\sm(2,1).
\end{align*}
\end{proof}

Now we look for the minimum and the minimal accumulation point of $\Ivol_\lc(2,1)$.
\begin{thm}\label{thm: min lc Ivol 0}
    We have $\min\Ivol_\lc(2,1)=\frac{1}{671}$, and the minimal accumulation point of $\Ivol_\lc(2,1)$ is $\frac{1}{66}$.
\end{thm}
\begin{proof}
We have 
\[
\Ivol_\lc(2,1) = \bigcup_{(\chi,g)\in\ZZ_{\geq 0}^2}V_\lc(\chi, g)
\]
where $V_\lc(\chi, g) = V^{\{0\}}_\lc(\chi, g)$ is defined as in Notation~\ref{nota: Ivol(S/Z) and V(chi,g)}.

%Obviously, $\Ivol_\sm(2,1)\subset \Ivol_\lc(2,1)$. 
% Let $\sJ$ be the set of Jacobian elliptic fibrations, which we can decompose as 
% \[
% \sJ=\bigcup_{(\chi,g)\in\ZZ_{\geq 0}^2}\sJ_{\chi, g}
% \]
% where $\sJ_{\chi, g}$ consists of those Jacobian elliptic fibrations $f\colon S\rightarrow Z$ with $\chi(\sO_S)=\chi$ and $g(Z)=g$. Then, we have 
%     \[
%     \min\Ivol_\lc(2,1) = \min\bigcup_{(S/Z)\in \sJ}\left(\bigcup_{(\widetilde S/Z)\in \sE(S/Z)}\Ivol_\lc(S/Z)\right) = \min\bigcup_{(\chi,g)\in\ZZ_{\geq 0}^2} V_{\chi, g}
%     \]
% where $V_{\chi,g} = \bigcup_{(S/Z)\in \sJ_{\chi,g}}\left(\bigcup_{(\widetilde S/Z)\in \sE(S/Z)}\Ivol_\lc(S/Z)\right)$.

By Proposition~\ref{prop: Ivol chi=0 lc}, we have $$\bigcup_{g\geq 0}V_\lc(0,g)\subset \Ivol_\sm(2,1).$$
Therefore, the minimum and the minimal accumulation point of $\bigcup_{g\geq 0}V_\lc(0,g)$ are no smaller than those of $\Ivol_\sm(2,1)$, which are $\frac{1}{6}$ and $\frac{1}{2}$ respectively by Theorems~\ref{thm: sm surf Ivol} and \ref{thm: Ivol sm derived}.

If $(\chi,g)\in \Zz_{>0}^2$ or $\chi\geq 3$, then by Proposition~\ref{prop: Ivol chi>0 lc}
\begin{equation}
\min V_{\chi,g} \geq 2g-2+\chi\geq 1.
\end{equation}

If $(\chi,g)=(2,0)$, we have by Proposition~\ref{prop: Ivol chi>0 lc}
\begin{equation}
\min V_\lc(2,0) \geq \min \left(\sC_0\cup\sC_1\cup\sC_2\cup\sC_3\cup\sC_4\cup \sC_0^*\cup\sC_1^*\cup\sC_2^*\cup\sC_3^*\cup\sC_4^*\right)_{>0} = \frac{1}{11}.    
\end{equation}
Combinig the above discussions with the following claim, we infer that $\min \Ivol_\lc(2,1)=\frac{1}{671}$.
\begin{claim}\label{claim: min V_lc(1,0)}
    We have $\min V_\lc(1,0)=\frac{1}{671}$. 
\end{claim}
\begin{proof}[Proof of Claim~\ref{claim: min V_lc(1,0)}]
Let $h\colon S\rightarrow Z=\PP^1$ be a relatively minimal elliptic fibration with $\chi(\sO_S)=1$. 
%Then its relative Jacobian  $J(f)\colon J(S)\rightarrow Z$ is a rational elliptic surface. Note that $f$ and $J(f)$ share the same reduced singular fibers, and the possible configurations of singular fibers of $J(f)$ are listed as in \cite{Per90, Mir90}. 
Since $\sum_z e(h^*z)=e(S)=12$, there are at most two singular fibers of $*$-type. We proceed according to the number of $*$-type fibers of $h$.

\medskip

Suppose that $h$ has two singular fibers of $*$-type. Then $h$ has two singular fibers of type $\I_0^*$, and the other possible singular fibers are of type $m\I_0$, $m\geq 2$. By Proposition~\ref{prop: Ivol chi>0 lc}, $\Ivol_\lc(S/Z)$ consists of positive rational numbers of the form
\[
v=-1+c_1 + c_0^*
\]
where $c_1\in \sum \sC_1$ and $c_0^*=c_{01}^*+c_{02}^*$ with $c_{0i}^*\in \sC_0^*$. Since $v>0$ and $c_0^*\leq 2\left(\max\sC_0^*\right) =1$, we have $c_1>0$. Then one of the following cases occurs:
\begin{enumerate}
    \item $c_1>1$. In this case, $v\geq -1+c_1\geq -1 + \frac{1}{2}+\frac{2}{3}=  \frac{1}{6}$.
    \item $c_1=1$. In this case,  and $v= c_0^* \geq \min\sC_0^* = \frac{1}{6}$.
    \item $0<c_1<1$. In this case, $c_1=1-\frac{1}{m}$ for some  $m\in \ZZ_{\geq 2}$. Thus, we have
    \[
    v=-1+\left(1-\frac{1}{m}\right) + c_0^* = -\frac{1}{m} + c_0^* \geq 
    \begin{cases}
       -\frac{1}{2}+\frac{1}{3}+\frac{1}{3} = \frac{1}{6} & \text{if $m=2$,}\\
       -\frac{1}{3}+\frac{1}{2} =\frac{1}{6} & \text{if $m=3$,}\\
        \frac{1}{3}-\frac{1}{m} \geq\frac{1}{12} & \text{if $m\geq 4$.}
    \end{cases}
    \] 
\end{enumerate} 

\medskip

In the following, we may assume that $h$ has at most one $*$-type singular fiber. Again, by Proposition~\ref{prop: Ivol chi>0 lc}, $\Ivol_\lc(S/Z)$ consists of positive rational numbers of the form
\[
v=-1+c_0+c_1 + c_2 + c_3 + c_4 + c^*
\]
where $c_i\in \sum \sC_i$ for $0\leq i\leq 4$, and $c^*\in \sC_k^*$ for some $0\leq k\leq 4$. Since $v>0$ and $c^*\leq \frac{1}{2}$ in any case, we have $$c_0+c_1 + c_2 + c_3 + c_4>\frac{1}{2}.$$ One of the following cases occurs:
\begin{enumerate}
    \item $c_0+c_1+c_2+c_3+c_4> 1$. In this case,
\[
v\geq -1+ c_0+c_1+c_2+c_3+c_4 \geq -1+\frac{1}{3} + \frac{9}{13}  =\frac{1}{39}.
\]
\item $c_0+c_1+c_2+c_3+c_4=1$. In this case,
\[
v = c^* \geq \frac{1}{11}.
\]
\item $\frac{1}{2}<c_0+c_1+c_2+c_3+c_4<1$. In this case, since $(\sC_0)_{>0}=\{1\}$,  $\min\left(\sC_i\right)_{>0}\geq \frac{1}{2}$ for $1\leq j\leq 3$, and $c_4=0,\frac{1}{3}$, or $\geq\frac{1}{2}$, we have $c_0=0$, and one of the following cases holds:
\begin{enumerate}
 \item $c_4=0$. In this case, $c_1+c_2+c_3\in (\sC_1\cup\sC_2\cup \sC_3)_{>0}$, and one checks that
    \[
    v = -1+c_1+c_2+c_3 + c^*\geq -1 + \frac{50}{61} + \frac{2}{11} = \frac{1}{671}.
    \]
    \item $c_4=\frac{1}{3}$. In this case, exactly one of the $c_1, c_2, c_3$ is positive and $c_1+c_2+c_3=\frac{1}{2}$ or $\frac{3}{5}$. It is straightforward to check that
    \[
     v = -1+ c_1+c_2+c_3+c_4+c^*\geq  -1 +\frac{1}{3}+\frac{3}{5}+ \frac{1}{11} = \frac{4}{165}.
     \]
    \item $c_4\geq\frac{1}{2}$. In this case, $c_1=c_2=c_3=0$, and one can check that
    \[
    v = -1+c_4+c^*\geq -1+\frac{12}{19} + \frac{3}{8} = \frac{1}{152}.
    \]
\end{enumerate}
\end{enumerate}
Therefore, we have $v\geq \frac{1}{671}$ in all cases, and the equality is realized by the pair $(S, \frac{50}{61}F_1+\frac{2}{11}F_2)$, where $h\colon S\rightarrow Z$ is a rational elliptic surface, $F_1$ is a fiber of type $\II$, and $F_2$ is of type $\III^*$ (\cite[page 7]{Per90}).
\end{proof}

Now we turn to the minimal accumulation point of $\Ivol_\lc(2,1)=\cup_{(\chi,g)}V_\lc(\chi, g)$.

First, we have
\[
\min \left(\bigcup_{(\chi,g)\neq (1,0)} V_\lc(\chi,g)\right)' \geq \min \left(\bigcup_{(\chi,g)\neq (1,0)} V_\lc(\chi, g) \right) \geq\frac{1}{11}.
\]
where the second inequality follows by by the analysis for the computation of $\min \Ivol_\lc(2,1)$ above.

Thus, from now on, we can focus on the case $(\chi,g)= (1,0)$. By the proof of Theorem~\ref{thm: closed}, $V_\lc(1,0)'\subset V_\lc(1,0)$, and hence for any $v\in V_\lc(1,0)$, there is a relatively minimal elliptic fibration $h\colon S\rightarrow Z=\PP^1$ with $\chi(\sO_S)=1$ such that $v = \vol_1(K_S+\sum_z c(z)h^*z)\in \Ivol_\lc(S/Z)$, where for at least one $z_0$, $c(z_0)$ is the log canonical threshold of $h^*z_0$ with respect to $S$. Therefore,
\[
v= \vol_1(K_S+\sum_{z} c(z)h^*z) = -2 + \chi(\sO_S)+c(z_0)+\sum_{i=1}^4(c_i+c_i^*)=-1+c(z_0) +\sum_{i=1}^4(c_i+c_i^*)
\]
% where $c_i\in \sum \sC_i$ and $c_i^*\in \sum\sC_i^*$ for any $0\leq i\leq 4$. Moreover, writing $c_i=\sum_j c_{ij}$ with $c_{ij}\in \sC_i$ and $c_i^* = \sum_j c_{ij}^*$, there is at least one $c_{ij}$ or $c_{ij}^*$ equal to $\max\,\sC_i$ or $\max\,\sC_i^*$, which means that $c_ij=c(z)$ is the log canonical threshold of $h^*z$ with respect to $S$.

As before, since $e(S)=12\chi(\sO_S)=12$, $h$ has at most two singular fibers of $*$-type.

%  is of the form $$v=\vol_1(K_S+\sum_z c(z) h^*z)=-1+\sum_z c(z),$$ where for at least one $z_0$, $c(z_0)$ is the log canonical threshold of $h^*z_0$ with respect to $S$. 
% We proceed as in the proof of the minimum of $\Ivol_\lc(2,1)$. 

Suppose that $h$ has two singular fibers of $*$-type. Then both of them are of type $\I_0^*$, and all the other possible singular fibers of $h$ are of type $m\I_0$ with $m\geq 2$. Then $v\geq \frac{1}{12}$, as we have seen in the proof of Claim~\ref{claim: min V_lc(1,0)}.

In the following, we may assume that $h$ has at most one $*$-type singular fiber, and hence $\sum_i c_i^*\in \sC_k$ for some $0\leq k\leq 4$. We proceed according to the type of the fiber $h^*z_0$.

% Let $n_2, n_3, n_4$ be the numbers of singular fibers of types $\II$, $\III$, $\IV$ respectively. 

% The analysis is similar as the one for the minimum of $\Ivol_\lc(2,1)$. There are eight cases, depending on the type of $h^*z_0$, whose coefficient $c(z_0)$ is its log canonical threshold with respect to $S$. We have 
% \[
%  \vol_1(K_S+\sum_{z} c(z)h^*z) = -2 + \chi(\sO_S)+c(z_0)+\sum_{i=1}^4(c_i+c_i^*)=-1+c(z_0) +\sum_{i=1}^4(c_i+c_i^*)
% \]
% where $c_i\in\sum\sC_i$ for $1\leq i\leq 4$ and $c_i^*\in \sum\sC_i^*$.
\begin{enumerate}
    \item $h^*z_0$ is of type $m\I_k$, $k\geq 0$. In this case, $c(z_0)=1$, and hence
    \[
\vol_1\left(K_S+\sum_{z} c(z)h^*z\right) =\sum_{i=1}^4(c_i+c_i^*)\geq \min(\sC_0\cup\dots\cup \sC_4\cup\sC_0^*\cup\dots \cup\sC_4^*)_{>0}= \frac{1}{11}.
    \]
    \item $h^*z_0$ is of type $\II$. In this case, $c(z_0)=\frac{5}{6}$, and hence
     \[
\vol_1\left(K_S+\sum_{z} c(z)h^*z\right) =-\frac{1}{6}+\sum_{i=1}^4(c_i+c_i^*)\geq -\frac{1}{6} + \frac{2}{11} = \frac{1}{66}.
    \]
    \item $h^*z_0$ is of type $\III$. In this case, $c(z_0)=\frac{3}{4}$, and hence
     \[
\vol_1\left(K_S+\sum_{z} c(z)h^*z\right) =-\frac{1}{4}+\sum_{i=1}^4(c_i+c_i^*)\geq -\frac{1}{4} + \frac{3}{11} = \frac{1}{44}.
    \]
    \item $h^*z_0$ is of type $\IV$. In this case, $c(z_0)=\frac{2}{3}$, and hence
     \[
\vol_1\left(K_S+\sum_{z} c(z)h^*z\right) =-\frac{1}{3}+\sum_{i=1}^4(c_i+c_i^*)\geq -\frac{1}{3} + \frac{2}{5} = \frac{1}{15}.
    \]
    \item $h^*z_0$ is of type $\I_k^*$, $k\geq 0$. In this case, $c(z_0)=\frac{1}{2}$, and $c_i^*=0$ for $1\leq i\leq 4$. Hence
     \[
\vol_1\left(K_S+\sum_{z} c(z)h^*z\right) =-\frac{1}{2}+\sum_{i=1}^4 c_i\geq -\frac{1}{2} + \frac{4}{7} = \frac{1}{14}.
    \]
     \item $h^*z_0$ is of type $\II^*$. In this case, $c(z_0)=\frac{1}{6}$, and $c_i^*=0$ for $1\leq i\leq 4$. Hence
     \[
\vol_1\left(K_S+\sum_{z} c(z)h^*z\right) =-\frac{5}{6}+\sum_{i=1}^4 c_i\geq -\frac{5}{6} + \frac{1}{3} + \frac{4}{7} = \frac{1}{14}.
    \]
     \item $h^*z_0$ is of type $\III^*$. In this case, $c(z_0)=\frac{1}{4}$, and $c_i^*=0$ for $1\leq i\leq 4$. Hence
     \[
\vol_1\left(K_S+\sum_{z} c(z)h^*z\right) =-\frac{3}{4}+\sum_{i=1}^4 c_i\geq -\frac{3}{4} + \frac{4}{5} = \frac{1}{20}.
    \]
      \item $h^*z_0$ is of type $\IV^*$. In this case, $c(z_0)=\frac{1}{3}$, and $c_i^*=0$ for $1\leq i\leq 4$. Hence
     \[
\vol_1\left(K_S+\sum_{z} c(z)h^*z\right) =-\frac{2}{3}+\sum_{i=1}^4 c_i\geq -\frac{2}{3} + \frac{9}{13} = \frac{1}{39}.
    \]
\end{enumerate}
Summarizing all the cases, we obtain $\min \Ivol_\lc(2,1)'\geq \frac{1}{66}$, and the equality is realized by $\left(S, \frac{5}{6}h^*z_0 + \frac{2}{11}h^*z_1\right)$, where $h\colon S\rightarrow Z$ is a rational elliptic surface, $h^*z_0$ is of type $\II$, and $h^*z_1$ is of type $\III^*$ (\cite[page 7]{Per90}).
\end{proof}

\begin{rmk}
   The elliptic fibrations realizing the minimum and minimal accumulation point of $\Ivol_\lc(2,1)$ respectively, admit the same configuration of singular fibers, namely one $\III^*$, one $\II$, and one $\I_1$ (see \cite[page 7]{Per90}).
   %Their Jacobians $J(h_i)\colon J(S_i)\rightarrow Z$, $i\in \{1,2\}$, turn out to be the same uniquely determined rational elliptic surface $X_{431}$ of \cite[VIII.1.4]{Mir89}.
\end{rmk}

\subsection{The set $\Ivol_\lc^{\{0,1\}}(2,1)$}
There are two possibilities: 
\begin{enumerate}
    \item The horizontal part of the boundary $B^h= 0$. In this case, the fibration $h\colon S\rightarrow Z$ in Theorem~\ref{thm: lc surf Ivol} is relatively minimal fibration of genus 1, and $B_S=B_S^v$ is vertical with respect to $h$.
    \item The horizontal part of the boundary $B^h\neq 0$. In this case, the fibration $h\colon S\rightarrow Z$ in Theorem~\ref{thm: lc surf Ivol} is a $\PP^1$-bundle, and $B_S^h$ is a reduced divisor such that $B_S^h\cdot F=2$ for a fiber $F$ of $h$.
\end{enumerate} 
Due to this division of cases, we introduce the following notation.
\begin{nota}\label{nota: V0 V1}
We set
    \begin{itemize}
    \item $ V_0^{\{0,1\}}:=\{\vol_1(K_X+B) \mid (X, B)\in \gP_{\lc}^{\{0,1\}}(2,1),\, \, B^h\neq 0\}$,
    \item $V_1^{\{0,1\}}:=\{\vol_1(K_X+B) \mid (X, B)\in \gP_{\lc}^{\{0,1\}}(2,1),\, B^h=0\}$.
\end{itemize}
Then we have 
\[
\Ivol_\lc^{\{0,1\}}(2, 1) = V_0^{\{0,1\}}\cup V_1^{\{0,1\}}.
\]
\end{nota}
\begin{lem}\label{lem: lc Ivol Gamma1 1}
Let $V_1^{\{0,1\}}$ be as in Notation~\ref{nota: V0 V1}. Then $V_1^{\{0,1\}}=\Ivol_\lc(2, 1)$, which is the set of Iitaka volumes of projective log canonical surfaces without boundary.
\end{lem}
\begin{proof}
We observe that, both $V_1^{\{0,1\}}$ and $\Ivol_\lc(2, 1)$ consist of positive numbers of the form $v=\vol_1(K_S+B_S)$, where $h\colon S\rightarrow Z$ is a relatively minimal elliptic fibration satisfying the conditions specified in Theorem~\ref{thm: lc surf Ivol} with $\Gamma=\{0\}$ or $\Gamma=\{0,1\}$.
\end{proof}

\begin{lem}\label{lem: lc Ivol Gamma1 0}
Let $V_0^{\{0,1\}}$ be as in Notation~\ref{nota: V0 V1}. Then 
\begin{equation}\label{eq: g0 vol}
V_0^{\{0,1\}} =\left\{v\,\bigg|\, v= -2 + k + \lambda + \lambda^*, k\in\ZZ_{\geq 0},  \lambda\in \sum\sC_1, \lambda^*\in \sum_{\leq 2k}\sC_1^*\right\}
\end{equation}
where the involved sets $\sC_1$ and $\sC_1^*$ are recalled as follows:
\begin{itemize}
\item $\sC_1=\{1, 1-\frac{1}{n}\mid n\in \Zz_{>0}\}=\{1,0,\frac{1}{2},\frac{2}{3},\frac{3}{4},\frac{4}{5},\frac{5}{6},\frac{6}{7},\frac{7}{8},\frac{8}{9},\frac{9}{10},\frac{10}{11},\frac{11}{12},\frac{12}{13},\frac{13}{14},\dots\}$,
\item $\sC_1^*=\left\{\frac{1}{2}, \frac{\lfloor (n-1)/2\rfloor}{n}\,\,\bigg|\,\, n\in \Zz_{>0}\right\}=\left\{\frac{1}{2},0,\frac{1}{4},\frac{1}{3},\frac{3}{8},\frac{2}{5},\frac{5}{12}, \frac{3}{7},\frac{7}{16},\frac{4}{9},\frac{5}{11},\frac{6}{13},\frac{7}{15},\frac{8}{17},\frac{9}{19},\dots\right\}$.
\end{itemize}
\end{lem}
\begin{proof}
The set $V_0^{\{0,1\}}$ consists of positive rational numbers $v$ of the following form
\[
v= 2g(Z)-2 - e + d + \sum_{1\leq j\leq s} c_j
\]
where the notation stems from the following situation:
\begin{enumerate}
    \item $h\colon S\rightarrow Z$ is a $\PP^1$-bundle, so $K_{S}\equiv -2C_0 + (2g(Z)-2-e)F$, where $C_0$ is a section of $f$ with the minimal self-intersection number $-e$, and $F$ is a fiber of $f$;
    \item $B$ is a divisor on $S$ such that $(S, B)$ is lc and $\kappa(K_S+B)=1$;
    \item $B=B^h + B^v$, where $B^h\equiv 2C_0 + dF$, and $B^v = \sum_{1\leq i\leq s} c(z_i) h^*z_i$ for $z_i\in Z$ such that there is a curve $E_i$ over $h^*z_i$ with $1-a(E_i, S, B)\in \{0,1\}$. 
\end{enumerate}

If $B^h$ is reducible, then it is a union of two sections, and by \cite[V, Propositions~2.20 and 2.21]{Har77}, we have in this case
    \[
    d \geq 
    \begin{cases}
        e & \text{if $e\geq 0$} \\
        0 & \text{if $e<0$}.
    \end{cases}
    \]
 If $B^h$ is irreducible, then $B^h$ is a $2$-seciton, and by \cite[V, Propositions~2.20 and 2.21]{Har77}, we have in this case
    \[
    d \geq 
    \begin{cases}
        2e & \text{if $e\geq 0$} \\
        e & \text{if $e<0$}.
    \end{cases}
    \]

Since $(S, B^h)$ is log canonical and $B^h$ is reduced, $B^h$ has at most nodes as singularities. If $x\in B^h$ is a node, then it is a nonklt center of $(S, B^h)$; since $(S, B)$ is lc, $x\notin \Supp(B^v)$. It follows that $\Supp B^v $ and $\Supp B^h$ only intersect at the smooth locus of $\Supp B^h$.

Locally around a fiber $h^*z\subset \Supp(B^v)$, one of the following patterns occurs for the intersection behaviour of $h^*z$ and $B^h$:
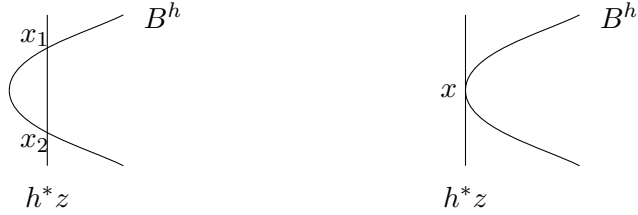
\begin{center}
    \begin{tikzpicture}
\begin{scope}
\draw (1.5,1)..controls (1,0.75) and (0,.5).. (0, 0);
\draw (1.5,-1)..controls (1,-.75) and (0,-.5).. (0, 0);
\draw (.5,-1) -- (.5,1);
\node[label=right: $B^h$] (Bh) at (1.5, 1){}; 
\node[label=below: $h^*z$] (Bv) at (.5, -1){}; 
\node[label=left: $x_1$] () at (.8,.7) {};
\node[label=left: $x_2$] () at (.8,-.7) {};
\end{scope}
\begin{scope}[xshift=6cm]
\draw (1.5,1)..controls (1,0.75) and (0,.5).. (0, 0);
\draw (1.5,-1)..controls (1,-.75) and (0,-.5).. (0, 0);
\draw (0,-1) -- (0,1);
\node[label=right: $B^h$] (Bh) at (1.5, 1){}; 
\node[label=below: $h^*z$] (Bv) at (0, -1){}; 
\node[label=right: $x$] (x) at (-.6,0) {};
\end{scope}
    \end{tikzpicture}
\end{center}
In the first picture, the fiber $h^*z$ and $B^h$ intersect transversely at two points $x_1, \,x_2$; in the second picture, the fiber $h^*z$ and $B^h$ are tangent at one point $x$ with local intersection number $(h^*z\cdot B^h)_{x}=2$. In order that $(S, B)$ is lc, and that there is an exceptional divisor $E$ over $h^*z$ such that $1-a(E, S, B)\in\{0,1\}$, the coefficient $c(z)$ of $h^*z$ in $B^v$ should come from $\sC_1=\{1, 1-\frac{1}{n}\mid n\in \Zz_{>0}\}$ in the first case, and from $\sC_1^*=\left\{\frac{1}{2}, \frac{\lfloor (n-1)/2\rfloor}{n}\,\,\bigg|\,\, n\in \Zz_{>0}\right\}$ in the second case, as specified in statement of the lemma. 

We relabel the fibers $F_i$ contained in $\Supp(B^v)$ so that $B^h$ and $F_i$ intersect transversely for $1\leq i\leq r$, and are tangent at a point for $r+1\leq i\leq s$. Note that fibers intersecting $B^h$ transversely are the generic case, and $r$ can be any non-negative number of choice. On the other hand, the fibers $h^*z$ that are tangent to $B^h$ are subject to the condition that $z$ is a branch point of $B^h\rightarrow Z$, and by the Riemann--Hurwitz formula, there are at most $2p_a(B^h)-4g(Z)+2$ of them. Denoting $g=g(Z)$, we have by the adjunction formula,
\[
2p_a(B^h) - 4g+2 = (K_X + B^h)\cdot B^h -4g+4 =(2g-2+d-e)F\cdot B^h - 4g+4 = 2(d-e).
\]

Setting $k=d-e$, we have
\[
\vol_1(K_X+B) =2g -2 + k + \lambda + \lambda^*= -2+k + \left(\lambda+\frac{1}{2}\cdot 4g\right) +\lambda^*, 
\]
where $k\in\ZZ_{\geq 0},  \lambda\in \sum\sC_1, \lambda^*\in \sum_{\leq 2k}\sC_1^*$, as described by the right hand side of \eqref{eq: g0 vol}.

On the other hand, each number in the set on the right hand side of \eqref{eq: g0 vol} can be realized by taking $S=\FF_e$, $Z=\PP^1$, and suitably chosen $B^h\in |2C_0+dF|$ with $d=k+e\geq e$ and $B^v$.
\end{proof}

\begin{prop}\label{prop: min lc Ivol 1 V0}
Let $V_0^{\{0,1\}}$ be as in Notation~\ref{nota: V0 V1}. Then the minimum and the minimal accumulation point of $V_0^{\{0,1\}}$ are as follows:
\[
\min V_0^{\{0,1\}} = \frac{1}{42}, \quad \min \left(V_0^{\{0,1\}}\right)' =\frac{1}{6}.
\]
\end{prop}
\begin{proof}
Given the explicit description $v=-2 + k + \lambda + \lambda^*$ for the elements of $V_0^{\{0,1\}}$ in Lemma~\ref{lem: lc Ivol Gamma1 0}, it is then straightforward to find the minimum:
\begin{itemize}
    \item If $k\geq 2$, then 
\[
v\geq \min(\sC_1\cup\sC_1^*)_{>0} \geq \frac{1}{4}.
\] 
\item If $k=1$, then $\lambda^*\in\sum_{\leq 2} \sC_1^*$. Therefore, $\lambda^*\leq 2\max \sC_1^*=1$, and hence $\lambda>0$. Now it is straightforward to check that
\[
v=-1+\lambda + \lambda^* \geq -1 + \frac{2}{3} + \frac{3}{8} = \frac{1}{24}.
\]
\item If $k=0$, then $\lambda^*=0$, and it is well-known that
\[
v\geq -2 + \frac{1}{2} + \frac{2}{3}+\frac{6}{7}=\frac{1}{42}.
\]
\end{itemize}

In conclusion, we have $\min V_0^{\{0,1\}} =\frac{1}{42}$.

If $v$ is an accumulation point of $V_0^{\{0,1\}}$, then one of the following cases occur:
\begin{enumerate}
    \item $\lambda=\lambda_1+\lambda_2$ with $\lambda_1 = \max\,\sC_1=1$ and $\lambda_2\in \sum\sC_1$. In this case, if $k\geq 1$, then 
\[
v =-1+k+\lambda_2+\lambda^*\geq \min(\sC_1\cup\sC_1^*)_{>0} \geq \frac{1}{4}.
\]
If $k=0$, then $\lambda^*=0$, and hence
\[
v = -2 + \lambda = -1 + \lambda_2\geq -1+\frac{1}{2} + \frac{2}{3}=\frac{1}{6}.
\]
\item $\lambda^*=\lambda_1^*+ \lambda_2^*$ with $\lambda_1^*= \max\,\sC_1^*=\frac{1}{2}$ and $\lambda_2^*\in\sum\sC_1^*$. In this case, we have $k\geq 1$. If $k\geq 2$, then 
\[
v=\lambda+\lambda^*\geq \min(\sC_1\cup \sC_1^*)_{>0}=\frac{1}{4}.
\]
If $k=1$, then $\lambda^*\leq 2\max\,\sC_1^*=1$. Since $v=-1+\lambda+\lambda^*>0$, we must have $\lambda>0$, and it follows that
\[
v=-2+k+\lambda +\lambda^*=-1+\lambda + \frac{1}{2}+\lambda_2^*\geq \frac{1}{6}.
\]
\end{enumerate}
\end{proof}
% Suppose that (1) occurs. Then $\lambda =1+\lambda_2$. In this case, 

% If $1$ is not a summand of $\lambda$ and $k=1$, then $\lambda^*\in \sum_{\leq 2}\sC_1^*$. In this case, $\lambda>0$, and one can check that
% \[
% v = -1 + \lambda + \lambda^* \geq -1 + \frac{2}{3} + \frac{3}{8} =\frac{1}{24}.
% \]
% In conclusion, the minimal accumulation point of $V_0^{\{0,1\}}$ is $\frac{1}{24}$.

\begin{thm}\label{thm: min lc Ivol 1}
The minimum and the minimal accumulation point of $\Ivol_\lc^{\{0,1\}}(2,1)$ are the same as those of $\Ivol_\lc(2,1)$, which are
\[
\min \Ivol_\lc^{\{0,1\}}(2,1) = \frac{1}{671},\quad \min \Ivol_\lc^{\{0,1\}}(2,1)' = \frac{1}{66}.
\]
\end{thm}
\begin{proof}
    By Lemma~\ref{lem: lc Ivol Gamma1 1}, we have 
    \[
    \Ivol_\lc^{\{0,1\}}(2,1) =\Ivol_\lc(2,1) \cup V_0^{\{0,1\}},
    \]
    and $V_0^{\{0,1\}}$ is described by Lemma~\ref{lem: lc Ivol Gamma1 0}. Now, by Theorem~\ref{thm: min lc Ivol 0}, the minimum and the minimal accumulation point of $\Ivol_\lc(2,1)$ are $\frac{1}{671}$ and $\frac{1}{66}$ respectively, which are smaller than the respective values $\frac{1}{42}$ and $\frac{1}{6}$ for $V_0^{\{0,1\}}$. The theorem follows.
\end{proof}

\begin{rmk}
For a fixed $v\in \QQ_{>0}$, the set of lc surface pairs $(X, B)$ with $\kappa(K_X+B)=1$, $\vol_1(K_X+B)=v$, $f\colon X\rightarrow Z$ a $\PP^1$-bundle is usually not bounded (compare Remark~\ref{rmk: not bdd}). For example, we may take a Hirzebruch surface $f\colon X=\FF_e\rightarrow \PP^1$ with a section $C_0$ such that $C_0^2=-e\leq 0$, and take $B=C_0+C_1$, where $C_1\in |C_0+(e+3)F|$ is general and $F$ is a fiber of $f$. Then $(X, B)\in \gP_\lc^{\{0, 1\}}(2,1)$ with Iitaka volume $\vol_1(K_X+B) = 1$. However, as $e$ varies, $X$ and hence also $(X, B)$ do not form a bounded family. This phenomenon deviates from that in the klt case, which was treated in \cite[Theorem~6.1]{Fil23}.
\end{rmk}

\section{Remarks on Iitaka volumes for $-K_X$}\label{sec: anticanonical}

\begin{lem}\label{lem:gfiso}
Let $(X,B)$ be an lc pair, and $f\colon X\rightarrow Z$ a projective morphism with $\dim Z>0$ such that $K_X+B\sim_\Rr D$ for some vertical$/Z$ $\Rr$-Cartier $\Rr$-divisor $D$ on $X$. Suppose that $\psi:X\dashrightarrow X'$ is a part of $(K_X+B)$-MMP over $Z.$ Then there exists a nonempty open subset $U\subset Z$ such that $X_z$ is isomorphic to $X'_z$ for any $z\in U$, where $X_z$ (resp., $X'_z$) is the fiber of $X\to Z$ (resp., $X'\to Z$) over $z$.
\end{lem}
\begin{proof}
It is enough to show the lemma assuming that $\psi$ is either a flip over $Z$ or a divisorial contraction over $Z$. Let $A_Z$ be an ample divisor on $Z$ such that $\Supp A_Z\supset f(\Supp D)$ and $A+D\ge0$, where $A:=f^*A_Z$. If $\psi$ is a flip, then we let $C$ be a flipping curve, and if $\psi$ is a divisorial contraction, then let $C$ be a contracted curve. In both cases, we have that
$$0>(K_X+B)\cdot C=(K_X+B+A)\cdot C=(A+D)\cdot C$$ 
which implies that the flipping locus (resp., $\Exc(\psi)$) is contained in $\Supp D$. Therefore the image of flipping locus (resp., $\Exc(\psi)$) on $Z$ has codimension $\ge1$. Then we can find a proper subset of $Z$ satisfying the property. 
\end{proof}

\begin{lem}\label{lem:cbfindex2}
Let $d$ be a positive integer, $u$ a positive rational numbers, and $\Gamma\subset[0,1]$ a DCC subset. Then there exists a positive integer $p$ depending only on $d,u,$ and $\Gamma$ satisfying the following. Assume that
\begin{enumerate}
    \item  $(X,B)$ is a projective klt pair of dimension $d$ with $B\in\Gamma$,
    %\item $\dim Z>0$,
    \item $f\colon X\to Z$ is a contraction with $\dim Z>0$,
    \item $K_X+B\sim_{\Rr,Z}0$ and $K_X+B\sim_\Qq0$ over the generic point $\eta_Z$ of $Z$, and
    \item there is a $\Qq$-Cartier $\ZZ$-divisor $N\ge0$ on $X$ such that $\vol(N|_F)=u$ for the general fibers $F$ of $f$.
\end{enumerate}
Then we can choose a moduli part $\bM$ of the canonical bundle formula for $(X,B)$ over $Z$ such that $p\bM$ is b-Cartier, and 
$$p(K_X+B)\sim pf^*(K_Z+B_Z+\bM_Z),$$
where $B_Z$ is the discriminant part of the canonical bundle formula for $(X,B)$ over $Z$.
\end{lem}

\begin{proof}
Let $(Y,B_Y)$ be a small $\Qq$-factorialization of $(X,B)$, and $N_Y$ the strict transform of $N$ on $Y$. Note that $N_{Y}$ is exactly the pullback of $N$. In particular, $\vol(N_Y|_{F_Y})=u$ for the general fibers $F_Y$ of $Y\to Z$. Since $K_X+B\sim_\Qq0$ over $\eta_Z$, $K_Y+B_Y\sim_\Qq0$ over $\eta_Z$. This implies that $B^h_Y\in\Qq$, where $B^h_Y$ denotes the horizontal$/Z$ part of $B_Y$.

We may run an MMP on $K_Y+B_Y^h$ over $Z$ which terminates with a good minimal model $Y'$ over $Z$ by \cite{Bir12,HX13}, i.e., $K_{Y'}+B_{Y'}^h$ is semi-ample over $Z$, where $B_{Y'}^h$ is the strict transform of $B_Y$ on $Y'$. Let $g'\colon Y'\to Z'$ be the ample model of $K_{Y'}+B_{Y'}$ over $Z$, then it is clear that $Z'\to Z$ is birational. Note that by Lemma \ref{lem:gfiso}, this MMP does not modify the general fibers $F_Y$ of $Y\to Z$. Denote by $N_{Y'}$ the strict transform of $N_Y$ on $Y'$, then $\vol(N_{Y'}|_{F_{Y'}})=u$ for the general fibers $F_{Y'}$ of $g'$. According to \cite{BCHM10}, there exists a ample model $X'$ of $N_{Y'}$ over $Z'$, such that the strict transform $N'$ of $N_{Y'}$ on $X'$ is ample over $Z'$. Let $(B^h)'$ be the strict transform of $B^h_Y$ on $X'$. From our construction, we see that $(X',(B^h)')$ is a klt pair, $K_{X'}+(B^h)'\sim_{\Qq,Z'}0$, $(B^h)'\in\Gamma\cap\Qq$, $N'$ is ample over $Z'$, and $\vol(N'|_{F'})=u$ for the general fibers $F'$ of the morphism $f'\colon X'\to Z'$. By Lemma \cite[Lemma 7.4]{Bir21}, there is a positive integer $p$ which only depends on $d,u$, and $\Gamma$ such that we can choose a moduli part $\bM$ of the canonical bundle formula for $(X',(B^h)')$ over $Z'$ such that $p\bM$ is b-Cartier, and
$$p\left(K_{X'}+(B^h)'\right)\sim pf'^*\left(K_{Z'}+(B^h)'_{Z'}+\bM_{Z'}\right),$$
where $(B^h)'_{Z'}$ is the discriminant part of the canonical bundle formula for $(X',(B^h)')$ over $Z'$.

Let $B'$ be the strict transform of $B_{Y}$ on $X'$, then $(X',B')$ is crepant to $(X,B)$ as $K_X+B\sim_{\Rr,Z}0$; see Lemma~\ref{lem: R-trivial}. In particular, $K_{X'}+B'\sim_{\Rr,Z}0$ and $K_{X'}+B'\sim_{\Rr,Z'}0$. Denote by $B'_{Z'}$ (resp., $B_Z$) the discriminant part of the canonical bundle formula for $(X',B')$ over $Z'$ (resp., for $(X,B)$ over $Z$). Since $B'-(B^h)'$ is vertical over $Z'$ and $B'-(B^h)'\sim_\Rr0$, $B'-(B^h)'=f'^*L'$ for some $\Rr$-Cartier $\Rr$-divisor $L'$ (cf. \cite[Lemma 2.11]{Li20}). Then one can see that $B'_{Z'}=(B^h)_{Z'}+L'$ by \cite[Lemma 7.4]{PS09}, and
$$p(K_{X'}+B')=p\left(K_{X'}+(B^h)'\right)+p\left(B'-(B^h)'\right)\sim pf'^*\left(K_{Z'}+B'_{Z'}+\bM_{Z'}\right).$$ 
Therefore we can also take $\bM$ to be the moduli part of the canonical bundle formula for $(X',B')$ over $Z'$ which implies that $\bM$ is also the moduli part of the canonical bundle formula for $(X,B)$ over $Z$, and $p(K_{X}+B)\sim pf^*(K_Z+B_Z+\bM_Z).$ This finishes the proof.
\end{proof}

\begin{comment}
\begin{thm}\label{thm: anti Iitaka ACC}
Let $d$ be a positive integer, $\epsilon,u$ two positive rational numbers, and $\Gamma\subset[0,1]$ a DCC set. Then there exists an ACC set $\mathcal{A}$ depending only on $d,\epsilon,u,$ and $\Gamma$ satisfying the following.

Assume that $(X,B)$ is a projective pair such that
\begin{enumerate}
    \item $(X,B)$ is $\epsilon$-lc of dimension $d$ with $B\in\Gamma$,
    \item $-(K_X+B)$ is semi-ample defining a contraction $f\colon X\to Z$, and
    \item there is a $\Qq$-Cartier $\ZZ$-divisor $N$ on $X$ with $\vol(N|_F)=u$ for the general fibers $F$ of $f$.
\end{enumerate}
Then the Iitaka volume of $-(K_X+B)$ belongs to $\mathcal{A}$.
\end{thm}
\end{comment}

\begin{proof}[Proof of Theorem \ref{thm: anti Iitaka ACC}]
Since $N$ is big over $Z$, according to \cite{BCHM10}, there exists a ample model $X'$ of $N$ over $Z$, such that $N'$ is ample over $Z$, where $N'$ is the strict transform of $N$ on $X'$. Note that $-(K_X+B)$ has the same Iitaka volume as $-(K_{X'}+B')$, and $\vol(N|_F)=\vol(N'|_{F'})$, where $B'$ is the strict transform of $B$ on $X'$, and $F'$ is the general fibers of $X'\to Z$. By \cite[Theorem 1.1]{Bir20}, there exists a positive integer $m$ which only depends on $\dim F$ and $\epsilon$, such that $|mN'|_{F}|$ define a birational map. In particular, by upper-semicontinuity of cohomology and \cite[Lemma 3.2.1]{BCHM10}, there exists $0\le N_0'\sim mN'$. Replacing $(X,B),N,u$ by $(X',B'),N_0',m^{\dim F'}u$ respectively, we may assume that $N\ge 0$, and $N$ is ample over $Z$.

Let $B_Z$ be the discriminant part of the canonical bundle formula for $(X,B)$ over $Z$. According to \cite{HMX14}, $B_Z\in\tGamma$ for some DCC set $\tGamma\subset[0,1]$ which only depends on $d$ and $\Gamma$. Since 
$\kappa(-(K_X+B))\ge0$, we can see that $K_X+B\sim_\Qq0$ over the generic point of $Z$. Then by Lemma \ref{lem:cbfindex2}, one can find a positive integer $p$ which only depends on $d,u,$ and $\Gamma$ such that we can choose a moduli part $\bM$ of the canonical bundle formula for $(X,B)$ over $Z$ such that $p\bM$ is b-Cartier, and
$$p(K_X+B)\sim pf^*(K_Z+B_Z+\bM_Z).$$ 
Furthermore, by \cite[Theorem 1.8]{Bir23}, $(Z,B_Z+\bM)$ is a $\delta$-lc generalized pair for some $\delta>0$ which only depends on $d,\epsilon$, and $u$. By \cite[II Lemma 2.11]{Nak04}, the Iitaka volume of $-(K_X+B)$ equals to the Iitaka volume of $-(K_Z+B_Z+\bM_Z)$, we can conclude our result by \cite[Theorem 1.4]{HL23b}.
\end{proof}

The following example shows that the condition (2) of Theorem~\ref{thm: anti Iitaka ACC} is necessary.

\begin{ex}
Take $Y=\mathbb{P}(p,q,r)$, where $p,\,q,\,r\in\ZZ_{>0}$ are prime to one another, and let $Z$ be the minimal resolution of $Y$. Take $X=Z\times C$, where $C$ is an elliptic curve, and $N$ a section of $X$ over $Z$. Then $X$ is smooth, $\vol(N|_{C\times\{z\}})=1$ for any $z\in Z$, but $\vol_2(-K_X)=\vol(-K_Z)=\vol(-K_Y)=\frac{(p+q+r)^2}{pqr}$ forms a dense subset of $\mathbb{R}_{>0}$ as $p,q,r$ vary.
\end{ex}

We conclude this section with a precise description of $\vol_1(-K_X)$ for smooth rational elliptic surfaces. 
%This implies that the condition (3) of Theorem~\ref{thm: anti Iitaka ACC} is necessary.
\begin{prop}
Let $f\colon X\rightarrow Z$ be a smooth rational elliptic surface. Then $\vol_1(-K_X)\in\{\frac{1}{m}\mid m\in \Zz_{>0}\}$. Conversely, every $\frac{1}{m}$ can be realized as $\vol_1(-K_X)$ for some rational elliptic surface $X$.
\end{prop}
\begin{proof}
We have $Z\cong \mathbb{P}^1$, and there is at most one multiple fiber. If $f$ does not have any multiple fiber, then $K_X=f^*(K_Z + L)$, where $\deg L =\chi(\mathcal{O}_X) =1$. Therefore, 
\[
\vol_1(-K_X) = -\deg(K_Z+L)=1.
\]
If $f$ has one multiple fiber with multiplicity $m$, then $K_X=f^*(K_Z + L + \frac{m-1}{m}Q)$, where $\deg L =\chi(\mathcal{O}_X) =1$, and hence 
\[
\vol_1(-K_X) = -\deg\left(K_Z+L+\frac{m-1}{m}Q\right)=\frac{1}{m}.
\]

For each $m\geq 2$, there exists a rational elliptic surface $X$ with one multiple fiber of multiplicity $m$ by \cite{Fuj90}. Therefore, each $\frac{1}{m}$ can be realized as $\vol_1(-K_X)$.
\end{proof}

\end{document}